\documentclass[12pt,a4paper]{article}
\pdfoutput=1
\usepackage{a4,a4wide}
\usepackage{amsfonts,amsmath,amssymb,amsthm,enumitem}

\usepackage{graphicx}
\usepackage{color}

\usepackage[latin1]{inputenc}
\numberwithin{equation}{section}

\usepackage{doi}

\newtheorem{theorem}{Theorem}[section]
\newtheorem{lemma}[theorem]{Lemma}
\newtheorem{proposition}[theorem]{Proposition}
\newtheorem{corollary}[theorem]{Corollary}

\newtheorem{remark}[theorem]{Remark}
\newtheorem*{theorem*}{Theorem}
\newtheorem*{lemma*}{Lemma}
\newtheorem*{proposition*}{Proposition}
\newtheorem*{corollary*}{Corollary}

\newtheorem*{theo1.2}{Theorem 1.2}
\newtheorem*{theo1.4}{Theorem 1.4}
\newtheorem*{cor1.1}{Corollary 1.1}
\newtheorem*{prop1.6}{Proposition 1.6}
\newtheorem*{lem1.7}{Lemma 1.7}
\newtheorem*{lem1.8}{Lemma 1.8}
\newtheorem*{prop1.9}{Proposition 1.9}
\newtheorem*{prop1.10}{Proposition 1.10}
\newtheorem*{lem1.11}{Lemma 1.11}
\newtheorem*{lem1.12}{Lemma 1.12}\newtheorem*{lem1.13}{Lemma 1.13}
\newtheorem*{lem1.14}{Lemma 1.14}

\renewcommand\tilde{\widetilde}

\newcommand\vp{\varphi}

\def\R{\mathbb{R}}

\def\N{\mathbb{N}}

\def\H{\mathcal{H}}

\def\P{\mathbb{P}}

\def\LM#1{\hbox{\vrule width.2pt \vbox to#1pt{\vfill \hrule width#1pt
height.2pt}}}
\def\LL{{\mathchoice {\>\LM7\>}{\>\LM7\>}{\,\LM5\,}{\,\LM{3.35}\,}}}
\def\restr{{\LL}}

\def\Div{\nabla \cdot}

\def\ds{\displaystyle}

\def\1{\mathbf{1}}

\def\XXint#1#2#3{{\setbox0=\hbox{$#1{#2#3}{\int}$ }
\vcenter{\hbox{$#2#3$ }}\kern-.57\wd0}}

\def\eps{\varepsilon}

\def\lt{\left}
\def\rt{\right}
\def\les{\lesssim}
\def\ges{\gtrsim}

\def\bdr{\textrm{bdr}}

\def\lead{\textrm{lead}}

\def\spt{\textup{Spt}\,}

\newcommand{\del}{\delta\hspace{-0.2ex}}

\newcommand\numberthis{\addtocounter{equation}{1}\tag{\theequation}}

\def\rhoin{\rho^{\textrm{in}}}
\def\jin{j^{\textrm{in}}}
\def\rhofin{\rho^{\textrm{fin}}}
\def\jfin{j^{\textrm{fin}}}
\def\kept{\textrm{kept}}
\def\tang{\textrm{tan}}
\def\barR{\overline{R}}
\def\underR{\underline{R}}

\begin{document}
\title{Quantitative linearization results for the Monge-Amp\`ere equation}

\author{Michael Goldman\thanks{Universit\'e de Paris, CNRS, Sorbonne-Universit\'e,  Laboratoire Jacques-Louis Lions (LJLL), F-75005 Paris, France, \texttt{michael.goldman@u-paris.fr}} \and Martin Huesmann\thanks{Martin Huesmann,  Universit\"at M\"unster, Einsteinstra\ss e 62, 48149 M\"unster, Germany, \texttt{martin.huesmann@uni-muenster.de}} \and Felix Otto\thanks{Max Planck Institute for Mathematics in the Sciences, 04103 Leipzig, Germany, \texttt{Felix.Otto@mis.mpg.de}}}

\date{\today}

\maketitle

\abstract{This paper is about quantitative linearization results for the Monge-Amp\`ere equation with rough data.
We develop a large-scale regularity theory and prove that if a measure $\mu$ is close to the Lebesgue measure in Wasserstein distance at all scales,
then the displacement of the macroscopic optimal coupling is quantitatively close at all scales to the gradient of the solution of the corresponding Poisson equation.
The main ingredient we use is a harmonic approximation result for the optimal transport plan between arbitrary measures. This is used in a Campanato iteration which transfers the information through the scales.    
}

\tableofcontents
\setlength{\parindent}{0in}
\section{Introduction}

Motivated by the optimal matching problem (see for instance \cite{Ta14,BaBo13,AmStTr16,Le18}), we develop  a large-scale 
regularity theory for optimal couplings between an arbitrary measure and the Lebesgue measure.
We start by giving a simplified version of our main result which already gives a good idea of what we achieved.
For this we fix a radial  cutoff function 
\begin{equation}\label{cutoff}
 \eta\in C^\infty_c(B_1) \ \textrm{ with } \ \int_{\R^d} \eta =1, \qquad \textrm{ and set } \ \eta_R:=\frac{1}{R^d}\eta\lt(\frac{\cdot}{R}\rt).
\end{equation}
We will work under the assumption that $\sup |\nabla^2 \eta|\le C$  for some fixed universal constant $C>0$ depending only on the dimension. 
We introduce the following notation: for $R>0$ and $\mu$ a positive  measure on $\R^d$, $W_{B_R}^2(\mu,\kappa):=W^2(\mu\restr B_R,\kappa dx\restr B_R)$ where $W$ is the $2-$Wasserstein distance (see \eqref{eq:W2} for a definition) and $\kappa=\frac{\mu(B_R)}{|B_R|}$ 
is the generic constant for which this makes sense. 
\begin{corollary}\label{cormain}
 For every  $0\le \alpha<1$, there is a constant $C>0$ depending only on the dimension $d$ and  $\alpha$  such that the following holds: consider $\barR\ge C$ and   a positive measure $\mu$
 on $B_{\barR}$ with $\mu(B_{\barR})=|B_{\barR}|$ and such that 
 \begin{equation}\label{hypdata}
  \frac{1}{|B_R|} W^2_{B_R}\lt(\mu, \kappa\rt)\le R^\alpha \qquad \forall R\in[1, \barR].
 \end{equation}
Let $u$ be a (distributional) solution of the Poisson equation (here $\nu$ denotes the exterior normal to $\partial B_{\barR}$)
\[
 \Delta u=\mu-1 \ \textrm{ in } B_{\barR} \qquad \textrm{ and } \qquad \nu\cdot \nabla u=0 \ \textrm{ on } \partial B_{\barR}.
\]
 Then, if $\pi$ is the optimal coupling for $W_{B_{\barR}}(\mu,1)$, for $\barR\ge R\ge C$,
\begin{equation}\label{mainweakcor}
\lt|\int_{\R^d\times \R^d} \eta_R(x)(y-x-\int_{\R^d} \eta_R \nabla u)d\pi\rt|\leq \frac{C}{R^{1-\alpha}}.
\end{equation}

\end{corollary}
This result gives a precise and quantitative description of the well-known fact that the Monge-Amp\`ere equation
linearizes to the Poisson equation around the Lebesgue measure (see for instance \cite{AmStTr16,BreLoe}). Indeed, \eqref{mainweakcor}
shows that provided the measure $\mu$ is close to the Lebesgue measure at every scale in the sense of \eqref{hypdata}, then the displacement $y-x$ is close to $\nabla u$ on average, that is,  in a weak topology.
\medskip 

Let us  comment on the scaling in \eqref{mainweakcor}. 
As discussed above, (\ref{mainweakcor}) amounts to the estimate of a linearization
error. Since the left-hand side obviously is linear in the displacement, one
would generically expect the right-hand side to be quadratic in the displacement.
Obviously, $\frac{1}{|B_R|}W_{B_R}^2(\mu,\kappa)$ is also a quadratic expression 
in the displacement. However, in terms of dimensions, the left-hand side
of (\ref{mainweakcor}) has units of length, so also the right-hand side has 
to have units of length. This leaves only the expression
$\frac{1}{R|B_R|}W_{B_R}^2(\mu,\kappa)$, which is consistent with (\ref{hypdata}). Setting $\beta(R):= R^\alpha$, another way to read \eqref{mainweakcor} is to divide it by the displacement scale $\sqrt{\beta(R)}$, so that the left-hand 
side turns into the relative linearization error and the right-hand side assumes the non-dimensional form $\frac{\sqrt{\beta(R)}}{R}\ll1$. However, in proving this, it is crucial that the dimensional
quantity $\frac{\beta(R)}{R}$ is decreasing with increasing $R$, to the point that the sum over $\frac{\beta(R_k)}{R_k}$ for geometrically decaying $(R_k)_{1\leq k\leq K}$  is dominated by the smallest scale, see \eqref{hypbeta} below. This explains why we require $\beta(R)\ll R$ (i.e.\ $\alpha<1$) rather than the
dimensionally more natural $\beta(R)\ll R^2$.
\medskip

Let us now introduce some notation and state our first main result. 
Let $\pi$ be a measure on $\R^d\times \R^d$ with cyclically monotone support. For a given $\barR>0$, let $\mu:=\pi_1\restr B_{\barR}$ and
assume that $\pi_2\restr B_{\barR}=dx$, where we denoted by $\pi_i$ the  marginals of $\pi$. 
For $t\in[0,1]$, $\zeta\in C^0_c(\R^d)$, and $\xi\in C^0_c(\R^d)^d$, define the flux-density pair
\begin{multline}\label{defrhojintro}
 \int_{\R^d} \zeta d\rho_t=\int_{\R^d\times \R^d} \zeta((1-t)x+ty)d\pi \\
 \textrm{and} \quad \int_{\R^d} \xi \cdot dj_t=\int_{\R^d\times \R^d} \xi((1-t)x+ty) \cdot(y-x)d\pi.
\end{multline}
We then let  
\begin{equation}\label{defbarjintro}
 \bar j:=\int_0^1 dj(\cdot,t)
\end{equation}
 be the time integral (see \eqref{defbar}) of the flux $j$. Notice that $(\rho,j)$ solves in a distributional sense the continuity equation
\begin{equation}\label{conteqintro}
 \partial_t \rho+\nabla\cdot j=0 \qquad \textrm{ and } \qquad  \rho_0=\mu, \ \rho_1= 1.
\end{equation}
 We will fix for the whole paper a nonnegative rate function $\beta(R)$ satisfying the following hypothesis: 
it is increasing, $R\mapsto\frac{\beta(R)}{R}$ is decreasing, and  
 there exists a constant $C_\beta>0$ such that  
\begin{equation}\label{hypbeta}
 \sum_{\ell\geq0} \frac{\beta(2^\ell R)}{2^\ell R}\le C_\beta \frac{\beta(R)}{R} \quad \textrm{for every } R\ge 1 \qquad \textrm{and} \qquad \beta(1)\le C_\beta.
\end{equation} 
 We do not claim optimality of these conditions. They are rather made to  include both the choices
$\beta(R)=R^\alpha$ for $0\le \alpha<1$ and $\beta(R)=\log (R+C)$
(for\footnote{
The notation $A\ll 1$, which we only use in assumptions, means that there exists an $\eps>0$ only depending on the dimension and on $C_\beta$, such that
if $A\leq \eps$ then the conclusion holds.
 Similarly, the notation $A\les B$, which we use in output statements, means that there exists a global
constant $C>0$ depending on the dimension and on $C_\beta$ such that $A\le C B$.} $R\gg1$).

\begin{theorem}\label{theo:main}
 Let $\mu$ be a positive measure on $\R^d$ such that for some  $\barR \gg 1$, 
 \begin{equation}\label{hypstart}
  \frac{1}{|B_{\barR}|} \int_{(B_{\barR}\times \R^d)\cup (\R^d\times B_{\barR})} |x-y|^2 d\pi \le \beta(\barR)
 \end{equation}
and 
\begin{equation}\label{hypdata2}
 \frac{1}{|B_R|} W^2_{B_R}\lt(\mu, \kappa\rt)\le \beta(R) \qquad \forall R\in[ 1, \barR].
\end{equation}
 Then there exists $\barR'\sim \bar R$ such that letting  $u$ be a (distributional) solution of the Poisson equation (recall the definition \eqref{defbarjintro} of $\bar j$)
\begin{equation}\label{ma87}
 \Delta u=\mu-1 \ \textrm{ in } B_{\barR'} \qquad \textrm{ and } \qquad \nu\cdot \nabla u=\nu\cdot \bar j \ \textrm{ on } \partial B_{\barR'}
\end{equation}
and for $R>0$  (recall the definition \eqref{cutoff} of $\eta_R$),
\[
 h_R:=\int_{\R^d} \eta_R \nabla u,
\]
there holds for $\barR\ges R \ges 1$, 
\begin{multline}\label{mainweaktheo}
 \lt|\int_{\R^d\times \R^d} \eta_R(x)(y-x-h_R)d\pi\rt|+\lt|\int_{\R^d\times \R^d} \eta_R(y-h_R)(y-x-h_R)d\pi\rt| \\\les \frac{\beta(R)}{R} ,
\end{multline}
and
\begin{multline}\label{strongestimintro}
 \sup \lt\{ |y-x-h_R| \ : \ (x,y)\in \spt\pi\cap ((B_{R}\times\R^d)\cup(\R^d\times B_{R}(h_R)))\rt\}\\
 \les   R \lt(\frac{\beta(R)}{R^2}\rt)^{\frac{1}{d+2}}.
 \end{multline}
\end{theorem}

For our future work on the optimal matching problem, the following observations will be useful: 
\begin{remark}
\begin{itemize}
\item[(i)]It will be clear from the proof of Theorem \ref{theo:main} that
 we do not need the hypothesis \eqref{hypdata2} to hold for every radius $R\in [1,\barR]$
 but that it is enough to have it for a sequence of approximately geometrically decaying radii i.e. for a sequence $R_0=\barR\ge \cdots\ge R_K\gg 1$ with $R_{k-1}\ge 2 R_k\ges R_{k-1}$;
 \item[(ii)]  in the statement of Theorem \ref{theo:main}, we have set the 'small' scale to be equal to one. By scaling, we see that for every $\underline{R}>0$, letting $\underline{\beta}(R):=\underline{R}^2\beta(R/\underline{R})$, if  hypothesis $\eqref{hypstart}\&\eqref{hypdata2}$ hold with $\underline{\beta}$ instead of $\beta$ and  for $R\in [\underR,\barR]$ for some $\barR\gg \underR$, then the conclusions $\eqref{mainweaktheo}\&\eqref{strongestimintro}$ also hold for  $\barR\ges R\ges \underR$ (still with $\underline{\beta}$ instead of $\beta$).
\end{itemize}
 \end{remark}

 As  Corollary \ref{cormain} does, 
Theorem \ref{theo:main} states that if the averaged square displacement is small in the sense of \eqref{hypstart} at the macroscopic scale $\barR$ and if the measure $\mu$
is close to the Lebesgue measure in the sense of \eqref{hypdata2} at every scale from the macroscopic one $\barR$ to the microscopic one which is set to unity, then
the displacement is quantitatively close to $\nabla u$ both in a weak sense (see \eqref{mainweaktheo}) and in a strong sense (see \eqref{strongestimintro}).
\medskip 

As mentioned above, our motivation for proving Theorem \ref{theo:main} comes
from the optimal matching problem.
When $\mu$ is given by a Poisson point process, it is known since \cite{AKT84} (see also \cite{AmStTr16} and \cite{GHO}) that for $d=2$, \eqref{hypdata2} holds with $\beta(R)=\log (R+C)$
with very high probability while for $d\ge 3$, \eqref{hypdata2}  holds for $\beta(R)=C$ also with very high probability.
As will be shown in a future work, \eqref{mainweaktheo} gives a rigorous validation, at the level of the displacement and down to the microscopic scale,
of the ansatz used by Caracciolo et al.\ in \cite{CaLuPaSi14}. From this point of view, \eqref{mainweaktheo} should be seen as a quenched result 
separating the deterministic from the stochastic part of the analysis. We refer to \cite{AmStTr16} and the recent paper \cite{AGS19} for a justification of this ansatz at the macroscopic scale on the level of the transport cost respectively the transport map (and to \cite{Le17, Le18} for application of these ideas in the case of Gaussian matching). 

In a future work we will also prove that for $d=2$, $|h_R|^2\sim \log \barR$ for $R\sim 1$ (notice that $\log \barR\gg 1$) with very high probability and thus \eqref{strongestimintro}
gives the expected deviation  for the transport  from the identity plus a shift. In higher dimension, since $|h_R|\sim 1$ for $R\sim 1$ this estimate is less meaningful (see \cite{HuSt13}). 
In another future work, we would  also like to address the thermodynamic limit in the matching problem;
this will require controlling how the optimal coupling depends on changing the positions of (distant) points.
This relates to the question of how sensitively, in terms of a H\"older modulus of continuity,
the coupling depends on the measure (see \cite{AGS19,Be18,MeDeCh19}). 
In our setting, we hope that  our large-scale regularity theory might lead to  more precise statements.
\medskip

Before explaining the strategy/main ideas for the proof of Theorem \ref{theo:main} we would like to single out one ingredient which can be seen as our second main result. It is a quantitative harmonic approximation result for the displacement $y-x$ under the optimal coupling between two {\it arbitrary measures}.

Let $\mu$ and $\lambda$ be two positive measures on $\R^d$ and let $\pi$ be a  coupling between $\mu$ and $\lambda$ with monotone support. For a given scale $R$, we introduce the local energy
\begin{equation}\label{Eintro}
 E(R):=\frac{1}{|B_{6R}|}\int_{(B_{6R}\times \R^d)\cup( \R^d\times B_{6R})} |x-y|^2 d\pi
\end{equation}
and the local measure of the squared distance between the data and the Lebesgue measure 
\begin{equation}\label{hypdata3}
 D(R):=\frac{1}{|B_{6R}|} W^2_{B_{6R}}(\mu, \kappa_\mu)+ \frac{R^2}{\kappa_\mu}(\kappa_\mu-1)^2+\frac{1}{|B_{6R}|} W^2_{B_{6R}}(\lambda, \kappa_\lambda)+ \frac{R^2}{\kappa_\lambda}(\kappa_\lambda-1)^2.
\end{equation}
The harmonic approximation theorem states that provided these two quantities are small (at the natural scale, i.e.\ in a dimensionless sense) for some radius $\barR>0$, then there exists a harmonic gradient field $\nabla \Phi$ such that the displacement $y-x$ is very close to $\nabla \Phi$ in $B_{\barR}$.
\begin{theorem}\label{theo:harmonicLagintro}
 For every $0<\tau\ll 1$, there exist $\eps(\tau)>0$ and $C(\tau)>0$ such that provided $\frac{1}{\barR^2}E(\barR)+\frac{1}{\barR^2}D(\barR)\le \eps$ the following holds. There exists a radius $R\in (3\barR,4\barR)$ such that if $\Phi$ is the up to an additive constant unique (distributional) solution of
 \begin{equation}\label{ma89}
  \Delta \Phi=c \ \textrm{ in } B_{R} \qquad \textrm{ and } \qquad \nu \cdot \nabla \Phi=\nu\cdot \bar j \ \textrm{ on } \partial B_{R},
 \end{equation}
where $\bar j$ is  defined in \eqref{defbarjintro} and $c$ is the generic constant for which this equation is solvable, then
\begin{equation}\label{eq:mainestimateintro}
 \frac{1}{|B_{\barR}|}\int_{(B_{\barR}\times \R^d)\cup(\R^d\times B_{\barR})} |x-y+\nabla \Phi(x)|^2 d\pi\le \tau E(\barR)+ C D(\barR)
\end{equation}
and 
\begin{equation}\label{introestimPhi}
  \barR^2 \sup_{B_{2\barR}}|\nabla^2 \Phi|^2+ \sup_{B_{2\barR}}|\nabla \Phi|^2\les E(\barR)+D(\barR).
\end{equation}
\end{theorem}
\begin{remark}\label{rem:DataLebesgueintro}
 If one of the two measures, say the target measure $\lambda$, is the Lebesgue measure, then the same conclusion holds replacing the definition \eqref{hypdata3} of $D(R)$ by $D'(R):=\frac{1}{|B_{6R}|}W^2_{B_{6R}}(\mu,\kappa)$, see Remark \ref{rem:DataLebesgue} for more details. 
 \end{remark}
Most of the paper is dedicated to the proof  of this far-reaching generalization of \cite[Prop. 4.7]{GO}. Let us stress that since  $E(\barR)$ behaves like a squared $H^1$ norm in terms
of the Kantorovich potentials and since the squared Wasserstein distance $D(\barR)$ behaves like a squared $H^{-1}$ norm (cf.\ \cite[Sec.\ 7.6]{Viltop}), 
 all quantities occur in the estimate \eqref{eq:mainestimateintro} as if we were dealing with a second order linear elliptic equation and looking at squared $L^2$-based quantities.
 Anticipating a bit, we point out that having as boundary flux $\nu\cdot \bar j$ on $\partial B_R$ for $\nabla \Phi$ will play a crucial role in the derivation of Theorem \ref{theo:main} from Theorem \ref{theo:harmonicLagintro}. 
For other applications, such as the small-scale regularity theory developped in \cite{GO},   it only matters that $\nabla \Phi$ is harmonic and that \eqref{eq:mainestimateintro} and \eqref{introestimPhi} hold 
(in fact we need also  an estimate on $\sup_{B_{\barR}} |\nabla^3 \Phi|$ but this follows from harmonicity).       
\medskip 

We would also like to comment on the fact that the right-hand side of \eqref{ma89} is a constant and not $\mu-\lambda$ as could be expected from \eqref{ma87} (see also \cite{GO}).
This is due to the fact that we need (interior) $L^2$ bounds on $\nabla \Phi$ which are generally not true for solutions of the Poisson equation with rough right-hand side. Similar issues were tackled in \cite{AmStTr16, Le17}  by mollification  with smooth kernels (the heat and the Mehler kernel respectively).
\medskip 

 Finally, we should say that while \cite[Prop. 4.7]{GO} served the purpose of an alternative proof for the partial regularity result of \cite{FigKim} (see also \cite{DePFig}), 
 it is not clear to us that the approach therein based on Caffarelli's regularity theory (see \cite{CafJAMS92,CafCPAM91}) could also be used in the present context.

\subsection{Strategy of the proofs}\label{sec:strategy}
 
In this section, we will explain the key ideas for the proofs of our two main theorems. The rough picture is as follows: 
The first step in the proof of Theorem \ref{theo:main} is the harmonic approximation result Theorem \ref{theo:harmonicLagintro} itself. As in  \cite{GO}, this allows to run a Campanato-style iteration scheme which transfers the information
\eqref{hypstart}, namely that the displacement is controlled at the macroscopic scale (here $\barR$) down to the microscopic scale (here $1$),
provided the data are well-behaved in the sense of \eqref{hypdata2}. Due to the iterative application of Theorem \ref{theo:harmonicLagintro} this results into a cumulative shift. 
The last ingredient for the proof of Theorem \ref{theo:main} is the fact that to leading order, this shift is equal 
to the flux  averaged over a microscopic region  of the solution
of the linearized equation i.e. the Poisson equation, at the macroscopic scale.
\medskip

This is inspired by the approach of Avellaneda and Lin \cite{MR0910954} to a regularity theory for (linear) elliptic equations
with periodic coefficients: The good regularity theory of the homogenized operator, i.e.~the regularity theory
on the thermodynamic scale, is passed down to the scale of the periodicity.
This approach has been converted from a qualitative to a quantitative one in the seminal work of Armstrong and Smart \cite{MR3481355}, in the context of stochastic homogenization. It has been further refined in this context, for instance by Gloria, Neukamm and
the last author \cite{arXiv:1409.2678} (see also \cite{AKMbook}).

The Campanato iteration driving the approach of Avellaneda and Lin was originally devised for comparing a nonlinear variational problem to its linearization across scales, like the minimal surface problem. Thus not surprisingly, large-scale regularity has been applied also to nonlinear problems in homogenization, see \cite{ArFeKu20, FiNe19} for recent work. 
\medskip

Let us now describe the different steps for the proofs of Theorems \ref{theo:harmonicLagintro} and \ref{theo:main} in more detail.
We start with the harmonic approximation result.

\subsubsection{The harmonic approximation theorem}
Since all the quantities appearing in Theorem \ref{theo:harmonicLagintro} have the same scaling, it is enough to prove it for $\barR=1$. We use here the short-hand notation
\[
 E:=E(1)=\frac{1}{|B_6|}   \int_{(B_{6}\times \R^d)\cup( \R^d\times B_{6})} |x-y|^2 d\pi
\]
         and
         \[
          D:=D(1)=  \frac{1}{|B_{6}|} W^2_{B_{6}}(\mu, \kappa_\mu)+ \frac{1}{\kappa_\mu}(\kappa_\mu-1)^2+\frac{1}{|B_{6}|} W^2_{B_{6}}(\lambda, \kappa_\lambda)+ \frac{1}{\kappa_\lambda}(\kappa_\lambda-1)^2.
         \]

As in \cite{GO}, the proof is done at the Eulerian (or Benamou-Brenier) level. 
Notice that this is somewhat reminiscent of the regularity theory for the pressure field in incompressible Euler equations \cite{AmbFig,Breincompr,Santamduality}.
Recalling the definition \eqref{defrhojintro} of the density-flux pair $(\rho,j)$, we will show the following 
\begin{proposition}\label{eulerian1stepintro}
 For every $0<\tau\ll1$, there exist positive constants $\eps(\tau)>0$ and $C(\tau)>0$ such that if $E+D\le \eps$, then there exists $R\in(3,4)$ such that if $\Phi$ solves,
 \[
  \Delta \Phi=c \ \textrm{ in } B_{R} \qquad \textrm{ and } \qquad \nu \cdot \nabla \Phi=\nu\cdot \bar j \ \textrm{ on } \partial B_{R},
 \]
 then
 \begin{equation}\label{ma36intro}
  \int_{B_2}\int_0^1\frac{1}{\rho} |j-\rho \nabla \Phi|^2\le \tau E+C D.
 \end{equation}
 Moreover,
 \begin{equation}\label{estimPhiintro}
  \sup_{B_{2}}|\nabla^2 \Phi|^2+\sup_{B_2} |\nabla \Phi|^2 \les E+D.
 \end{equation}
\end{proposition}
To ease notation, let $f:= \nu\cdot j$ and $\bar f:=\int_0^1 f=\nu\cdot \bar j$. Notice that $(\rho,j)$ solves the local version of \eqref{conteqintro}:
\begin{equation}\label{conteqintroloc}
      \begin{cases}
         \partial_t \rho+\nabla\cdot j=0  & \textrm{ in } B_R,\\
          \rho_0=\mu, \ \rho_1= \lambda  &    \textrm{ in } B_R,   \\
          \nu\cdot j=  f & \textrm{ on } \partial B_R\times (0,1).
      \end{cases}                                   
\end{equation}
The strategy of the proof of Proposition \ref{eulerian1stepintro} is the same as for \cite[Prop. 3.3]{GO}, namely to choose first a good radius $R$ for which the flux $\bar f$ is well behaved on $\partial B_R$
and then rely on an almost orthogonality property (see Lemma \ref{Lortintro} below) in order to be able to use the minimization properties of $(\rho,j)$ and finally obtain the desired 
estimates through the construction of a competitor (see Lemma \ref{Lconintro}). However, each of these steps is considerably harder than in \cite{GO}. 
Indeed, since our analysis is at the $L^2$ level, to obtain $L^2$ bounds on $\nabla \Phi$, we would need $L^2$ bounds on $\bar f$ (or at least $L^{\frac{2(d-1)}{d}}$ bounds by the Sobolev trace embedding). 
However since in general $j$ is just a measure, there is no hope in this setting to obtain such bounds. In order to solve this issue, we first prove an analog of Proposition \ref{eulerian1stepintro} but with $\Phi$ replaced by
$\phi$, which is a solution of 
\begin{equation}\label{defphiintro}
 \Delta \phi=c \ \textrm{ in } B_{R} \qquad \textrm{ and } \qquad \nu \cdot \nabla \phi=\hat g \ \textrm{ on } \partial B_{R},
\end{equation}
where $\hat{g}$ is a regularization through rearrangement of $\bar f$ with good $L^2$ bounds. More specifically, treating separately incoming and outgoing fluxes $\bar f_\pm$ (so that $\bar f=\bar f_+-\bar f_-$ and similarly $\hat{g}=\hat{g}_+-\hat{g}_-$), 
we construct in Lemma \ref{Lapp} densities $\hat{g}_{\pm}$ on $\partial B_R$ with (we refer to \eqref{BBwassboundary} for the definition of $W_{\partial B_R}$)
\begin{equation}\label{ma69}
 \int_{\partial B_R} \hat{g}_\pm^2\les E+D \qquad \textrm{ and } \qquad W_{\partial B_R}^2(\bar f_{\pm},\hat g_\pm)\les (E+D)^{\frac{d+3}{d+2}}.
\end{equation}
 \begin{figure}\begin{center}
 \resizebox{7.cm}{!}{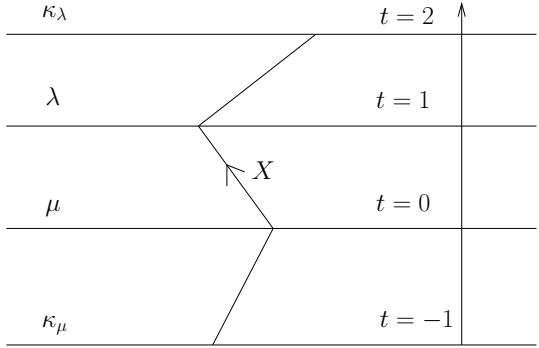} 
   \caption{The trajectories $X$.} \label{fig:gluing}
 \end{center}
 \end{figure}
 \medskip
 
\noindent In order to construct $\hat{g}_{\pm}$, we consider the time dependent version of the Lagrangian problem and couple it with the data to the effect of having trajectories
connecting the constant density $\kappa_\mu$  at time $t=-1$, 
 the measure $\mu$ at time $t=0$,  the measure $\lambda$ at time $t=1$ and finally  the constant density $\kappa_\lambda$ at time $t=2$ (see Figure \ref{fig:gluing}). Since $E+D\ll1$,
thanks to the $L^\infty$ bound on the displacement proven in Lemma \ref{Linf}, particles hitting $\partial B_R\times (0,1)$ must come from (respectively end in)  a small neighborhood of $\partial B_R$ at time $t=-1$ (respectively at time $t=2$).
A key point in   deriving \eqref{ma69} is that at time $t=-1$ and $t=2$, the densities are well-behaved (being constant) and thus the number of particles coming from such a small neighborhood is under control.
\medskip 

Once the good radius $R$ is chosen, the second step in the proof of Proposition \ref{eulerian1stepintro} is the following approximate orthogonality result (cf.\ \cite[Rem. 3.4]{GO} for an explanation of the name).

\begin{lemma}[Orthogonality]\label{Lortintro}
For every $0<r\ll 1$, there exist $\eps(r)>0$ and $C(r)>0$ such that if $E+D\le \eps$, 
  \begin{equation}\label{io15intro}
   \int_{B_2}\int_0^1 \frac{1}{\rho}|j-\rho\nabla \phi|^2- \lt(\int_{B_R}\int_0^1 \frac{1}{\rho}|j|^2-\int_{B_R}|\nabla \phi|^2\rt)
   \le r E +C D.
  \end{equation}
 \end{lemma}

In order to prove Lemma \ref{Lortintro}, we consider $\hat{g}_r$ the convolution at scale $r$ (on $\partial B_R$) of $\hat{g}$ and show that for $\phi^r$ the solution of \eqref{defphiintro} but with the (further) regularized flux $\hat{g}_r$ instead of $\hat g$,
\begin{align*}
\int_{B_R}\int_0^1 \frac{1}{\rho}|j&-\rho\nabla \phi^r|^2=\int_{B_R}\int_0^1 \frac{1}{\rho}|j|^2-\int_{B_R} |\nabla \phi^r|^2\\
 &+2\int_{B_R} \phi^r d(\mu-\lambda) +2\int_{\partial B_R} \phi^r d(\hat{g}_r-\bar f) +\int_{B_R}\int_0^1 |\nabla \phi^r|^2(d \rho-dx)
\end{align*}
and estimate the three error terms in the second line. The first term is estimated using that $\mu$ and $\lambda$ are very close in Wasserstein distance. 
The second term is estimated using the second part of \eqref{ma69}. While in \cite{GO}, the last error term was easily estimated since in that case   $\rho\le1$ (up to  a small error) by McCann's displacement convexity, we need here a more delicate argument based 
on elliptic regularity and the fact that 
\[
 W_{B_R}^2\lt(\int_0^1 \rho, \kappa\rt)\les E+D.
\]
This estimate is a consequence of a restriction result for the Wasserstein distance (see Lemma \ref{Lres}). It is here that we need the further regularization $\hat g_r$ of $\hat g$. The proof of \eqref{io15intro} is concluded using that $\nabla \phi$ and $\nabla \phi^r$ are very close  in $B_2$. 
\medskip 

In the third step we construct a competitor for the variational problem solved by $(\rho,j)$:
\begin{lemma}\label{Lconintro}
 For every $0<\tau\ll1$, there exist $\eps(\tau)>0$ and $C(\tau)>0$ such that if $E+D\le \eps$, then there exists a density-flux pair $(\tilde \rho, \tilde j)$ satisfying \eqref{conteqintroloc} and such that 
 \begin{equation}\label{io31intro}
  \int\int_0^1 \frac{1}{\tilde \rho}|\tilde{j}|^2-\int_{B_R}|\nabla \phi|^2 \le \tau E +C D.
 \end{equation}
\end{lemma}

Choosing $r=\tau$, Lemma \ref{Lconintro} together with \eqref{io15intro} and the fact that $(\rho,j)$ is minimizing concludes the proof of \eqref{ma36intro} (with $\nabla \phi$ instead of $\nabla \Phi$). 
\medskip

To show Lemma \ref{Lconintro} we introduce a small time like parameter $0<\tau\ll1$ and work separately in the initial layer $(0,\tau)$, in $(\tau,1-\tau)$ and in the final layer $(1-\tau,1)$ (this can again be compared with the approach of \cite{AmStTr16,Le17}).
We need to  separately treat a small set of exceptional trajectories which either enter $B_R$ too early or exit too late. Even though this leads to some non trivial complications (see Section \ref{sec:exceptional}),
we will ignore this issue here to keep the discussion simpler and assume that the flux $f$ is supported in $\partial B_R\times (3\tau,1-3\tau)$. Using that the data term $D$ is small, we connect the measure $\mu$ to a constant $\kappa_\mu$ in $(0,\tau)$ and another constant $\kappa_\lambda$ to $\lambda$ in $(1-\tau,1)$.  
The leading order construction in $B_R\times (\tau,1-\tau)$ is given by
\[
 \tilde{\rho}_t:= \frac{t-\tau}{1-2\tau} \kappa_\lambda+\frac{1-\tau-t}{1-2\tau} \kappa_\mu, \qquad \tilde{j}_t:=\frac{1}{1-2\tau} \nabla \phi,
\]
which connect the two constants but has the flux boundary condition $\hat{g}$ instead of $f$. To correct the boundary condition, we first make a construction in the boundary layer
$B_R\backslash \bar B_{R-r}\times (\tau, 1-\tau)$  with $0<r\ll1$, using \cite[Lem. 2.4]{GO} to connect the constant-in-$t$ density $\hat{g}$ to
\[
 \tilde{g}:=\begin{cases}
             \frac{1}{\tau} \hat{g}_+ & \textrm{for } t\in (\tau,2\tau)\\
             -\frac{1}{\tau} \hat{g}_- &\textrm{for } t\in (1-2\tau,1-\tau).
            \end{cases}
\]
Since $\tilde{g}$ is supported in $(\tau,2\tau)\cup(1-2\tau,1-\tau)$ and $f$ is supported in $(3\tau,1-3\tau)$, we can then use the second estimate in \eqref{ma69} to connnect $\tilde{g}$ to $f$ in $\partial B_R\times [0,1]$ with cost $\frac{1}{\tau} (E+D)^{\frac{d+3}{d+2}}$ (see Lemma \ref{lem:connectboundarylayer}).
\medskip 

In a  a final step, we conclude the proof of \eqref{ma36intro} using that by elliptic regularity and \eqref{ma69}, $\nabla \phi$ and $\nabla \Phi$ are very close in $B_2$,

\subsubsection{Proof of Theorem \ref{theo:main}}
Let us recall that from now on, we assume that the target measure is the Lebesgue measure. The reason for this is that we will heavily use its invariance properties under shifts.
\medskip

The first step in the proof of Theorem \ref{theo:main} is a Campanato iteration scheme based on Theorem \ref{theo:harmonicLagintro}. To state the result we need to fix some notation. 
For a sequence of approximately geometric radii i.e.
\begin{equation*}
 \barR\ge R_0\ge\cdots\ge R_K\ge C \qquad \textrm{with } \ R_{k-1}\ge 2 R_k\ge \frac{1}{C} R_{k-1}
\end{equation*}
with $\barR\sim R_0$ and $R_K\sim 1$, we set (recall \eqref{hypdata3})
\[
 E_k:=E(R_k)=\frac{1}{|B_{6R_k}|}\int_{(B_{6R_k}\times \R^d)\cup(\R^d\times B_{6R_k})} |x-y|^2 d\pi_k,
\]
where $\pi_k$ is the coupling recursively defined by $\pi_0:=\pi$ and 
\begin{equation}\label{so34intro}
 \pi_k:=({\rm id},{\rm id}-\nabla \Phi_{k-1}(0))\# \pi_{k-1},
\end{equation}
where $\Phi_k$ solves the Poisson equation
\begin{equation}\label{so33intro}
 \Delta \Phi_k=c \ \textrm{ in } B_{R_k} \qquad \textrm{ and } \qquad \nu\cdot\nabla\Phi_k=\nu\cdot \bar{j}_k \ \textrm{ on } \partial B_{R_k}.
\end{equation}
The flux $j_k$ is defined as in \eqref{defrhojintro} with $\pi_k$ playing the role of $\pi$ and 
 $\bar j_k$ is then  obtained by integrating in time.

 \begin{proposition}\label{theo:campintro}
 Assume that
 \begin{multline*}
  \frac{1}{|B_{\barR}|}\int_{(B_{\barR}\times\R^d)\cup(\R^d\times B_{\barR})}|x-y|^2 d\pi \le \beta(\barR) \\
  \textrm{and } \quad \frac{1}{|B_R|} W^2_{B_R}(\mu,\kappa)\le \beta(R) \qquad \forall R\in [1,\barR].
 \end{multline*}
Then, there exists a sequence of approximately geometric radii $(R_k)_{0\leq k \leq K}$ such that 
\begin{equation}\label{so01 intro}
 E_k\les \beta(R_k),
\end{equation}
and 
\begin{equation}\label{so39 intro}
 |\nabla \Phi_k(0)|^2 \les \beta(R_k).
\end{equation}
 \end{proposition}
Assuming that $\beta(R)=R^\alpha$, we can see that we are performing here  a Campanato iteration   at the $C^{0,\frac{\alpha}{2}}$ level for the displacement rather than at the $C^{1,\frac{\alpha}{2}}$ level as in \cite{GO}.
In fact, we can run the iteration as long as $\frac1{R_k^2}E_k + \frac{1}{R_k^2}D(R_k)\ll 1$. For large scales $R_k\gg 1$ this holds for $\alpha<2$ 
(and notice that Proposition \ref{theo:campintro} actually would hold under this weaker assumption on $\alpha$),
while  for small scales $R_k\ll1$ it would require $\alpha>2$ (and this would be inconsistent).          
\medskip

From \eqref{so01 intro} and the iterative definition \eqref{so34intro}, we obtain an estimate in the strong $L^2$ topology   of the deviation of the displacement $y-x$ from the cumulated shift $\tilde{h}:=\sum_{k=0}^{K-1} \nabla \Phi_k(0)$.  
The main ingredient to upgrade this into the statement of Theorem \ref{theo:main} is a better identification of the shift, taking into account cancellation effects. Indeed, a naive computation gives
\begin{equation}\label{eqtildeh}|\tilde{h}|\le \sum_{k=0}^{K-1} |\nabla \Phi_k(0)|\stackrel{\eqref{so39 intro}}{\les} K \beta^{\frac{1}{2}}(R_0)\les (\log \barR )\,  \beta^{\frac{1}{2}}(\barR),  \end{equation}
which fails by $\log \barR$ the expected estimate $ |\tilde{h}|\les \beta^{\frac{1}{2}}(\barR)$. 
\medskip 

To state the identification result  let $(R_k)_k$ be the sequence of radii from Proposition \ref{theo:campintro} and
 for each $k\in [0,K]$,  let  $u_k$ be the (distributional) solution to
 \begin{align}\label{so16intro}
\Delta u_k=\mu-1\ \textrm{ in } B_{R_k}\qquad\textrm{ and }\qquad
	\nu\cdot\nabla u_k=\nu\cdot\bar j_k\ \textrm{ on } \partial B_{R_k}.
\end{align}
Then we have
\begin{proposition}\label{prop:repres intro}
For every $k\in [0,K-1]$ and $R\in[R_{k+1}, R_k]$, we have the two estimates on the cumulative flux
\begin{align}\label{ao08intro}
\lt|\sum_{\ell=0}^k\nabla\Phi_\ell(0)
-\frac{1}{|B_{R_k}|}\int_{\partial B_{R_k}}x\nu\cdot\nabla u_0\rt|\lesssim \frac{\beta(R_k)}{R_k}
\end{align}
and
\begin{align}\label{ao21intro}
\lt|\sum_{\ell=0}^{k-1}\nabla\Phi_\ell(0)+\int\eta_R\nabla u_k
-\int\eta_R\nabla u_0\rt|\lesssim \frac{\beta(R_k)}{R_k},
\end{align}
where $\eta_R$ is as in \eqref{cutoff}.
 \end{proposition} 
   Since \eqref{ao21intro} is a fairly direct consequence of \eqref{ao08intro}, we will just comment on the proof of the latter.
For  each $k\in[0,K]$ let $v_k$ be the distributional solution to
\begin{equation}\label{defvkintro}
 \Delta v_k=\mu-1-c\ \textrm{ in } B_{R_k}\qquad\textrm{ and }\qquad
	\nu\cdot\nabla v_k=0\ \textrm{ on } \partial B_{R_k}
\end{equation}
so that $u_k=v_k+\Phi_k$. The first ingredient for the proof of \eqref{ao08intro} is the following estimate for fixed $k$.
\begin{lemma}\label{lem:shiftonestep}
 For every $k\in[1,K-1]$,
 \begin{equation}\label{so45intro}
  \lt|\nabla\Phi_{k}(0)
-\frac{1}{|B_{R_{k}}|}\int_{\partial B_{R_{k}}}x\nu\cdot\nabla v_{k-1}\rt|
	\lesssim \frac{\beta(R_{k})}{R_{k}}.
 \end{equation}

\end{lemma}
This result is in itself quite remarkable since it proves that the local shift $\nabla\Phi_{k}(0)$ depends (to leading order) only on the local structure of $\mu$ in $B_{R_{k-1}}$
(since $v_{k-1}$ only depends on $\mu\restr B_{R_{k-1}}$).
\medskip

   Lemma \ref{lem:shiftonestep} follows itself from the combination of two facts. The first is that the shift $\nabla \Phi_{k-1}(0)$, cf.\ \eqref{so34intro}, which is additive in the displacement retains an almost additive effect on the Eulerian flux:
   \begin{lemma}\label{lem:additive}
    For every $k\in [1,K]$,
    \begin{equation}\label{so42intro}
      \lt| \frac{1}{|B_{R_{k}}|}\int_{B_{R_{k}}} (\bar j_{k} +\nabla \Phi_{k-1}(0)-\bar j_{k-1})\rt|\les\frac{\beta(R_{k})}{R_{k}}.
    \end{equation}

   \end{lemma}
This is proven using that $R_{k}$ is a good radius. The second, and more substantial 
result leading to Lemma \ref{lem:shiftonestep} is a linearization result at the level of  the Eulerian flux:
\begin{lemma}\label{lem:Le1intro}
 Recall the definition \eqref{so16intro} of $u_k$. For $R\in[R_{k+1},R_k]$, there holds 

\begin{align}
 	\lt|\int\eta_R(\bar j_k-\nabla u_k)\rt|&\lesssim\frac{E_k}{R_k},\label{ao13intro}\\
	\frac{1}{|B_{R_{k+1}}|}\lt|\int_{B_{R_{k+1}}}(\bar j_k-\nabla u_k)\rt|&\lesssim \frac{E_k}{R_k}.\label{so06intro}
\end{align}
\end{lemma}
While \eqref{ao13intro} is very important in combination with \eqref{ao21intro} to obtain  estimate \eqref{mainweaktheo}, we are mainly interested  in \eqref{so06intro} for this discussion. 
Both estimates are obtained by proving that the divergence free vector field $\bar j_k-\nabla u_k$ is weakly close on large scale to a gradient field, namely the gradient of
$\frac{1}{2}|y|^2-\psi_k-u_k$, where $\psi_k$ is the Kantorovich potential of $\pi_k$.\\

Since it is at the basis of our argument, let us derive Lemma \ref{lem:shiftonestep} from Lemma \ref{lem:additive} and Lemma \ref{lem:Le1intro} 
\begin{proof}[Proof of Lemma \ref{lem:shiftonestep}]
 The starting point is that since $\nabla \Phi_{k}$ is harmonic and thus satisfies the mean-value property and since, concentrating for instance on the first variable, 
 $\nabla\cdot (x_1\nabla \Phi_{k})= \partial_1 \Phi_k +x_1 c$, we have by the divergence theorem
 \begin{equation}\label{eq:div}
  \nabla \Phi_{k}(0)=\frac{1}{|B_{R_k}|}\int_{B_{R_k}} \nabla \Phi_{k}=\frac{1}{|B_{R_k}|}\int_{\partial B_{R_k}} x \nu \cdot \nabla \Phi_{k}.
 \end{equation}
Observe that by \eqref{conteqintro} $\nabla \cdot \bar{j}_k=\mu-1$ in $B_{R_k}$. Combining this with \eqref{defvkintro} implies that $\nabla \cdot (\bar{j}_{k}-\nabla v_{k-1})= c$ in $B_{R_{k}}$. Hence, we have again by the divergence theorem

\begin{align*}
 \nabla \Phi_{k}(0)-\frac{1}{|B_{R_{k}}|}\int_{\partial B_{R_{k}}}x\nu\cdot\nabla v_{k-1}&= \frac{1}{|B_{R_k}|}\int_{\partial B_{R_k}} x \nu \cdot (\nabla \Phi_{k}- \nabla v_{k-1})\\
 &\stackrel{\eqref{so33intro}}{=}\frac{1}{|B_{R_k}|}\int_{\partial B_{R_k}} x \nu \cdot (\bar j_{k}- \nabla v_{k-1})\\
 &=\frac{1}{|B_{R_k}|}\int_{B_{R_k}} (\bar{j}_{k}-\nabla v_{k-1}).
\end{align*}
Writing  
\[
 \bar{j}_{k}-\nabla v_{k-1}= (\bar{j}_{k} +\nabla \Phi_{k-1}(0)-\bar j_{k-1})+ (\bar j_{k-1}-\nabla \Phi_{k-1}(0)-\nabla v_{k-1}),
\]
using \eqref{eq:div} for $\nabla \Phi_{k-1}$ and the fact that $u_{k-1}=\Phi_{k-1}+v_{k-1}$, we obtain as desired
\begin{align*}
 \lefteqn{\lt|\nabla \Phi_{k}(0)-\frac{1}{|B_{R_{k}}|}\int_{\partial B_{R_{k}}}x\nu\cdot\nabla v_{k-1}\rt|}\\
 &\le \lt|\frac{1}{|B_{R_k}|}\int_{B_{R_k}} (\bar{j}_{k} +\nabla \Phi_{k-1}(0)-\bar j_{k-1}) \rt| +\lt|\frac{1}{|B_{R_k}|}\int_{B_{R_k}}(\bar j_{k-1}-\nabla u_{k-1})\rt|\\
 &\stackrel{\eqref{so42intro}\&\eqref{so06intro}}{\les} \frac{\beta(R_{k})}{R_{k}} +\frac{E_{k-1}}{R_{k-1}}\\
 &\stackrel{\eqref{so01 intro}}{\les} \frac{\beta(R_{k})}{R_{k}}.
\end{align*}

\end{proof}

Besides Lemma \ref{lem:shiftonestep}, the second main ingredient for the proof of \eqref{ao08intro} is a summation formula for the fluxes of $v_k$.
\begin{lemma}\label{lem:sumvk}
 For every $k\in [1,K]$,
\begin{equation}\label{eq2intro} \frac{1}{|B_{R_k}|}\int_{\partial B_{R_k}}x\nu\cdot\nabla v_{0}
	=\sum_{\ell=1}^{k}\frac{1}{|B_{R_{\ell}}|}\int_{\partial B_{R_{\ell}}}x\nu\cdot\nabla v_{\ell-1}.
	\end{equation}
\end{lemma}
The proof of \eqref{eq2intro} is based on the fact that each $v_k$ type of solves the same equation (see \eqref{defvkintro}) and that we may find an explicit formula relating the flux  $\frac{1}{|B_{R_{k}}|}\int_{\partial B_{R_{k}}} x\nu\cdot \nabla v_{k-1}$ and the measure $\mu$.
\medskip 

Let us point out that \eqref{ao08intro} is a direct consequence of \eqref{so45intro}, \eqref{eq2intro} and hypothesis \eqref{hypbeta} on $\beta$, once observed that   by \eqref{eq:div}, and since $u_0=v_0+\Phi_0$,
\[
 \nabla \Phi_0(0)-\frac{1}{|B_{R_k}|}\int_{\partial B_{R_k}} x\nu \cdot \nabla u_0=-\frac{1}{|B_{R_k}|}\int_{\partial B_{R_k}} x\nu \cdot \nabla v_0.
\]

\subsection{Comparison with the first part of \cite{GHO}}
We close this introduction by stressing the connection and differences between this paper and \cite{GHO}. This paper replaces and supersedes the first (deterministic) part of \cite{GHO} (which will thus not be submitted anywhere).
The second part containing the application to the optimal matching problem will appear elsewhere. A first main difference with \cite{GHO}
is that we are  able here to directly relate the displacements to the flux given by the Poisson equation (see \eqref{mainweaktheo}) and thus confirming the linearization ansatz by Caracciolo and al. from \cite{CaLuPaSi14}.
As a by-product, we obtain optimal estimates for the shift $h$ in \eqref{strongestimintro} while \cite[Th. 1.1]{GHO} was only optimal up to a logarithmic factor  in the system size i.e $|h_R|\les\log^{\frac{3}{2}} \barR$ in dimension $d=2$ and $|h_R|\les \log \barR$ in higher dimension (see \eqref{eqtildeh}).   
Another main difference is in the statement and proofs of the harmonic approximation result (see Theorem \ref{theo:harmonicLagintro}). Compared with \cite[Prop. 1.5]{GHO}, where the source was given by the Lebesgue measure, 
it is stated  here for general source and target measures. While this might look like a minor point, it actually means that we cannot rely at all on McCann's 
interpolation convexity estimate to obtain an $L^\infty$ bound on the Eulerian density $\rho$ (since in general it is only a measure), see for instance \cite[Lem.2.4]{GHO}.
Having to go around this relatively rigid argument actually led to a completely revised, simpler and hopefully more flexible proof. 
Let us point out that having to connect arbitrary source and target measures also called for a new proof of the $L^\infty$ bound on the displacement (compare the proofs of Lemma \ref{Linf} and \cite[Lem. 4.1]{GO}). 
A last difference, more on the technical side, is that we substantially simplified the Camapanato iteration with respect to \cite[Th. 1.1]{GHO}.

\subsection*{Acknowledgements}
 MH gratefully acknowledges partial support by the DFG through the CRC 1060 ``The Mathematics of Emerging Effects'' and by the Hausdorff Center for Mathematics during the first part of this project. During the second part of this work MH has been funded by the Vienna Science and Technology Fund (WWTF) through project VRG17-005. In the final part MH has been funded by the Deutsche Forschungsgemeinschaft (DFG, German Research Foundation) under Germany's Excellence Strategy EXC 2044 -390685587, 
 Mathematics M\"unster: Dynamics-Geometry-Structure
 MG acknowledges partial support from the ANR project SHAPO.  MG and MH thank the Max Planck Institute MIS for its warm hospitality.

 \section{Preliminaries}
 \subsection{Notation}

In this paper we will use the following notation.  We recall from the introduction that the symbols $\sim$, $\ges$, $\les$ indicate estimates that hold up to a global constant $C$,
which typically only depends on the dimension $d$.
For instance, $f\les g$ means that there exists such a constant with $f\le Cg$,
$f\sim g$ means $f\les g$ and $g\les f$. An assumption of the form $f\ll1$ means that there exists $\eps>0$, typically only
depending on dimension, such that if $f\le\eps$, 
then the conclusion holds.  We write $\log$ for the natural logarithm. 

We denote by $\H^k$ the $k-$dimensional Hausdorff measure. 
For a set $E$,  $\nu_E$ will always denote the external normal to $E$. When clear from the context we will drop the explicit dependence on the set. 
We write $|E|$ for the  Lebesgue measure of a set $E$ and $I(E)$ for the indicator function of $E$. 
When no confusion is possible, we will drop the integration measures in the integrals as well as the domain of integration. Similarly, we will often identify, if possible, measures with their densities with respect to the Lebesgue measure.
For $R>0$ and $x_0\in \R^d$, $B_R(x_0)$ denotes the ball of radius $R$ centered in $x_0$. 
When $x_0=0$, we will simply write $B_R$ for $B_R(0)$. We denote the gradient on $\partial B_R$ by $\nabla_{\tang}$. For a measure $\rho$ on $\R^d\times [0,1]$, we denote by $\bar \rho$
its time integral i.e., the measure defined for $\zeta\in C^0_c(\R^d)$ as 
\begin{equation}\label{defbar}
 \int_{\R^d} \zeta d\bar \rho=\int_{\R^d}\int_0^1 \zeta d\rho.
\end{equation}
From now on, we use the letter  $\zeta$ for a generic real-valued test function and $\xi$ for a generic vector-valued test function, that is $\xi\in C^0_c(\R^d)^d$.
For a given radius $R>0$, we let $\Pi_R(x):= R \frac{x}{|x|}$ be the projection on $\partial B_R$ and define for every measure $\rho$ on $\R^d$, the projected measure on $\partial B_R$, $\hat{\rho}:=\Pi_R\# \rho$, that is
\begin{equation}\label{defhat}
\int \zeta d\hat{\rho}=\int \zeta \lt(R \frac{x}{|x|}\rt) d\rho(x).
\end{equation}
For given $f$ and $g$ (potentially measures), we let $c$ be the generic constant such that the Poisson equation
\[
 \Delta \vp = g+c \ \textrm{ in } B_R \qquad \textrm{ and } \qquad \nu\cdot \nabla \vp=f \ \textrm{ on } \partial B_R
\]
is solvable i.e 
\[
 c=\frac{1}{|B_R|}\lt(\int_{\partial B_R} df -\int_{B_R} dg\rt).
\]
We use $W^2$ for the squared $2-$Wasserstein distance (see Section \ref{sec:OT} for more details). For a (positive) measure $\mu$ and a set $O\subset \R^d$, we define $\kappa:=\frac{\mu(O)}{|O|}$
the generic constant such that
\begin{equation}\label{defkappa}
 W^2_{O}(\mu,\kappa):=W^2(\mu\restr O, \kappa dx\restr O)
\end{equation}
is finite. For $\zeta$ a function on $\R^d\times \R$ and $t\in \R$, we will often use the short-hand notation $\zeta_t:=\zeta(\cdot,t)$.

 \subsection{Elliptic estimates}
 We gather in this section a few technical lemmas giving more or less standard estimates for solutions of the Poisson equation.
 \begin{lemma}\label{lem:ellipticr}
  Let $R\in (3,4)$ and $0\le r\ll1$ be fixed. Letting $f_r$ denote the convolution   at scale $r$ on $\partial B_R$ (with the convention that $f_0=f$) and then $\vp^r$ with $\int_{B_R} \vp^r=0$ be the solution of 
  \[
   \Delta \vp^r=c \ \textrm{ in } B_R \qquad \textrm{ and } \qquad 
   \nu\cdot \nabla \vp^r= f_r \ \textrm{ on } \partial B_R.
  \]
  Then $\vp^r$ is regular up to the boundary with
\begin{align}\label{wg03}
r^2\sup_{B_R}|\nabla^3\vp^r|+r\sup_{B_R}|\nabla^2\vp^r|+\sup_{B_R}|\nabla\vp^r|
\lesssim\lt(\frac{1}{r^{d-1}}\int_{\partial B_R}f^2\rt)^\frac{1}{2},
\end{align}
 is close to $\vp$ in the interior
\begin{align}\label{wg17}
\sup_{B_2}|\nabla(\vp^r-\vp)|\lesssim r\lt(\int_{\partial B_R}f^2\rt)^\frac{1}{2}
\end{align}
and satisfies the uniform-in-$r$ regularity estimates both up to the boundary 
\begin{align}\label{wg15}
\lt(\int_{\partial B_s}|\nabla\vp^r|^2\rt)^\frac{1}{2}
\lesssim\lt(\int_{\partial B_R}f_r^2\rt)^\frac{1}{2} \qquad \forall\  0<s\le R,
\end{align}
and in the  interior  
\begin{align}\label{wg12}
\sup_{B_2}|\nabla^2\vp^r|+\sup_{B_2}|\nabla\vp^r|\lesssim\lt(\int_{B_R}|\nabla\vp^r|^2\rt)^\frac{1}{2}
\lesssim\lt(\int_{\partial B_R}f_r^2\rt)^\frac{1}{2}.
\end{align}
 \end{lemma}
\begin{proof}
A proof of  \eqref{wg15} and \eqref{wg12} can be found for instance in \cite[Lem. 2.1]{GO}, so we are left to show \eqref{wg03} and \eqref{wg17}.

\medskip 
 Up to subtracting a quadratic function to $\vp^r$, we may assume without loss of generality that $c=0$ (and thus $\vp^r$ is harmonic in $B_R$). We start by the argument for (\ref{wg03}) and restrict to the (more involved) estimate on $\nabla^3\vp^r$. 
By the  maximum principle for the component-wise harmonic $\nabla^3\vp^r$, it is enough to establish
\begin{align*}
r^2\sup_{\partial B_R}|\nabla^3\vp^r|\lesssim\lt(\frac{1}{r^{d-1}}\int_{\partial B_R}f^2\rt)^\frac{1}{2}.
\end{align*}
Fixing an integer $n>\frac{d-1}{2}$,  we first claim that for $r\in(0,1]$,
\begin{equation}\label{Sob}
 \sup_{\partial B_R}|\zeta|\lesssim r^n\lt(\frac{1}{r^{d-1}}\int_{\partial B_R}|\nabla_{tan}^n\zeta|^2\rt)^\frac{1}{2}+\lt(\frac{1}{r^{d-1}}\int_{\partial B_R}|\zeta|^2\rt)^\frac{1}{2}.
\end{equation}
 To prove \eqref{Sob}, we start with the standard Sobolev embedding on $\R^{d-1}$, $\sup_{\mathbb{R}^{d-1}}|\zeta|^2\lesssim\int_{\mathbb{R}^d}(|\nabla^n\zeta|^2+\zeta^2)$, which by scaling turns into 
\[
 \sup_{\mathbb{R}^{d-1}}|\zeta|^2\lesssim r^{2n- (d-1)} \int_{\mathbb{R}^d}|\nabla^n\zeta|^2+ \frac{1}{r^{d-1}}\int_{\R^{d-1}}\zeta^2. 
\]
This can easily be localized (with help of an extension operator) to
\[
 \sup_{[0,1]^{d-1}}|\zeta|^2\lesssim r^{2n- (d-1)} \int_{[0,1]^{d-1}}|\nabla^n\zeta|^2+ \frac{1}{r^{d-1}}\int_{[0,1]^{d-1}}\zeta^2. 
\]
Covering $\partial B_R$ by finitely many patches that are diffeomorphic to
$[0,1]^{d-1}$, we conclude the proof of \eqref{Sob}.

\medskip 
 Applying \eqref{Sob} to $\zeta=\nabla^3\vp^r$, we are left with the proof of 
 \[
  r^{n+2}\lt(\int_{\partial B_R}|\nabla_{tan}^n\nabla^3\vp^r|^2\rt)^\frac{1}{2}+r^2\lt(\int_{\partial B_R}|\nabla^3\vp^r|^2\rt)^\frac{1}{2} \lesssim\lt(\int_{\partial B_R}f^2\rt)^\frac{1}{2}.
 \]
 This is a direct consequence of the estimate
 \begin{equation}\label{ven01}
 \int_{\partial B_R}|\nabla_{tan}^m\nabla^3\vp^r|^2\les \int_{\partial B_R}|\nabla_{tan}^{m+2}f_r|^2,
 \end{equation}
which will be shown to  hold for every $m\in\N$, and the fact that 
\[r^{m+2}\lt(\int_{\partial B_R}|\nabla_{tan}^{m+2}f_r|^2\rt)^\frac{1}{2}\lesssim\lt(\int_{\partial B_R}f^2\rt)^\frac{1}{2},\] which follows from the properties of the convolution.
\medskip

 Let us prove \eqref{ven01} and fix $m\in \N$. By Pohozaev identity (see \cite[Lem. 8.3.2]{Dupaigne}) we have for every harmonic function $\zeta$ in $B_R$,
\begin{equation}\label{Poho}
 \int_{\partial B_R} |\nabla \zeta|^2\les \int_{\partial B_R} |\nabla_{tan} \zeta|^2\les{\int_{\partial B_R} |\nu\cdot\nabla \zeta|^2}.
\end{equation}
In particular, the Dirichlet-to-Neumann map
$ \zeta\mapsto \nabla \zeta$ is bounded from $H^1(\partial B_R)$ into $L^2(\partial B_R)$.
As the map commutes with rotations, it also commutes
with their infinitesimal generators $\{x_i\partial_j-x_j\partial_i\}_{i\not=j=1,\cdots,d}$, 
which in turn span the tangent space of $\partial B_R$. Hence, for $k\in \N$,
the map is also bounded from $H^{k+1}(\partial B_R)$ to $H^{k}(\partial B_R)$. Applying this first to $\zeta=\nabla^2\vp^r$ and $k=m$ and then to $\zeta=\nabla \vp^r$ and $k=m+1$, we get 
\[
 \int_{\partial B_R}|\nabla_{tan}^m\nabla^3\vp^r|^2\les \int_{\partial B_R} |\nabla^{m+2}_{tan} \nabla \vp^r|^2.
\]
Arguing as above we obtain that also the Neumann-to-Dirichlet map  $\nu\cdot \nabla\zeta\to \zeta$ is bounded from $H^{k}(\partial B_R)$ to $H^{k+1}(\partial B_R)$ for $k\in \N$.
Applying this to $\zeta=\vp^r$ and $k=m+2$, we conclude the proof of \eqref{ven01}.

\medskip 

We finally prove \eqref{wg17}. By interior regularity for the harmonic $\vp^r-\vp$, we have
\begin{align*}
\sup_{B_2}|\nabla(\vp^r-\vp)|\lesssim\lt(\int_{\partial B_R}(\vp^r-\vp)^2\rt)^\frac{1}{2}\les  r \lt(\int_{\partial B_R}|\nabla_{tan} \vp|^2\rt)^\frac{1}{2}\stackrel{\eqref{Poho}}{\les}r \lt(\int_{\partial B_R}|f|^2\rt)^\frac{1}{2},
\end{align*}
where in the last inequality we used that $\vp$ is harmonic to apply \eqref{Poho}.
 \end{proof}

 \begin{lemma}\label{lem:elliptic}
 For $R>0$, recall that  $\Pi_R(x)= R \frac{x}{|x|}$. For  every $(f,g)$, with $g\ge 0$ and $\spt g \subset \overline{B}_R\backslash B_{\frac{R}{2}}$, letting $\hat g:= \Pi_R\# g$ (see \eqref{defhat}), the   unique solution $\vp$  (up to additive constants) of 
 \[
   \Delta \vp=g +c \ \textrm{ in } B_R \qquad \textrm{ and } \qquad 
    \nu\cdot \nabla\vp= f \ \textrm{ on } \partial B_R\]
  satisfies
  \begin{equation}\label{estim:elliptic}
   \int_{B_R}|\nabla \vp|^2\les R \lt( \int_{\partial B_R} f^2+ \int_{\partial B_R}\hat g^2\rt).
  \end{equation}
\end{lemma}
\begin{proof}
By scaling, we may assume that $R=1$. 
 Without loss of generality, we may also assume that $\int_{B_1} \vp=0$ and that $\int_{\partial B_1} f^2+ \int_{\partial B_1} \hat g^2<\infty$, otherwise there is nothing to prove.
 Using integration by parts, the trace inequality for Sobolev functions together with the Poincar\'e inequality for functions of mean zero,
 \begin{align*}
  \int_{B_1}|\nabla \vp|^2&=\int_{\partial B_1} \vp f -\int_{B_1} \vp \Delta \vp\\
  &\le \lt(\int_{\partial B_1} \vp^2\rt)^{\frac12} \lt(\int_{\partial B_1} f^2\rt)^{\frac12} +\lt|\int_{B_1} g \vp\rt|\\
  &\les \lt(\int_{ B_1} |\nabla \vp|^2\rt)^{\frac12} \lt(\int_{\partial B_1} f^2\rt)^{\frac12} +\lt|\int_{B_1} g \vp\rt|.
 \end{align*}
Hence, it is  now enough to prove that 
\begin{equation}\label{eq:toproveelliptic}
 \lt|\int_{B_1} g \vp\rt|\les\lt(\int_{ B_1} |\nabla \vp|^2\rt)^{\frac12} \lt(\int_{\partial B_1} \hat g^2\rt)^{\frac12}. 
\end{equation}
Using that by hypothesis $\spt g \subset \overline{B}_1\backslash B_{\frac{1}{2}}$, we can write by definition \eqref{defhat} 
\begin{align*}
 \lt|\int_{B_1} g \vp\rt| &\le \lt|\int_{\partial B_1} \vp \hat g\rt|+\lt|\int_{\partial B_1} \int_{\frac{1}{2}}^1 (\vp(r \omega)-\vp(\omega)) g(r\omega) r^{d-1} dr d\omega\rt|\\
 &\les \lt(\int_{\partial B_1} \vp^2\rt)^{\frac12} \lt(\int_{\partial B_1} \hat g^2\rt)^{\frac12} +\int_{\partial B_1} \int_{\frac{1}{2}}^1 |\vp(r \omega)-\vp(\omega)| g(r\omega) r^{d-1} dr d\omega\\
 &\les \lt(\int_{ B_1} |\nabla \vp|^2\rt)^{\frac12} \lt(\int_{\partial B_1}\hat g^2\rt)^{\frac12}+\int_{\partial B_1} \int_{\frac{1}{2}}^1 |\vp(r \omega)-\vp(\omega)| g(r\omega) r^{d-1} dr d\omega,
\end{align*}
where in the last line we  used once more  that $\int_{\partial B_1} \vp^2\les \int_{B_1} |\nabla \vp|^2$. Since for $r\in(\frac{1}{2},1)$
\[
 |\vp(r\omega)-\vp(\omega)|\le \lt(\int_{\frac{1}{2}}^1 |\partial_r \vp(s \omega)|^2  ds \rt)^{\frac12}\les  \lt(\int_{0}^1 |\partial_r \vp(s \omega)|^2 s^{d-1} ds \rt)^{\frac12},
\]
estimate \eqref{eq:toproveelliptic} follows from 
\begin{align*}
& \int_{\partial B_1} \int_{\frac{1}{2}}^1 |\vp(r \omega)-\vp(\omega)| g(r\omega) r^{d-1} dr d\omega\\
&\les \int_{\partial B_1} \int_{\frac{1}{2}}^1 \lt(\int_{0}^1 |\partial_r \vp(s \omega)|^2 s^{d-1} ds\rt)^{\frac12}g(r\omega) r^{d-1} dr d\omega\\
 &\stackrel{{\eqref{defhat}}}{=}\int_{\partial B_1}  \lt(\int_{0}^1 |\partial_r \vp(s \omega)|^2 s^{d-1}ds \rt)^{\frac12} \hat g(\omega) d\omega\\
 &\le \lt(\int_{\partial B_1}\int_{0}^1 |\partial_r \vp(s \omega)|^2 s^{d-1}ds\rt)^{\frac12} \lt(\int_{\partial B_1}\hat g^2\rt)^{\frac12}.
\end{align*}
\end{proof}
Lemma \ref{lem:elliptic} will often be used in combination with the following elementary estimate.
\begin{lemma}\label{2.3}
  For $R>0$, recall that $\Pi_R(x)= R \frac{x}{|x|}$ denotes the projection on $\partial B_R$. There exists $1\geq\eps(d)>0$ (depending only on the dimension) such that
  for every   $g\ge 0$ with $\spt g\subset B_{(1+\eps)R}\backslash B_{(1-\eps)R}$,  setting $\hat g:= \Pi_R \# g$ (see \eqref{defhat}), we have 
  \begin{equation}\label{estimz}
   R^{1-d}\lt(\int g\rt)^2\les \int_{\partial B_R} \hat g^2\les \sup g \int |R-|x|| g.
  \end{equation}

 \end{lemma}
\begin{proof}
 By scaling we may assume that $R=\sup g=1$. The first inequality is a direct consequence of Cauchy-Schwarz inequality in the form
 \[
  \lt(\int g\rt)^2=\lt(\int_{\partial B_1}\hat g\rt)^2\les \int_{\partial B_1}\hat g^2.
 \]
We now turn to the second inequality.  Let us first point out that if $\eps\ll1$, we have $\sup_{\partial B_1}|\hat{g}|\ll1$ since $\spt g\subset B_{1+\eps}\backslash B_{1-\eps}$. Momentarily fixing $\omega\in \partial B_1$ and setting $\psi(r):= r^{d-1} g(r\omega)$ for $r>0$, 
we have $0\le \psi\le (1+\eps)^{d-1}\le 2$ and 
\[
 \int_{0}^{\infty} \psi = \hat g(\omega),
\]
so that for  $\omega\in \partial B_1$,
\[
 \int_{0}^\infty |1-r| r^{d-1} g(r\omega)\ge \min_{\stackrel{0\le \tilde\psi\le 2}{\int \tilde\psi =\hat g(\omega)}} \int_{0}^\infty |1-r| \tilde\psi(r) \ges \hat g(\omega),
\]
where the last inequality follows since the minimizer of 
\[
 \min_{\stackrel{0\le \tilde\psi\le 2}{\int \tilde\psi =\hat g(\omega)}} \int_{0}^\infty |1-r| \tilde\psi(r)
\]
is given  by  $2I(|r-1|\le \frac{1}{4}\hat g(\omega))$ (which is admissible since $|\hat{g}(\omega)|\ll1$).
\end{proof}

\begin{lemma}\label{lem:difphi}
 For every $R\in (3,4)$, every positive measures $f$ and $g$ on $\partial B_R$ with $f(\partial B_R)=g(\partial B_R)$ and every $0<\alpha<1$, the solution $\vp$ with $\int_{B_R} \vp=0$ of 
 \[
  \Delta \vp= 0 \ \textrm{ in } B_R \qquad \textrm{ and } \qquad  \nu\cdot \nabla \vp = f-g \  \textrm{ on } \partial B_R
 \]
satisfies
\begin{equation}\label{ma34}
 \sup_{B_2} |\nabla^2 \vp|+\sup_{B_2} |\nabla \vp|\les W_{\partial B_R}^\alpha(f,g) \lt[g(\partial B_R)\rt]^{1-\frac{\alpha}{2}},
\end{equation}
where the implicit constant depends on $\alpha$.
\end{lemma}
\begin{proof}
 Let $1<p<\frac{d}{d-1}$ be defined by $\frac{1}{p}= \frac{\alpha}{d}+ \frac{d-1}{d}$. By the mean value formula, 
 \[
  \sup_{B_2} |\nabla^2 \vp|+ \sup_{B_2} |\nabla \vp|\les \lt(\int_{B_R} |\nabla \vp|^p\rt)^{\frac{1}{p}}
 \]
and it is therefore enough to prove   that 
\begin{equation}\label{ma35}
 \lt(\int_{B_R} |\nabla \vp|^p\rt)^{\frac{1}{p}}\les W_{\partial B_R}^\alpha(f,g) \lt[g(\partial B_R)\rt]^{1-\frac{\alpha}{2}}.
\end{equation}
We prove this estimate  by duality. Consider $q$ the dual exponent of $p$, i.e. $\frac{1}{p}+\frac{1}{q}=1$, and for arbitrary $\xi$ smooth and compactly supported in $B_R$ denote by $\omega$ the solution of the (dual) problem
\begin{equation}\label{ao09}
 \Delta \omega =\Div \xi \ \textrm{ in } B_R \qquad \textrm{ and } \qquad \nu\cdot \nabla \omega=0 \  \textrm{ on } \partial B_R
\end{equation}
with $\int_{B_R} \omega=0$. Noting that $q>d$ and $\alpha=1-\frac{d}{q}$, Sobolev's embedding  followed by  Calder\'on-Zygmund estimates for the Neumann-Laplacian 
(see for instance \cite[Th. 3.16]{Troianello}) yields
\[
 [\omega]_{\alpha,B_R}\les \lt(\int_{B_R} |\nabla \omega|^q\rt)^{\frac{1}{q}}\les \lt(\int_{B_R} |\xi|^q\rt)^{\frac{1}{q}}.
\]
Let $\pi$ be an optimal transference plan between $f$ and $g$. Using integration by parts we obtain
\begin{align*}
\int_{B_R} \nabla \vp \cdot \xi= -\int_{B_R} \vp\, \Div \xi&=\int_{\partial B_R} \omega d(f-g) \\
 &= \int_{\partial B_R\times \partial B_R} (\omega(x)-\omega(y)) d\pi(x,y)\\
 &\les [\omega]_{\alpha, B_R} \int_{\partial B_R\times \partial B_R} |x-y|^\alpha d\pi\\
  &\stackrel{\textrm{H\"older}}{\les} \lt(\int_{B_R} |\xi|^q\rt)^{\frac{1}{q}} W^\alpha(f,g) \lt[g(\partial B_R)\rt]^{1-\frac{\alpha}{2}}\\
 &\les \lt(\int_{B_R} |\xi|^q\rt)^{\frac{1}{q}} W_{\partial B_R}^\alpha(f,g) \lt[g(\partial B_R)\rt]^{1-\frac{\alpha}{2}}.
\end{align*}
 This concludes the proof of \eqref{ma35}.
\end{proof}

 Arguing along the same lines, we may also prove:
  \begin{lemma}\label{lem:gradu}
 For $R>0$, let $\mu$ be a signed measure on $B_R$ and $f$ be a signed measure on $\partial B_R$, both with finite total variation. Then, if $\vp$ solves
 \[
  \Delta \vp=\mu \ \textrm{ in } B_R \qquad \textrm{ and } \qquad \nu \cdot \nabla \vp=f \ \textrm{ on } \partial B_R,
 \]
we have $\nabla \vp \in L^p(B_{R})$ for each exponent $1\leq p < \frac{d}{d-1}.$
\end{lemma}
 \subsection{The optimal transport problem}\label{sec:OT}

To fix notation, let us recall a few definitions and well-known facts about optimal transportation. Much more may be found for instance in the books \cite{Viltop,Santam}.

\medskip 

For a measure $\pi$ on $\R^d\times \R^d$, we denote its marginals by $\pi_1$ and $\pi_2$ i.e. $\pi_1(A)=\pi(A\times \R^d)$ and $\pi_2(A)=\pi(\R^d\times A)$. 
For $\mu$ and $\lambda$ two positive measures with compact support and equal mass, we define the Wasserstein distance between $\mu$ and $\lambda$ as
\begin{equation}\label{eq:W2}
 W^2(\mu,\lambda):=\min_{\pi_1=\mu,\pi_2=\lambda} \int |x-y|^2 d\pi.
\end{equation}
Under our hypothesis, an optimal coupling always exists. By the Knott-Smith optimality criterion (see \cite[Th. 2.12]{Viltop}), the coupling $\pi$ is optimal if and only if its support is cyclically monotone, i.e. there
exists a convex lower semi-continuous function $\psi$ such that $\spt \pi \subset \textrm{Graph}(\partial \psi)$.
Moreover, if $\mu$ does not give mass to Lebesgue negligible sets, by Brenier's Theorem (see \cite[Th. 2.12]{Viltop} again), $\pi$ is optimal if and only if $\pi=({\rm id}\times \nabla \psi)\# \mu$, where $\nabla \psi$ is the unique gradient of a convex function with $\nabla \psi\# \mu=\lambda$.
\medskip 

For $O\subset \R^d$, we denote $W_O(\mu,\lambda):=W(\mu\restr O,\lambda\restr O)$ and we recall that we let $\kappa$ be the generic constant such that
$W_O(\mu,\kappa)$ is well-defined\footnote{here we identified the measure $\kappa dx$ with $\kappa$}   (see \eqref{defkappa}). 
As in \cite{GO}, a central point for our analysis is the Eulerian version of optimal transportation, also known as the Benamou-Brenier formulation (see for instance \cite[Th. 8.1]{Viltop} or \cite[Ch. 8]{AGS}). It states that 
\begin{equation}\label{BBwass}
 W^2(\mu, \lambda)=\min_{(\rho,j)} \lt\{  \int_{\R^d}\int_0^1 \frac{1}{\rho}|j|^2 \ : \ \partial_t \rho+ \Div j=0, \ \rho_0=\mu \textrm{ and } \rho_1=\lambda\rt\},
\end{equation}
where the continuity equation and the boundary data are understood in the distributional sense, i.e.\ for every $\zeta\in C^1_c(\R^d\times [0,1])$,
\begin{equation}\label{conteqprel}
 \int_{\R^d}\int_0^1 \partial_t \zeta d\rho +\nabla \zeta \cdot dj=\int_{\R^d} \zeta_1 d\lambda- \int_{\R^d} \zeta_0 d\mu, 
\end{equation}
(recall that $\zeta_t=\zeta(\cdot,t)$) and where 
\[
 \int_{\R^d}\int_0^1 \frac{1}{\rho}|j|^2= \int_{\R^d}\int_0^1 |v|^2 d\rho
\]
if $j\ll \rho$ with $\frac{dj}{d\rho}=v$ and infinity otherwise (see \cite[Th. 2.34]{AFP}). Let us point out that in particular, the admissible measures for \eqref{BBwass} are allowed to
contain singular parts with respect to the Lebesgue measure. We also recall that for every $R>0$ we have the following dual definition of the Eulerian energy (see \cite[Prop. 5.18]{Santam})
\begin{equation}\label{dualBB}
 \frac12\int_{B_R}\int_0^1 \frac{1}{\rho}|j|^2=\sup_{\xi\in C^0_c(B_R\times[0,1])^d}\int_{B_R}\int_0^1 \xi \cdot dj-\frac{|\xi|^2}{2} d\rho.
\end{equation}
We also note that if $\pi$ is optimal for $W^2(\mu,\lambda)$, then  defining for $t\in [0,1]$,
\begin{equation}\label{defrhoj}
 \int \zeta d\rho_t=\int \zeta((1-t)x+ty) d\pi \quad \textrm{and} \quad \int \xi\cdot dj_t=\int\xi((1-t)x+ty) \cdot(y-x) d\pi,  
\end{equation}
the couple $(\rho,j)=(\rho_t dt, j_t dt)$ is optimal for \eqref{BBwass}. Using \eqref{defrhoj} in combination with \eqref{dualBB} and \eqref{BBwass}, we obtain\footnote{Indeed, from \eqref{defrhoj} and \eqref{dualBB},
together with $\xi\cdot (y-x)-\frac{1}{2}|\xi|^2\le \frac{1}{2}|x-y|^2$ for every $\xi\in \R^d$, we have that $\int_{B_R}\int_0^1 \frac{1}{\rho}|j|^2\le\int I(|(1-t)x+ty|<R)|x-y|^2 d \pi$
and the other inequality then follows  from \eqref{BBwass} and another application of \eqref{dualBB} with $B_R$ replaced by $B_R^c$ to rule out strict inequality. }  the localized counterpart of \eqref{BBwass} 
\begin{equation}\label{je15}
\int_{B_R}\int_0^1 \frac{1}{\rho}|j|^2=\int \int_0^1 I(|(1-t)x+ty|<R)|x-y|^2 d \pi. 
\end{equation}

We notice that since any admissible $(\rho,j)$ for \eqref{BBwass} is divergence free (in $(t,x)$),  for every $R>0$ it admits internal (and external) traces on $\partial (B_R\times(0,1))$ (see  \cite{ChenFried}) i.e.
there exists a measure $f$ on $\partial B_R\times (0,1)$ such that 
\begin{equation}\label{deftraces}
 \int_{B_R}\int_0^1 \partial_t \zeta d\rho +\nabla \zeta \cdot dj=\int_{B_R} \zeta_1 d\lambda- \int_{B_R} \zeta_0 d\mu +\int_{\partial B_R}\int_0^1 \zeta df.
\end{equation}
Moreover, recalling that  $\bar j=\int_0^1 dj_t$ (see \eqref{defbar}), we have 
\begin{equation}\label{eq:divjbar}
\nabla\cdot \bar j=\mu-\lambda 
\end{equation}
 and also $\bar j$ has internal\footnote{and of course the internal trace coincides with $\bar f=\int_0^1 f$.}
(and external) traces on $\partial B_R$ which agree for every $R$ such that  $|\bar j|(\partial B_R)=\mu(\partial B_R)=\lambda(\partial B_R)=0$.
\medskip 

For $R>0$, $\mu$ and $\lambda$ positive measures on $\partial B_R$ with equal mass, we denote (with a slight abuse of notation)
\[
 W_{\partial B_R}^2(\mu,\lambda):=\min_{\pi_1=\mu, \pi_2=\lambda}\int_{\partial B_R\times \partial B_R} d_{\partial B_R}^2(x,y) d\pi,
\]
where $d_{\partial B_R}$ denotes the geodesic distance on $\partial B_R$. Notice that $W_{\partial B_R}^2(\mu,\lambda)$ is comparable with the Euclidean Wasserstein distance $W^2(\mu,\lambda)$ and that the analog
of the Benamou-Brenier formulation \eqref{BBwass} holds (see \cite[Th. 13.8]{VilOandN})
\begin{equation}\label{BBwassboundary}
 W^2_{\partial B_R}(\mu,\lambda)=\min_{(\rho,j)}\lt\{\int_{\partial B_R}\int_0^1 \frac{1}{\rho}|j|^2  :  \partial_t \rho+\nabla_{\tang}\cdot j=0,  \rho_0=\mu \textrm{ and } \rho_1=\lambda\rt\}
\end{equation}
The following lemma will be useful in the construction part of the proof of the harmonic approximation result.
\begin{lemma}\label{lem:consbound}
 Let $f$ and $g$ be two positive measures on $\partial B_R\times [0,1]$ such that $\spt f\subset \partial B_R\times[0,t_-]$, $\spt g \subset \partial B_R \times[t_+,1]$ with $\tau:=t_+-t_->0$ and $\bar f(\partial B_R)=\bar g(\partial B_R)$
 (where $\bar f:=\int_0^1 f$, see \eqref{defbar}). Then,  
 there exists a density-flux pair $(\rho,j)$ on $\partial B_R\times [0,1]$  such that
\begin{equation}\label{ma24}
 \int_{\partial B_R}\int_0^1 \partial_t \zeta d\rho+\nabla_{tan} \,\zeta\cdot dj=\int_{\partial B_R}\int_0^1 \zeta d(g-f)
\end{equation}
and 
\begin{equation}\label{ma25}
 \int_{\partial B_R}\int_0^1 \frac{1}{\rho}|j|^2= \frac{1}{\tau}\, W_{\partial B_R}^2(\bar f,\bar g).
\end{equation}

\end{lemma}
\begin{proof}
 We first connect in $(0,t_-)$, the density $f$ to $\bar f$ (at no cost), then in $(t_-,t_+)$ we connect $\bar f$ to $\bar g$ (at cost $\frac{1}{\tau}\, W_{\partial B_R}^2(\bar f,\bar g)$) and finally connect in $(t_+,1]$, $\bar g$ to $g$ (again at no cost).\\
More precisely, for $t\in [0,a]$, let 
 \[
  \int \zeta d\rho_t=\int_{\partial B_R} \int_0^t \zeta(x) df(x,s) \qquad \textrm{and } \qquad j_t:=0
 \]
 so that 
 \begin{equation}\label{ma26}
  \int_{\partial B_R}\int_0^{t_-} \partial_t \zeta d\rho+\nabla_{tan} \,\zeta\cdot dj= \int_{\partial B_R} \zeta_{t_-} d\bar f -\int_{\partial B_R}\int_0^1 \zeta df.
 \end{equation}

 Let then $(\tilde \rho,\tilde j)$ be a minimizer of the Benamou-Brenier version \eqref{BBwassboundary} of $W_{\partial B_R}(\bar f,\bar g)$ and consider for $t\in(t_-,t_+)$ the rescaling
 \[
  (\rho_t,j_t):=\lt(\tilde{\rho}_{\frac{t-t_-}{\tau}},  \frac{1}{\tau} \tilde{j}_{\frac{t-t_-}{\tau}}\rt),
 \]
so that 
\begin{equation}\label{ma27}
 \int_{\partial {B_R}} \int_{t_-}^{t_+} \partial_t \zeta d\rho+\nabla_{tan} \,\zeta\cdot dj=\int_{\partial B_R} \zeta_{t_+} \, d \bar g -\int_{\partial B_R} \zeta_{t_-}\, d\bar f
\end{equation}
and 
\[
 \int_{\partial B_R}\int_{t_-}^{t_+}\frac{1}{\rho}|j|^2= \frac{1}{\tau} W_{\partial B_R}^2(\bar f,\bar g).
\]
Finally for $t\in (t_+,1)$, we let 
\[
  \int \zeta d\rho_t= \int_{\partial B_R} \int_t^1 \zeta(x) dg(x,s) \qquad \textrm{and } \qquad j_t:=0
 \]
so that 
\begin{equation}\label{ma28}
 \int_{\partial B_R}\int_{t_+}^1 \partial_t \zeta d\rho+\nabla_{tan} \,\zeta\cdot dj= \int_{\partial B_R}\int_0^1 \zeta dg -\int_{\partial B_R} \zeta_{t_+} d\bar g.
\end{equation}
Putting together \eqref{ma26}, \eqref{ma27} and \eqref{ma28}, we get \eqref{ma24}.
\end{proof}

 \subsubsection{A few technical results}
 We collect in this section a couple of technical results which will be useful later on.
\begin{lemma}\label{2.7}
 For every $R>0$ and every  measure $\mu$ on $B_R$,
 \begin{equation}\label{wg84}
  W_{B_R}(\mu,\kappa)\les W_{B_R}(\mu+\kappa dx, 2\kappa).
 \end{equation}

\end{lemma}
\begin{proof}
 By the triangle inequality and the monotonicity of $W_{B_R}$,
 \begin{align*}
  W_{B_R}(\mu,\kappa)&\le W_{B_R} \lt(\mu, \frac{1}{2}(\mu+\kappa dx)\rt)+W_{B_R} \lt(\frac{1}{2}(\mu+\kappa dx),\kappa\rt)\\
  &\le W_{B_R} \lt(\frac{1}{2}\mu, \frac{1}{2}\kappa\rt)+W_{B_R} \lt(\frac{1}{2}(\mu+\kappa dx),\kappa\rt)\\
  &= \frac{1}{\sqrt{2}} \lt(W_{B_R}(\mu,\kappa)+W_{B_R}(\mu+\kappa dx, 2\kappa)\rt),
 \end{align*}
which concludes the proof of \eqref{wg84}.
\end{proof}

\begin{lemma}
 For every $R>0$, $\zeta\in C^\infty(B_R)$, and  $\mu$ supported in $B_R$,
 \begin{equation}\label{wg80}
  \lt|\int_{B_R} \zeta (d\mu-\kappa dx)\rt|\le \lt(\kappa\int_{B_R} |\nabla \zeta|^2\rt)^{\frac{1}{2}}W_{B_R}(\mu,\kappa)+\frac{1}{2}  \sup_{B_R} |\nabla^2\zeta| W_{B_R}^2(\mu,\kappa).
 \end{equation}

\end{lemma}
\begin{proof}
 Integrating the Taylor expansion $|\zeta(x)-\zeta(y)-\nabla \zeta(y)\cdot(x-y)|\le \frac{1}{2}\sup_{B_R} |\nabla^2\zeta| |x-y|^2$ against an optimal transport plan between $\mu$ and $\kappa dx$ gives \eqref{wg80} after an application of Cauchy-Schwarz.
\end{proof}
\subsubsection{An $L^\infty$ bound on the displacement}
We now prove an $L^\infty$ bound on the transport. This is one of the fundamental ingredients for the proof of Theorem \ref{theo:harmonicLagintro}. It is similar to \cite[Lem. 3.1]{GO} with the technical difference that we consider here couplings between arbitrary measures.
 \begin{lemma}\label{Linf}
  Let $\pi$ be a coupling between two measures $\mu$ and $\lambda$. Assume that $\spt \pi$ is monotone and that for some $R>0$, $\frac{1}{R^2}E+\frac{1}{R^2}D\ll1$ where 
 \[
   E:= \frac{1}{|B_{6R}|}\int_{(B_{6R}\times \R^d)\cup(\R^d\times B_{6R})} |x-y|^2 d\pi \]
   and
   \[
   D:= \frac{1}{|B_{6R}|}W^2_{B_{6R}}(\mu,\kappa_{\mu})+ \frac{R^2}{\kappa_{\mu}}(\kappa_\mu-1)^2+\frac{1}{|B_{6R}|}W^2_{B_{6R}}(\lambda,\kappa_{\lambda})+ \frac{R^2}{\kappa_{\lambda}}(\kappa_\lambda-1)^2.
  \]
Then, for every $(x,y)\in \spt \pi \cap \lt((B_{5R}\times \R^d)\cup (\R^d\times B_{5R})\rt)$ 
\begin{equation}\label{wg31}
 |x-y|\les R \lt(\frac{1}{R^2}E+\frac{1}{R^2}D\rt)^{\frac{1}{d+2}}=: R M.
\end{equation}
As a consequence, for $(x,y)\in \spt \pi$ and $t\in[0,1]$,
\begin{equation}\label{ma38}
 (1-t)x+ty\in B_{2R} \qquad \Longrightarrow \qquad  x\in B_{3R} \ \textrm{ or } \ y\in B_{3R}
\end{equation}
and 
\begin{equation}\label{ma39}
 x\in B_R \ \textrm{ or } \ y\in B_R \qquad  \Longrightarrow \qquad  (1-t)x+ty\in B_{2R}.
\end{equation}
Moreover,
\begin{equation}\label{ma88}
 \frac{1}{|B_{5R}|}\int_{(B_{5R}\times \R^d)\cup(\R^d\times B_{5R})} d\pi\les 1.
\end{equation}

 \end{lemma}

 \begin{proof}
 By scaling we may assume that $R=1$.
  We first prove \eqref{wg31}. Fix $(x,y)\in \spt \pi \cap \lt((B_5\times \R^d)\cup (\R^d\times B_5)\rt)$. Without loss of generality, we may assume that $(x,y)\in B_5\times \R^d$. 
  By monotonicity, we then have for every $(x',y')\in \spt \pi$ that  $(y-y')\cdot(x-x')\ge 0$. Writing $y-y'=(y-x)+(x-x')+(x'-y')$ this may be rewritten as
  \begin{align*}
  (x-y)\cdot  (x-x')&\le |x-x'|^2+(x-x')\cdot(x'-y')\\
   &\le \frac{3}{2} |x-x'|^2 +\frac{1}{2}|x'-y'|^2.
  \end{align*}
Integrating this inequality against the measure $\eta(x') d\pi(x',y')$ for some smooth function $\eta\in[0,1]$ supported in $B_6$ we get 
\begin{multline}\label{wg30}
 (x-y)\cdot \int (x-x') \eta(x') d\mu(x')\les \int |x-x'|^2 \eta(x') d\mu(x') +\int_{B_6\times \R^d} |x'-y'|^2 d\pi\\
 \les \int |x-x'|^2 \eta(x') d\mu(x') +E.
\end{multline}
We now fix an arbitrary unit vector $e$ and a radius $0<r\ll 1$ and assume that $\eta$ is supported in $B_r(x-re)$ (which is contained in $B_6$ since $x\in B_5$) with $r\sup |\nabla \eta|+r^2 |\nabla^2 \eta|\les 1$. 
We want to replace the measure $\mu$ by the Lebesgue measure and for this use \eqref{wg80} and $\kappa_\mu\sim 1$ to get
\begin{equation}\label{ven02}
 \lt|\int (x-x') \eta(x') d\mu(x') -\kappa_\mu \int (x-x') \eta(x') dx'\rt|\les (r^d D)^{\frac{1}{2}} + \frac{1}{r}D
\end{equation}
and 
\begin{equation}\label{ven03}
 \lt| \int |x-x'|^2 \eta(x') d\mu(x')- \kappa_\mu \int |x-x'|^2 \eta(x') dx'\rt|\les (r^{d+2} D)^{\frac{1}{2}}+D.
\end{equation}
Appropriately normalizing $\eta$ and choosing it to be radially symmetric around $x-re$ we have
\[
 \int (x-x') \eta(x') dx'=r^{d+1} e \qquad \textrm{and} \qquad \int |x-x'|^2 \eta(x')dx'\les r^{d+2}.
\]
Using again that $\kappa_\mu\sim 1$, \eqref{ven02} and \eqref{ven03}, \eqref{wg30} becomes
\[
 r^{d+1}(x-y)\cdot e\les \lt((r^d D)^{\frac{1}{2}}+\frac{1}{r} D\rt)|x-y|+ r^{d+2}+(r^{d+2} D)^{\frac{1}{2}}+D+E.
\]
Since $e$ is arbitrary, using Young's inequality i.e. $(r^d D)^{\frac{1}{2}}\les \eps r^{d+1}+\frac{1}{r\eps} D$, this yields
\[
 r^{d+1}|x-y|\les \frac{1}{r} D|x-y|+ r^{d+2}+D+E.
\]
Choosing $r$ to be a large multiple of $(E+D)^{\frac{1}{d+2}}$, we conclude the proof of \eqref{wg31}.

\medskip
To show \eqref{ma88} we estimate
\[
 \frac{1}{|B_{5}|}\int_{(B_{5}\times \R^d)\cup(\R^d\times B_{5})} d\pi\le\frac{1}{|B_{5}|}\int_{B_{6}\times \R^d} d\pi + \frac{1}{|B_{5}|}\int_{\R^d\times B_{6}} d\pi =\frac{\mu(B_{6})+\lambda(B_6)}{|B_{5}|}\les \kappa_\mu + \kappa_\lambda
\]
so that \eqref{ma88} follows from $\kappa_\mu+\kappa_\lambda\sim1$. Since \eqref{ma39} is a direct consequence of \eqref{wg31}, we are left with the proof of \eqref{ma38}. Since $E+D\ll1$, there exists $(x',y')\in \spt \pi$ with $|x'|+|y'|\ll1$. By monotonicity of $\spt \pi$, we then have
\begin{align*}
 4+o(1)&\ge |(1-t)x+ty -(1-t)x'-ty'|^2\\
 &= (1-t)^2|x-x'|^2+2t(1-t) (x-x')\cdot(y-y')+ t^2|y-y'|^2\\
 &\ge (1-t)^2|x-x'|^2+ t^2|y-y'|^2\\
 &\ge \frac{1}{2} \min(|x-x'|^2,|y-y'|^2)\ge \frac{1}{2} \min(|x|^2,|y|^2)- o(1),
\end{align*}
which concludes the proof.
 \end{proof}

\subsubsection{A localization result}
Our next lemma shows that the quantity $W_{B_5}^2(\mu,\kappa)+\frac{1}{\kappa}(\kappa-1)^2$,
which captures the (squared) distance of a measure $\mu$ to the Lebesgue measure,
is well behaved under restricting $\mu$ to balls $B_R$ with $R\sim 1$. However,
this seems true only on average in $R$.
Since in the application to the matching problem we will need the lemma for cubes instead of balls, we will state it in a general form.
\begin{lemma}\label{Lres} We have for any positive  measure $\mu$ 
\begin{align*}
\int_3^4\lt(W_{B_R}^2(\mu,\kappa_R)+\frac{1}{\kappa_R}(\kappa_R-1)^2\rt)dR\lesssim D,
\end{align*}
provided 
\begin{align}\label{wg86}
D:=W_{O}^2(\mu,\kappa)+\frac{1}{\kappa}(\kappa-1)^2\ll 1,
\end{align}
for some set $O\supset B_5$.
\end{lemma}
Before starting the proof, let us point out that the proofs in Section \ref{sec:radius} use similar ideas and techniques in a slightly more complicated setup. Hence, inside the 
proof we use similar notation.  
\begin{proof}

Let $\pi$ denote the optimal transference plan for $W_{O}^2(\mu,\kappa)$.
It is convenient to introduce the corresponding (non-normalized) measure $\mathbb{P}\ge 0$ on straight
trajectories $[0,1]\ni t\mapsto X(t)$: It is the push-forward of $\pi$
under $X(t)=ty+(1-t)x$.
For given $R\in(0,5)$ consider the set of exiting and entering trajectories, respectively (see Figure \ref{fig:Wmukappa}),
\begin{figure}\begin{center}
 \resizebox{7.cm}{!}{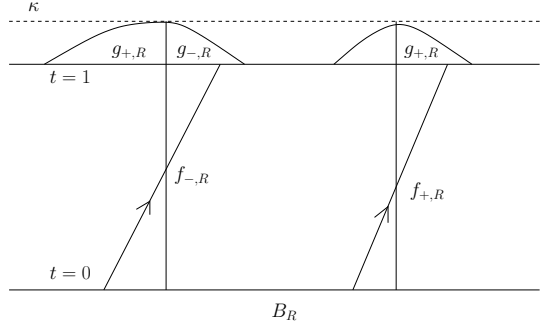} 
   \caption{The definition of $g_{\pm,R}$.} \label{fig:Wmukappa}
 \end{center}
 \end{figure}
\begin{align}\label{wg90}
\begin{array}{cc}\Omega_{+,R}&:=\{\,X\,|\,|X(0)|<R,\;|X(1)|\ge R\,\},\\
\Omega_{-,R}&:=\{\,X\,|\,|X(0)|\ge R,\;|X(1)|< R\,\}.
\end{array}
\end{align}
We define the measures $0\le g_{+,R}\le \kappa$ on $B_R^c$ and $0\le g_{-,R}\le \kappa$ on $B_R$ which record where
exiting and entering trajectories end up:
\begin{align}\label{wg96}
\int\zeta dg_{\pm,R}:=\int_{\Omega_{\pm,R}}\zeta(X(1))d\P.
\end{align}
We claim that these measures are, on average in $R$, concentrated close to $\partial B_R$:
\begin{equation}\label{ma1}
\int_{3}^4\lt(\int(|x|-R)d g_{+,R} +\int(R-|x|) d g_{-,R}\rt)dR\le W_{O}^2(\mu,\kappa).
\end{equation}
Indeed, by definition (\ref{wg96}) of $g_{\pm,R}$, we have
\begin{align*}
\lefteqn{\int(|x|-R)d g_{+,R}+\int(R-|x|) d g_{-,R}}\nonumber\\
&=\int_{|X(0)|<R\le|X(1)|}(|X(1)|-R)d\P+\int_{|X(1)|<R\le|X(0)|}(R-|X(1)|)d\P\nonumber\\
&\le\int_{\min\{|X(0)|,|X(1)|\}<R\le\max\{|X(0)|,|X(1)|\}}||X(1)|-|X(0)||d\P,
\end{align*}
so that
\begin{align*}
\lefteqn{\int_{3}^4\lt(\int(|x|-R)d g_{+,R} +\int(R-|x|) d g_{-,R}\rt)dR}\nonumber\\
&\le\int|\max\{|X(0)|,|X(1)|\}-\min\{|X(0)|,|X(1)|\}||X(1)|-|X(0)||d\P\nonumber\\
&\le\int|X(1)-X(0)|^2 d\P = W_O^2(\mu,\kappa).
\end{align*}
Momentarily fixing $R$ and setting
\begin{align}\label{wg91}
\kappa_{+,R}:=\frac{1}{|B_R|}g_{+,R}(\mathbb{R}^d),\quad \kappa_{-,R}:=\frac{1}{|B_R|} g_{-,R}(\mathbb{R}^d),
\end{align}
we claim that for every $R\in (3,4)$ with $\mu(\partial B_R)=0$, 
\begin{equation}\label{wg88}
W_{B_R}^2(\mu,\kappa_R)\lesssim W_{O}^2(\mu,\kappa)+\int(|x|-R)d g_{+,R}+\int(R-|x|) d g_{-,R},
\end{equation}
where\footnote{Notice that by \eqref{wg90} (while the test function $I(|x|<R)$ is not continuous, a simple approximation argument shows that this is actually not an issue), \[\mu(B_R)=\P(|X(0)|<R)=\P(|X(1)|<R)-\P(\Omega_{-,R})+\P(\Omega_{+,R})=|B_R|(\kappa-\kappa_{-,R}+\kappa_{+,R}).\]
} 
\begin{align}\label{wg89}
\kappa_R:=\kappa-\kappa_{-,R}+\kappa_{+,R}.
\end{align}
We first notice that thanks to the $L^\infty$ bound \eqref{wg31}, and the definitions \eqref{wg90}, \eqref{wg96} and \eqref{wg91}, for any $R\in(3,4)$, 
\begin{equation}\label{wg92}
 \kappa_{\pm,R}=\frac{1}{|B_R|}\P(\Omega_{\pm,R})\stackrel{\eqref{wg31}}{\les} \P( ||X(1)|-R|\les D^{\frac{1}{d+2}}) \stackrel{\eqref{wg86}}{\ll} \kappa.
\end{equation}

We now turn to (\ref{wg88}) and to this purpose introduce
the measure $f_{+,R}\ge 0$ on $ \partial B_R\times [0,1]$ that keeps track
of where and when trajectories exit (see Figure \ref{fig:Wmukappa})	
\begin{align*}
\int\zeta df_{+,R}:=\int_{\Omega_{+,R}}\zeta(X(t_+),t_+)d\P,
\end{align*}
here $t_+$ denotes the unique $t_+\in(0,1]$ such that $|X(t_+)|=R$.
We also introduce the measures $\hat g_{\pm,R}\ge 0$ and $\bar f_{+,R}\ge 0$ on $\partial B_R$ by projection (note that $\hat g_{\pm,R}$ has a density)
\begin{align}\label{wg95}
\int\zeta d\hat g_{\pm,R}:=\int\zeta\lt(R\frac{x}{|x|}\rt) dg_{\pm,R}(x),\quad
\int\zeta d\bar f_{+,R}:=\int\zeta(x)df_{+,R}(x,t).
\end{align}
We now use the triangle inequality for $W_{B_R}$ and to split the estimate (\ref{wg88}) into
\begin{align}\label{wg81}
W_{ B_R}^2(\mu,\kappa_R)\les \ &W^2_{\bar B_R}(\mu,\kappa dx- g_{-,R}+\bar f_{+,R})\nonumber\\
&+W^2_{\bar B_R}(\kappa dx- g_{-,R}+\bar f_{+,R},(\kappa-\kappa_{-,R})dx+\hat g_{+,R})\nonumber\\
&+W^2_{\bar B_R}((\kappa-\kappa_{-,R})dx+\hat g_{+,R},\kappa_R),
\end{align}
and then the sub-additivity of $W^2_{\bar B_R}$ 
to split the middle term further into
\begin{multline}\label{wg82}
W^2_{\bar B_R}(\kappa dx- g_{-,R}+\bar f_{+,R},(\kappa-\kappa_{-,R})dx+\hat g_{+,R})
\\
\le W^2_{ B_R}(\kappa dx- g_{-,R},\kappa-\kappa_{-,R})
+W^2_{\partial B_R}(\bar f_{+,R},\hat g_{+,R}).
\end{multline}

\medskip

We start with the two easier terms.
We claim that the first right-hand side term in (\ref{wg81}) is estimated as follows
\begin{align*}
W^2_{\bar B_R}(\mu,\kappa dx- g_{-,R}+\bar f_{+,R})\le W_{O}^2(\mu, \kappa).
\end{align*}
Indeed, an admissible transference plan $\pi_R$ is given by restricting the plan $\pi$
\[
\int\zeta d\pi_R=\int_{|X(0)|<R,|X(1)|<R}\zeta(X(0),X(1))d\P+\int_{\Omega_{+,R}}\zeta(X(0),X(t_+))d\P.
\]
Likewise, the second right-hand side term of (\ref{wg82}) is estimated in the same way:
\begin{align*}
W^2_{\partial B_R}(\bar f_{+,R},\hat g_{+,R})\le W_{O}^2(\mu, \kappa).
\end{align*}
Again, an admissible transference plan $\pi_{+,R}$ is given by
\begin{align*}
\int\zeta d\pi_{+,R}
=\int_{\Omega_{+,R}}\zeta\lt(X(t_+),R\frac{X(1)}{|X(1)|}\rt)d\P.
\end{align*}

\medskip 

We now turn to the remaining two more subtle terms.
For the last right-hand side term in (\ref{wg81}) we claim that 
\begin{equation}\label{wg93}
W_{\bar B_R}^2((\kappa-\kappa_{-,R})dx+\hat g_{+,R}\,,\,\kappa_R)\lesssim\int(|x|-R)d g_{+,R}.
\end{equation}
Indeed, recalling the definition \eqref{wg89} of $\kappa_R$, an admissible density-flux pair $(\rho,j)$ for the Benamou-Brenier formulation \eqref{BBwass} 
is given by
\begin{align*}
\rho_t:=(\kappa-\kappa_{-,R}+t\kappa_{+,R})dx+(1-t)\hat g_{+,R},\qquad j_t:=
\left\{\begin{array}{cl}\nabla\vp&\mbox{in}\;B_R\\0&\mbox{else}\end{array}\right\},
\end{align*}
where $\vp$ is a  solution of
\[
\Delta\vp=-\kappa_{+,R}\ \textrm{ in } B_R\qquad \textrm{ and } \qquad \nu\cdot\nabla\vp=-\hat g_{+,R}\ \textrm{ on }\partial B_R,
\]
so that we have
\begin{align*}
W_{\bar B_R}^2((\kappa-\kappa_{-,R})dx+\hat g_{+,R}\,,\,\kappa_R)\le\frac{1}{\kappa-\kappa_{-,R}}\int_{B_R}|\nabla\vp|^2.
\end{align*}
Combining this with the energy estimate \eqref{estim:elliptic} (applied to $f=\hat g_{+,R}$ and $g=0$), \eqref{wg31} to be able to apply \eqref{estimz} and \eqref{wg92}, this concludes the proof of \eqref{wg93}.

\medskip

Likewise, the first right-hand side term in (\ref{wg82}) is estimated as
\begin{align*}
W^2_{ B_R}(\kappa dx- g_{-,R},\kappa-\kappa_{-,R})\lesssim\int(R-|x|) d g_{-,R}.
\end{align*}
Indeed, by (\ref{wg84}) it is a consequence of 
\[
W^2_{ B_R}((2\kappa-\kappa_{-,R})dx- g_{-,R},2(\kappa-\kappa_{-,R}))\lesssim\int(R-|x|) d g_{-,R}.
\]
In turn, this follows arguing as in \eqref{wg93} and using as admissible density-flux pair $(\rho,j)$ 
\[
\rho_t:=2\kappa-\kappa_{-,R}-t\kappa_{-,R}-(1-t) g_{-,R},\qquad j_t:=
\left\{\begin{array}{cl}\nabla\vp&\mbox{in}\;B_R\\0&\mbox{else}\end{array}\right\},
\]
where $\vp$ is a solution of
\[
\Delta\vp=\kappa_{-,R}- g_{-,R}\ \textrm{ in } B_R \qquad \textrm{ and } \qquad \nu\cdot\nabla\vp=0\ \textrm{ on } \partial B_R.
\]
This concludes the proof of \eqref{wg88}.

\medskip

Finally, by the definition \eqref{wg89} of $\kappa_R$ and \eqref{wg92}, 
\[
 \frac{1}{\kappa_R}(1-\kappa_R)^2\lesssim \frac{1}{\kappa} (1-\kappa)^2+ \kappa_{-,R}^2+\kappa_{+,R}^2.
\]
Combining this with $\kappa_{\pm,R}= \frac{1}{|B_R|}g_{\pm,R}(\R^d)$ (cf. \eqref{wg91}) and \eqref{estimz}, we obtain 
\[
 \frac{1}{\kappa_R}(1-\kappa_R)^2\lesssim \frac{1}{\kappa}(1-\kappa)^2+\int(|x|-R)dg_{+,R}+\int(R-|x|) d g_{-,R},
\]
which together with \eqref{wg88} and \eqref{ma1} concludes the proof of the lemma.
 \end{proof}
 \begin{remark}\label{rem:DataLebesgue}
  Let us point out that with a very similar proof, it may be shown that if $\pi$ is the optimal coupling between a measure $\mu$ and the Lebesgue measure for which $E+D\ll1$ where
  \[
E:=\frac{1}{|B_5|}\int_{(B_5\times \R^d)\cup(\R^d\times B_5)} |x-y|^2 d\pi \qquad \textrm{and} \qquad D:= \frac{1}{|B_5|}W_{B_5}^2(\mu,\kappa),   
  \]
then
\[
 \int_3^4 \lt(W_{B_R}^2(\mu,\kappa_R)+\frac{1}{\kappa_R}(\kappa_R-1)^2\rt) dR\les E+D.
\]
In particular, this means that when the target measure is the Lebesgue measure, up to choosing a good radius, we have a good control on $\frac{1}{\kappa}(\kappa-1)^2)$.
 \end{remark}

 \section{The harmonic approximation result}
 In this section, which is the core of the paper, we prove Theorem \ref{theo:harmonicLagintro}. As already explained, by scaling we will prove it on the unit scale. For $\pi$ an optimal coupling between two measures $\mu$ and $\lambda$, we recall that we set
 \begin{align*}
  E&=\frac{1}{|B_6|}\int_{(B_6\times \R^d)\cup(\R^d\times B_6)} |x-y|^2 d\pi,\\
  D&= \frac{1}{|B_6|} W_{B_6}^2(\mu,\kappa_\mu)+\frac{1}{\kappa_\mu}(\kappa_\mu-1)^2+\frac{1}{|B_6|} W_{B_6}^2(\lambda,\kappa_\lambda)+\frac{1}{\kappa_\lambda}(\kappa_\lambda-1)^2
 \end{align*}
and work under the hypothesis that $E+D\ll1$. 
As explained in the introduction, the strategy of the proof is to argue at the Eulerian level and first choose  a good radius $R\in (3,4)$
for which the solution $\Phi$ (respectively $\phi$) of the Poisson equation in $B_R$ 
with Neumann boundary condition given by the Eulerian flux (respectively a regularized version of) $\bar j$ on $\partial B_R$ is 'well-behaved'. This is done in Section \ref{sec:radius}.
Then, we prove in Section \ref{sec:approxorth}
an approximate orthogonality result, which combined with the construction of a competitor  with controlled energy in Section \ref{sec:construct}
and the minimality of the density-flux pair $(\rho,j)$ yields Theorem \ref{theo:harmonicLagintro}.

 \subsection{Choice of a good radius}\label{sec:radius}
 In this section we define the main densities and fluxes which will appear in the proof of Theorem \ref{theo:harmonicLagintro} and prove that we may choose a good radius for which they are well-behaved. 
 
 \subsubsection{Unregularized fluxes}
 Using the Gluing Lemma (see \cite[Lem. 7.6]{Viltop}) and disintegration, we couple the  optimal transport between $\mu$ and $\lambda$ on the time interval $t\in(0,1)$
to the optimal transport between $\kappa_\mu$ and $\mu$ (restricted to $B_6\times B_6$) in the interval $(-1,0)$ and
to the optimal transport between $\lambda$ and $\kappa_\lambda$ (restricted to $B_6\times B_6$) in the interval $(1,2)$. 
In view of Lemma \ref{Linf}, there exists a  (not normalized) measure $\mathbb{P}$ on
continuous trajectories $[-1,2]\ni t\mapsto X(t)\in\mathbb{R}^d$ that couples $\kappa_\mu dx$, $\mu$,
$\lambda$, and $\kappa_\lambda dx$ in the sense that for any $\zeta\in C^0_c(B_5)$ (see Figure \ref{fig:gluing}) 

\begin{align}\label{io17}
\begin{array}{rcl}
\int\zeta(X(-1))d\P&=&\int\zeta \kappa_\mu dx,\\[1ex]
\int\zeta(X( 0))d\P&=&\int\zeta d\mu,\\[1ex]
\int\zeta(X( 1))d\P&=&\int\zeta d\lambda,\\[1ex]
\int\zeta(X( 2))d\P&=&\int\zeta \kappa_\lambda dx.
\end{array}
\end{align}
Moreover, the trajectories are piecewise affine on
$(-1,0)\cup(0,1)\cup(1,2)$ and optimal in the sense that 
\begin{align}\label{io16}
\begin{array}{llc}
\int|X(0)-X(-1)|^2d\P&\le W_{B_6}^2(\mu,\kappa_\mu)&\le  |B_6| D,\\[1ex]
\int|X(1)-X(0)|^2d\P&
&\le  |B_6| E,\\[1ex]
\int|X(2)-X(1)|^2d\P&\le W_{B_6}^2(\lambda,\kappa_\lambda)&\le |B_6| D.
\end{array}
\end{align}

\medskip

Finally, by \eqref{defrhoj} and \eqref{io17}, the density/flux pair $(\rho=\rho_tdt,j=j_tdt)$ can be recovered from $\mathbb{P}$ on $B_5\times(0,1)$ 
in the sense that for $t\in[0,1]$, $\zeta\in C^0_c(B_5)$, and $\xi\in C^0_c(B_5)^d$
\begin{equation}\label{io18}
\begin{array}{ll}
\ds \int \zeta(X(t))d\P&=\ds \int\zeta d\rho_t,\\
 \ds \int \xi(X(t))\cdot\dot X(t)d\P&=\ds\int\xi\cdot dj_t.
\end{array}
\end{equation}
We then recall the definitions of  the time integrated fluxes (see \eqref{defbar})
\begin{equation}\label{ma15}
\bar \rho=\int_0^1 \rho_t dt  \qquad \textrm{and} \qquad \bar j=\int_0^1 j_t dt
\end{equation}
Let us notice that using    \eqref{ma38} (in combination with $E+D\ll1$) on the one hand and \eqref{io16} together with \eqref{io18} on the other hand, we have  
\begin{equation}\label{ma42}
 \int_{B_2} d\bar \rho \les 1 \qquad \textrm{ and } \qquad \int_{B_5}\int_0^1 \frac{1}{\rho} |j|^2\le E.
\end{equation}

In the upcoming Lemma \ref{lemcompatib}, we  derive that for any radius $R$, the (inner) flux boundary data $f_R$ of $j$ on $\partial B_R\times[0,1]$
admit a representation in terms of $\mathbb{P}$. Indeed, consider the set $\Omega_{+,R}$
of trajectories that exit $B_R$ and the set $\Omega_{-,R}$ of those that enter:
\begin{equation}\label{ma2}
 \begin{array}{ll}
  \Omega_{+,R}:=\{\,X\,|\,|X(1)|\ge R\;\mbox{and}\;|X(t)|<R\;\mbox{for some}\;t\in[0,1]\,\},\\
\Omega_{-,R}:=\{\,X\,|\,|X(0)|\ge R\;\mbox{and}\;|X(t)|<R\;\mbox{for some}\;t\in[0,1]\,\}.
 \end{array}
\end{equation}

We notice that these sets are not necessarily disjoint. We also define the set of trajectories $\Omega_R$ which enter, exit or are entirely contained in $B_R$:
\begin{equation}\label{OmR}
 \Omega_R:= \{\,X\,|\,|X(t)|<R\;\mbox{for some}\;t\in[0,1]\,\}
\end{equation}
and point out  that by the $L^\infty$ bound \eqref{wg31}, if $E+D\ll1$ then for  $R\le 5$,  $X\in \Omega_R$ and  $(s,t)\in [-1,2]$,
\begin{equation}\label{ma3}
 |X(t)-X(s)|\les M,
\end{equation}
where we recall that $M= (E+D)^{\frac{1}{d+2}}$.
We associate exiting and entering times to the trajectories that spend time
in $B_R$:
\begin{align}\label{io26}
\left.\begin{array}{c@{\hspace{0.2ex}}c@{\hspace{0.2ex}}ccc}
t_+&:=&\sup\{t\in[0,1]||X(t)|<R\}&\in&(0,1]\\[1ex]
t_-&:=&\inf\{t\in[0,1]||X(t)|<R\}&\in&[0,1)
\end{array}\right\}\quad \mbox{for}\;X\in\Omega_R;
\end{align}
note that $t_+=1$ if the trajectory ends in $\bar B_R$,
and $t_-=0$ if the trajectory starts in $\bar B_R$. 
We introduce the non-negative measures $f_{\pm,R}$ on $\partial B_R\times[0,1]$ 
that keep track of where in space-time trajectories exit and enter:
\begin{align}\label{io03}
\displaystyle\int\zeta df_{\pm,R}=\displaystyle\int_{\Omega_{\pm,R}}
\zeta(X(t_\pm),t_\pm)d\mathbb{P}.
\end{align}
We are now in the position to introduce signed measure on $\partial B_R\times[0,1]$:
\begin{align}\label{io14}
f_R:=f_{+,R}-f_{-,R}.
\end{align}

\begin{lemma}\label{lemcompatib}
 For every $R\in(0,5)$, the measure $f_R$ coincides with the inner trace of $j$ on $\partial B_R\times (0,1)$ in the sense of \eqref{deftraces} i.e. for every $\zeta\in C^1_c(\R^d\times[0,1])$
 \begin{equation}\label{locconteq}
  \int_{B_R}\int_0^1 \partial_t \zeta d\rho +\nabla \zeta \cdot dj=\int_{B_R} (\zeta_1 d\lambda -\zeta_0 d\mu) + \int_{\partial B_R}\int_0^1 \zeta df_R. 
 \end{equation}
As a consequence, for every pair $(\tilde \rho,\tilde j)$  satisfying 
\begin{equation}\label{locconteqcomp}
  \int\int_0^1 \partial_t \zeta d\tilde\rho +\nabla \zeta \cdot d\tilde j=\int_{B_R} (\zeta_1 d\lambda -\zeta_0 d\mu) + \int_{\partial B_R}\int_0^1 \zeta df_R,
\end{equation}
there holds,
\begin{equation}\label{locmin}
 \int_{B_R}\int_0^1 \frac{1}{\rho}|j|^2\le \int\int_0^1 \frac{1}{\tilde \rho}|\tilde j|^2.
\end{equation}

\end{lemma}

\begin{proof}
 We first prove \eqref{locconteq}.
For $\zeta\in C^1_c(\R^d\times[0,1])$, by definition \eqref{io18} of $(\rho,j)$\footnote{as in the proof of Lemma \ref{Lres},
 it can be seen that we may use the (discontinuous) test functions  $I(|x|<R) \partial_t\zeta$ and $I(|x|<R) \nabla\zeta$.} 
 \begin{align*}
 \lefteqn{ \int_{B_R}\int_0^1 \partial_t \zeta d\rho +\nabla \zeta \cdot dj}\\
 &=\int \int_0^1\lt[\partial_t \zeta (X(t),t) +\nabla \zeta(X(t),t)\cdot \dot{X}(t)\rt] I(|X(t)|<R) dt d\P\\
  &=\int_{\Omega_R} \lt(\zeta(X(t_+),t_+)-\zeta(X(t_-),t_-)\rt) d\P.
 \end{align*}
Now by \eqref{ma2}, $\Omega_{+,R}=\Omega_R\cap \{ X \, |\, |X(1)|\ge R\}$ and thus,
\begin{align*}
 \int_{\Omega_R} \zeta(X(t_+),t_+)d\P&= \int_{\Omega_R} \zeta(X(1),1) I(|X(1)|<R)d\P\\
 &\qquad + \int_{\Omega_R} \zeta(X(t_+),t_+) I(|X(1)|\ge R)d\P\\
 &\stackrel{\eqref{io17}\&\eqref{io03}}{=} \int_{B_R} \zeta_1 d\lambda +\int_{\partial B_R}\int_0^1\zeta d f_{+,R}.
\end{align*}
Analogously,
\[
 \int_{\Omega_R} \zeta(X(t_-),t_-) d\P=\int_{B_R} \zeta_0 d\mu +\int_{\partial B_R}\int_0^1\zeta d f_{-,R},
\]
which concludes the proof of \eqref{locconteq}.\\
Let now $(\tilde \rho,\tilde j)$ be a pair satisfying \eqref{locconteqcomp}. The pair $(\check \rho,\check j):=(\tilde \rho, \tilde j)+ (\rho,j)\restr B_R^c$ 
is then a competitor for \eqref{BBwass}  so that by subadditivity of $\int \frac{1}{\rho}|j|^2$ (which follows from \eqref{dualBB}),
\[
 \int \frac{1}{\rho}|j|^2\le \int \frac{1}{\check \rho}|\check j|^2\le \int \frac{1}{\tilde \rho} |\tilde j|^2 +\int_{B^c_R}\int_0^1 \frac{1}{\rho}|j|^2,
\]
from which \eqref{locmin} follows.
\end{proof}

 \subsubsection{Regularized fluxes and good radius for the approximate orthogonality}
 We now introduce the regularized boundary flux (see Figure \ref{fig:hatg})
 
  \begin{figure}\begin{center}
 \resizebox{7.cm}{!}{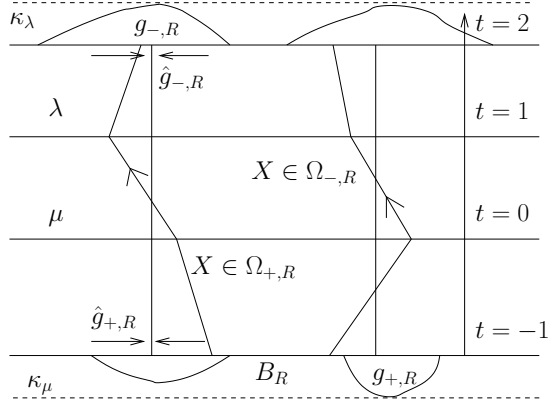} 
   \caption{The regularized fluxes $\hat{g}_{\pm}$.} \label{fig:hatg}
 \end{center}
 \end{figure}

\begin{align}\label{wg97}
\hat g_R:=\hat g_{+,R}-\hat g_{-,R}
\end{align}
where $\hat g_{+,R}\ge 0$ and $\hat g_{-,R}\ge 0$ on $\partial B_R$ are defined through
the projection onto $\partial B_R$ of where these trajectories started and ended,
respectively,
\begin{align}\label{wg78}
\ds\int\zeta d\hat{g}_{+,R}:=\ds\int_{\Omega_{+,R}}\zeta\lt(R\frac{X(-1)}{|X(-1)|}\rt)d\P, \quad
\ds\int\zeta d\hat{g}_{-,R}:=\ds\int_{\Omega_{-,R}}\zeta\lt(R\frac{X(2)}{|X(2)|}\rt)d\P.
\end{align}
From the definition it follows that these measures have densities. We will show that these densities are square integrable and $\hat{g}_{\pm,R}$ are close
to the following  non-negative measures on $\partial B_R$: 
\begin{equation}\label{wg79}
\int\zeta d\bar f_{\pm,R}:=\int_{\Omega_{\pm,R}}\zeta(X(t_\pm))d\P.
\end{equation}
For later reference, we also introduce the notation 
\begin{equation}\label{barf}
 \bar f_R:=\bar f_{+,R}-\bar f_{-,R}.
\end{equation}

\begin{lemma}[Approximation]\label{Lapp} With $M=(E+D)^\frac{1}{d+2}$ we have
\begin{align}
\int_3^4\int_{\partial B_R}\hat{g}_{\pm,R}^2\,  dR\lesssim E+D,\label{wg98}\\
\int_3^4W_{\partial B_R}^2(\hat{g}_{\pm,R},\bar f_{\pm,R})dR\lesssim M(E+D).\label{wg43}
\end{align}
\end{lemma}
\begin{proof}
 By symmetry, we can focus on the case with the $+$-sign (exiting trajectories).
We consider the measure $0\le g_{+,R}\le \kappa_\mu$ on $\mathbb{R}^d$ where trajectories in $\Omega_{+,R}$ initially started (so that in particular $\hat{g}_{+,R}= \Pi_R\# g_{+,R}$, where $\Pi_R$ is the projection on $\partial B_R$)
\begin{align}\label{wg41}
\int\zeta dg_{+,R}
=\int_{\Omega_{+,R}}\zeta(X(-1))d\P
\end{align}
and claim that it is concentrated near $\partial B_R$ in the sense that
\begin{align}
\int_3^4\int|R-|x||dg_{+,R}dR\lesssim E+D.\label{wg44}
\end{align}
Then, (\ref{wg98}) would follow from \eqref{estimz}.
\medskip

 By definition (\ref{wg41}) of $g_{+,R}$ and definition 
(\ref{io26}) of $t_+$ on $\Omega_{+,R}$ we have 
\begin{align*}
\int|R-|x||dg_{+,R}=\int_{\Omega_{+,R}}||X(t_+)|-|X(-1)||d\P
\end{align*}
and thus by the triangle inequality and the straightness of trajectories on $(0,1)$
\begin{equation}\label{io12}
\int|R-|x||dg_{+,R}\le\int_{\Omega_{+,R}}(|X(1)-X(0)|+|X(0)-X(-1)|)d\P.
\end{equation}
Furthermore, since by definition \eqref{ma2}, for $X\in\Omega_{+,R}$ we have $|X(1)|\ge R\ge \min_{[0,1]}|X|$ and thus
\begin{align}\label{wg45}
\int_0^\infty I(X\in\Omega_{+,R})dR=|X(1)|-\min_{[0,1]}|X|\le|X(1)-X(0)|.
\end{align}
Combining (\ref{io12}) with (\ref{wg45}) we obtain
\begin{align*}
\lefteqn{\int_3^4\int|R-|x||dg_{+,R}dR}\nonumber\\
&\le\int|X(1)-X(0)|(|X(1)-X(0)|+|X(0)-X(-1)|)d\P.
\end{align*}
By Cauchy-Schwarz and (\ref{io16})
this yields as desired $\int_3^4\int(R-|x|)dg_{+,R}dR\le E+(ED)^\frac{1}{2}$.

\medskip

We now turn to (\ref{wg43}).
It follows from the definitions \eqref{wg78} and \eqref{wg79}  that
\begin{align*}
\int\zeta(x,y)d\pi_R(x,y)=\int_{\Omega_{+,R}}\zeta\lt(X(t_+),R\frac{X(-1)}{|X(-1)|}\rt)d\P
\end{align*}
defines a coupling $\pi_R$ between $\bar f_{+,R}$ and $\hat{g}_{+,R}$ so that in particular
\begin{align*}
W^2_{\partial B_R}(\bar f_{+,R},\hat{g}_{+,R})\les W^2(\bar f_{+,R},\hat{g}_{+,R})
\le\int_{\Omega_{+,R}}\lt|X(t_+)-R\frac{X(-1)}{|X(-1)|}\rt|^2d\P.
\end{align*}
Given $X\in\Omega_{+,R}$ note that since $|X(t_+)|=R$ we have (recall $R\in (3,4)$)
\[\lt|X(t_+)-R\frac{X(-1)}{|X(-1)|}\rt|=R\lt|\frac{X(t_+)}{|X(t_+)|}-\frac{X(-1)}{|X(-1)|}\rt|
\le 4\lt|\frac{X(t_+)}{|X(t_+)|}-\frac{X(-1)}{|X(-1)|}\rt|\]
and  since by \eqref{ma3} and $M\ll 1$,
\[|X(-1)|\ge|X(t_+)|-|X(t_+)-X(0)|-|X(0)-X(-1)|\ge 3-CM\ge 2,\]
we thus get
$|X(t_+)-R\frac{X(-1)}{|X(-1)|}|\lesssim |X(t_+)-X(-1)|$, so that
\begin{align*}
W^2_{\partial B_R}(\bar f_{+,R},\hat{g}_{+,R})\lesssim\int_{\Omega_{+,R}}|X(t_+)-X(-1)|^2d\P.
\end{align*}
By the triangle inequality and the straightness of trajectories $X$ on $(0,1)$ this yields
\begin{align*}
W^2_{\partial B_R}(\bar f_{+,R},\hat{g}_{+,R})
\lesssim\int_{\Omega_{+,R}}(|X(1)-X(0)|^2+|X(0)-X(-1)|^2)d\P.
\end{align*}
By (\ref{wg45}) this upgrades to
\begin{align*}
\lefteqn{\int_3^4W^2_{\partial B_R}(\bar f_{+,R},\hat{g}_{+,R})dR}\nonumber\\
&\lesssim\int_{\Omega_5}|X(1)-X(0)|(|X(1)-X(0)|^2+|X(0)-X(-1)|^2)d\P\nonumber\\
&\stackrel{\eqref{ma3}}{\les} M(E+D).
\end{align*}
\end{proof}
Using Lemma \ref{Lres}, we now show that we can choose $R$ such that $\bar \rho$ (recall \eqref{ma15}) is close to the Lebesgue measure in Wasserstein distance.
\begin{lemma}\label{Lres'}
 We have 
 \begin{equation}\label{wg23'}
 \int_3^4 (W^2_{B_R}(\bar \rho, \kappa_R)+\frac{1}{\kappa_R}(\kappa_R-1)^2)dR\les E+D.
\end{equation}
\end{lemma}

\begin{proof}

 From the first item in
(\ref{io17}) and from (\ref{io16}) we have for all $t\in[0,1]$ that 
\begin{align*}
W^2_{B_5}(\rho_t',\kappa_\mu)\le E+D,
\end{align*}
where the measure $0\le\rho_t'\le\rho_t$ is defined via
\begin{align*}
\int\zeta d\rho_t'=\int \zeta(X(t)) I(|X(-1)|<5) d\P.
\end{align*}
By \eqref{ma3}, we obtain from (\ref{io18}) that $\rho_t'=\rho_t$ on $B_4$.
Hence, applying Lemma \ref{Lres} with $\rho_t'$ playing the role of $\mu$
yields (recall \eqref{defkappa})
\begin{align*}
\int_3^4\lt(W_{B_R}^2(\rho_t,\kappa_t)+\frac{1}{\kappa_t}(\kappa_t-1)^2\rt)dR\lesssim E+D,
\end{align*}
which once integrated in time and  thanks to the sub-additivity of $W^2$ gives the claim.
\end{proof}
 \subsubsection{Good radius for the construction}\label{sec:exceptional}
 For the construction part of the proof, we will need to introduce initial/terminal layers of size (multiples of) $\tau\ll1$. 
 We will treat a class of exceptional trajectories separately, namely exiting trajectories that exit
too early and entering trajectories that enter too late:
\begin{align}\label{io01}
\del\Omega_R:=&(\Omega_{+,R}\cap\{t_+\le 3\tau\})\cup
(\Omega_{-,R}\cap\{t_-\ge 1-3\tau\}).
\end{align}
These trajectories will be inserted into our construction of a competitor.

This means that we will have to pass (however only inside the construction part of the proof) to modified versions of the fluxes. We thus introduce the reduced sets of trajectories 
\begin{equation}\label{ma4}
 \Omega'_{+,R}:=\Omega_{+,R}\backslash\, \del\Omega_R, \qquad \Omega'_{-,R}:=\Omega_{-,R}\backslash \, \del\Omega_R
\end{equation}
and then the reduced fluxes on $\partial B_R\times [0,1]$
\begin{equation}\label{io05}
 \int \zeta df'_{\pm,R}= \int_{ \Omega'_{\pm,R}} \zeta (X(t_\pm),t_\pm) d\P  \qquad\textrm{ and} \qquad f'_R:=f'_{+,R}-f'_{-,R}, 
\end{equation}
from which $\bar f'_{\pm,R}$ and $\bar f'_R$ are naturally defined as their time integral. Notice that $f'_R$ is the flux created by all the trajectories which are not in $\del \Omega_R$. We finally introduce the corresponding
(constant in time) regularized fluxes:
\begin{equation}\label{io11}
 \ds \int\zeta d\hat{g}_{+,R}'=\ds\int_{\Omega'_{+,R}}\zeta\lt(R\frac{X(-1)}{|X(-1)|}\rt)d\P, \quad 
\ds\int\zeta d\hat g_{-,R}'=\ds\int_{\Omega'_{-,R}}\zeta\lt(R\frac{X(2)}{|X(2)|}\rt)d\P
\end{equation}
and
\begin{equation}\label{io10}
 \hat{g}'_R:=\hat{g}'_{+,R}-\hat{g}'_{-,R}.
\end{equation}

Arguing as in Lemma \ref{Lapp}, we obtain the following analogous estimates:
\begin{lemma}\label{Lapp'}
 With $M=(E+D)^{\frac{1}{d+2}}$, we have
 \begin{align}
\int_3^4\int_{\partial B_R}(\hat{g}'_{\pm,R})^2\,  dR\lesssim E+D,\label{wg98'}\\
\int_3^4W_{\partial B_R}^2(\hat{g}'_{\pm,R},\bar f'_{\pm,R})dR\lesssim M(E+D).\label{wg43'}
\end{align}
Moreover,
\begin{equation}\label{wg98diff}
 \int_3^4\int_{\partial B_R}(\hat{g}'_{\pm,R}-\hat{g}_{\pm,R})^2\,  dR\lesssim \tau^2 E+D,
\end{equation}
\end{lemma}
\begin{proof}
 The proofs of \eqref{wg98'} and \eqref{wg43'} are the same as the proofs of \eqref{wg98} and \eqref{wg43} so that we only need to comment on \eqref{wg98diff}. The proof follows the lines of \eqref{wg98}. 
 We focus once again only on the case with +-sign. Recall the definition \eqref{wg41} of $g_{+,R}$ and  analogously define  $g'_{+,R}$ so that 
 \[
  \int \zeta d(g_{+,R}-g'_{+,R})=\int_{\Omega_{+,R}\cap \del\Omega_R} \zeta(X(-1)) d\P.
 \]
 Since by definition \eqref{io01} of $\del\Omega_R$
 \[
  \Omega_{+,R}\cap \del\Omega_R=\lt(\Omega_{+,R}\cap\{ t_+\le  3 \tau\}\rt) \cup  \lt(\Omega_{+,R}\cap\{ t_-\ge 1-3\tau\}\rt),
 \]
using the piecewise straightness of trajecteries in form of $|X(t_+)-X(0)|\le 3\tau|X(1)-X(0)|$, the analog of \eqref{io12} becomes 
\begin{multline*}
 \int |R-|x|| d(g_{+,R}-g'_{+,R})\le \int_{\Omega_{+,R}\cap\{t_+\le 3\tau \}} (3\tau |X(1)-X(0)|+ |X(0)-X(-1)|) d\P\\
 + \int_{\Omega_{+,R}\cap\{ t_-\ge 1-3\tau\}} (|X(1)-X(0)|+ |X(0)-X(-1)|) d\P.
\end{multline*}
Since  for $X\in \Omega_{+,R}\cap\{t_+\le 3\tau \}$ we have $\min_{[0,3\tau]}|X|<R\le |X(3\tau)|$ and therefore, the analog of \eqref{wg45} for the first term on the right-hand side 
turns into
\[
 \int_0^\infty I(X\in \Omega_{+,R}\cap\{t_+\le  3\tau \})dR \les \tau|X(1)-X(0)|.
\]

We now estimate  the second term on the right-hand side and note that 
for $X\in\Omega_{+,R}\cap\{t_-\ge 1-3\tau\}$ we have by definitions (\ref{ma2})
and (\ref{io26}) 
that $|X(1)|,|X(1-3\tau)|\ge R$ while $\min_{[1-3\tau,1]}|X|<R$. By 
the straightness of trajectories and the $L^\infty$ bound (\ref{ma3}) in form of
\begin{align}\label{as01}
\max_{t\in [1-3\tau,1]}|X(1)-X(t)|\le 3\tau|X(1)-X(0)|\le 3\tau M
\end{align}
in conjunction with $R\ge 3$, we always have $\min_{[1-3\tau,1]}|X|\ge 2$.
Hence the elementary inequality
$\min\{|X(1)|,|X(1-3\tau)|\}$ $-\min_{[1-3\tau,1]}|X|$ 
$\le\frac{|X(1)-X(1-3\tau)|^2}{8\min_{[1-3\tau,1]}|X|}$ implies
\begin{align*}
\int_3^4&I(X\in\Omega_{+,R}\cap\{t_-\ge 1-3\tau\}) dR
\le\frac{1}{16}|X(1)-X(1-3\tau)|^2\nonumber\\
&\stackrel{(\ref{as01})}{\le}\frac{3}{16}\tau^2M|X(1)-X(0)|\ll \tau^2|X(1)-X(0)|,
\end{align*}
(in fact, this contribution will be of higher order).
Summing up, we obtain 
\begin{align*}
\lefteqn{\int_3^4\int(R-|x|)d(g_{+,R}-g_{+,R}')d\mathbb{P}}\nonumber\\
&\lesssim\int\tau|X(1)-X(0)|(\tau|X(1)-X(0)|+|X(0)-X(-1)|)d\mathbb{P}\\
&+\int\tau^2|X(1)-X(0)|(|X(1)-X(0)|+|X(0)-X(-1)|)d\mathbb{P},
\end{align*}
so that (\ref{wg98diff}) follows from Young's inequality in form
of $2\tau|X(1)-X(0)||X(0)-X(-1)|$ $\le\tau^2|X(1)-X(0)|^2+|X(0)-X(-1)|^2$ and \eqref{io16}.
\end{proof}
 Define the sets of generic trajectories (recall \eqref{OmR} and \eqref{io01}) that on $[0,1]$ start in $B_R$ and do not exit too early or that end in $B_R$ and do not enter too late, respectively,  
  \begin{align}\label{ma5}
\begin{array}{ccc}
 \Omega_{0,R}&:=&(\Omega_R\backslash \, \del \Omega_R)\cap \{ |X(0)|<R\}=\{ |X(0)|<R \textrm{ and } t_+> 3\tau\},\\[6pt]
  \Omega_{1,R}&:=&(\Omega_R\backslash \, \del \Omega_R)\cap \{ |X(1)|<R\}=\{|X(1)|<R \textrm{ and } t_-< 1-3\tau\}.
\end{array}
  \end{align}
   In line with \eqref{io17}, we define the measures in $B_R$
  \begin{equation}\label{io07bis}
\ds\int\zeta d\mu'_R=\ds\int_{\Omega_{0,R}}\zeta(X(0))d\P, \qquad 
\ds\int\zeta d\lambda'_R=\ds\int_{\Omega_{1,R}}\zeta(X(1))d\P,
\end{equation}
which correspond to where the trajectories which are not in $\del\Omega_R$ start and end provided they start or end in $B_R$. An important step in the construction is to connect over a time $\tau$ the measure $0\le \mu'_R\le \mu$ 
to a constant $\kappa_{\mu'_R}$ and the measure $0\le \lambda'_R\le \lambda$ to another constant $\kappa_{\lambda'_R}$.
\begin{lemma}\label{lem:connectboundarylayer}
 We have
 \begin{multline}\label{ma6}
  \int_3^4 \lt(W^2_{B_R}(\mu'_R, \kappa_{\mu'_R})+\frac{1}{\kappa_{\mu'_R}}(\kappa_{\mu'_R}-1)^2+ W^2_{B_R}(\lambda'_R, \kappa_{\lambda'_R})+\frac{1}{\kappa_{\lambda'_R}}(\kappa_{\lambda'_R}-1)^2\rt) \, dR\\
  \les \tau^2E +D.
 \end{multline}
\end{lemma}
\begin{proof}
By symmetry, we may restrict to $\mu_R'$. 
The statement follows from combining the outcome of the four steps below.

\medskip

{\sc Step 1}. Introducing the measure $0\le g_{0,R}\le\kappa_{\mu}$ that captures where
the generic trajectories $X\in\Omega_{0,R}$ start at time $t=-1$, that is,
\begin{align}\label{wg58}
\int\zeta dg_{0,R}=\int_{\Omega_{0,R}}\zeta(X(-1))d\mathbb{P},
\end{align}
we claim that its deviation from the characteristic function $\kappa_{\mu}I(B_R)$ is small
(on average in $R\in(3,4)$):
\begin{align}\label{wg51}
\int_3^4\int_{B_R}&(R-|x|)d(\kappa_{\mu}-g_{0,R})+\int_{B_R^c}(|x|-R)dg_{0,R}dR
\nonumber\\
&\lesssim\tau^2 E+D.
\end{align}
The argument is a refinement of (\ref{wg44}) in Lemma \ref{Lapp}. 
By $\Omega_{0,R}\subset\{|X(0)|<R\}$, cf.~(\ref{ma5}), we have on the one hand
\begin{align*}
\int_{B_R^c}(|x|-R)dg_{0,R}\le \int_{|X(0)|<R\le|X(-1)|}|X(-1)|-|X(0)|d\mathbb{P}.
\end{align*}
Note that by (\ref{io17}) and (\ref{wg58}), we have on the other hand
\begin{align*}
\int_{B_R}(R-|x|)d(\kappa_{\mu}-g_{0,R})
=\int_{\Omega_{0,R}^c}(R-|X(-1)|)_+d\mathbb{P}
\end{align*}
and thus, since $\Omega_{0,R}^c$ $\subset\{|X(0)|\ge R\}\cup\{t_+\le3\tau\}$
$\subset\{\max(|X(0)|,|X(3\tau)|)\ge R\}$, cf.~(\ref{ma5})
and (\ref{io26}),
\begin{align*}
\lefteqn{\int_{B_R}(R-|x|)d(\kappa_{\mu}-g_{0,R})}\nonumber\\
&\le\int_{|X(-1)|<R\le\max(|X(0)|,|X(3\tau)|)}\max(|X(0)|,|X(3\tau)|)-|X(-1)|d\mathbb{P}.
\end{align*}
Integrating these two estimates  in $R$ and using 
 the straightness of $X$ on $[0,1]$ 
in form of 
\[\max(|X(0)|,|X(3\tau)|)-|X(-1)| 
\le 3\tau|X(1)-X(0)|+|X(0)-X(-1)|,\]
we obtain \eqref{wg51}.

\medskip

{\sc Step 2}. We claim that the total mass density $\kappa_{\mu_R'}=\frac{\mu_R'(B_R)}{|B_R|}$ is
close to $1$ on average: 
\begin{align}\label{ma7}
\int_3^4\frac{1}{\kappa_{\mu_R'}}(\kappa_{\mu_R'}-1)^2dR\lesssim\tau^2E+D.
\end{align}
Indeed, from definitions (\ref{io07bis}) and (\ref{wg58}) we obtain
$\mu_R'(B_R)$ $=\mathbb{P}(\Omega_{0,R})$ $=\int dg_{0,R}$ which we write as
\begin{align*}
|B_R|(\kappa_{\mu_R'}-\kappa_\mu)=\int_{B_R}(g_{0,R}-\kappa_\mu)+\int_{B_R^c}g_{0,R}.
\end{align*}
We now appeal to (\ref{estimz}) in Lemma \ref{2.3}, and use
$0\le g_{0,R}\le\kappa_\mu\sim 1$ to obtain
\begin{align*}
(\kappa_{\mu_R'}-\kappa_\mu)^2
\lesssim 
\int_{B_R}(R-|x|)(\kappa_{\mu}-g_{0,R})+\int_{B_R^c}(|x|-R)g_{0,R}.
\end{align*}
Thus (\ref{ma7}) follows from (\ref{wg51}), $\frac{1}{\kappa_\mu}(\kappa_\mu-1)^2\le D\ll 1$ and the fact that by \eqref{ma3}
\begin{equation}\label{kappamuprime}
 \kappa_{\mu_R'}\sim 1 \qquad \forall R\in (3,4).
\end{equation}

\medskip

{\sc Step 3}. We claim the inequality
\begin{align}\label{as05}
\lefteqn{W_{B_R}(\mu_R',\kappa_{\mu_R'})}\nonumber\\
&\le D^\frac{1}{2}
+\lt(M\int_{B_R^c}(|x|-R)dg_{0,R}\rt)^\frac{1}{2}
+W_{\bar B_R}(g_{0,R}+\hat g_{0,R},\kappa_{\mu_R'}),
\end{align}
where in line with (\ref{io11}), $\hat g_{0,R}$ is defined as the projection
of $g_{0,R}$ from $B_R^c$ onto $\partial B_R$, that is,
\begin{align}\label{as06}
\int\zeta d\hat g_{0,R}=\int_{B_R^c}\zeta\lt(R\frac{x}{|x|}\rt) dg_{0,R}.
\end{align}
Indeed, we start from the triangle inequality
\begin{align*}
\lefteqn{W_{B_R}(\mu_R',\kappa_{\mu_R'})}\nonumber\\
&\le W(\mu_R',g_{0,R})+W(g_{0,R},g_{0,R}I(B_R)+\hat g_{0,R})
+W_{\bar B_R}(g_{0,R}+\hat g_{0,R},\kappa_{\mu_R'}).
\end{align*}
Clearly, by the definitions (\ref{io07bis}) and (\ref{wg58}), it follows from
(\ref{io16}) that  $ W(\mu_R',g_{0,R})\le D^\frac{1}{2}$. By monotonicity,
 $W(g_{0,R},g_{0,R}I(B_R)+\hat g_{0,R})\le W(g_{0,R}I(B_R^c),\hat g_{0,R})$; in view of
definition (\ref{as06}) and a purely radial coupling we obtain
$W^2(g_{0,R}I(B_R^c),\hat g_{0,R})$ $\le \int_{B_R^c}(|x|-R)^2 dg_{0,R}$; 
by definition (\ref{wg58}) and the $L^\infty$-bound (\ref{ma3}), 
$g_{0,R}$ is supported in $B_{R+M}$, so that
we obtain the claimed inequality on this term, too.

\medskip

{\sc Step 4}. We claim that the last term in (\ref{as05}) is controlled as follows
\begin{align*}
\lefteqn{W_{\bar B_R}^2(g_{0,R}+\hat g_{0,R},\kappa_{\mu_R'})}\nonumber\\
&\lesssim
\int_{B_R}(R-|x|)d(\kappa_{\mu}-g_{0,R})+\int_{B_R^c}(|x|-R)dg_{0,R}.
\end{align*}
We follow the argument from Lemma \ref{Lres}. We first appeal to (\ref{wg84})
in Lemma \ref{2.7} to reduce to an estimate of
$W_{\bar B_R}^2(g_{0,R}+\hat g_{0,R}+\kappa_{\mu_R'},2\kappa_{\mu_R'})$.
We then argue as for (\ref{wg93}) to reduce further to an estimate of
$\frac{1}{\kappa_{\mu_R'}}\int_{B_R}|\nabla\vp|^2$, where $\nabla\vp$ is determined through
\begin{align*}
\Delta \vp=c-g_{0,R}\ \textrm{ in } B_R
\qquad\mbox{ and }\qquad\nu\cdot\nabla\vp=\hat g_{0,R}\ \textrm{ on } \partial B_R,
\end{align*}
and $c$ is such that the problem is solvable. By (\ref{kappamuprime})  we just need to control $\int_{B_R}|\nabla\vp|^2$, for
which we  appeal to the energy estimate (\ref{estim:elliptic})
in conjunction with (\ref{estimz}).
\end{proof}

\subsubsection{Choice of the good radius and definition of $\Phi$ and $\phi$}
We can finally choose a good radius. We refer to \eqref{ma15} for the definition of $\bar \rho$, to \eqref{wg78} for the definition of $\hat{g}_{\pm,R}$, to \eqref{wg79} for the definition of $\bar f_{\pm,R}$,  to \eqref{io01} for the definition of $\del\Omega_R$, to \eqref{io05} for the definition of $\bar f'_{\pm,R}$,
to \eqref{io11} for the definition of $\hat{g}'_{\pm,R}$,  and to \eqref{io07bis} for the definitions of $\mu'_R$ and $\lambda'_R$. 
\begin{proposition}\label{prop:goodR}
 Let $\pi$ be an optimal coupling between $\mu$ and $\lambda$ such that $E+D\ll1$. Recall the definition $M= (E+D)^{\frac{1}{d+2}}$. Then, there exists a set of radii $R\in (3,4)$ of positive measure such that  we have on the one hand:
 \begin{align}
 & W^2_{B_R}(\mu,\kappa_{\mu,R})+W^2_{B_R}(\lambda, \kappa_{\lambda,R})\les D, \label{wg21}\\
&  W_{B_R}^2\lt(\bar \rho ,\kappa\rt)+\frac{1}{\kappa}(\kappa-1)^2\les E+D, \label{wg23}\\
&  W_{\partial B_R}^2(\bar f_{\pm,R}, \hat{g}_{\pm,R})\les M(E+D), \label{ma9}\\
 & \int_{\partial B_R} \hat{g}_{\pm,R}^2 \les E+D \label{ma10}
 \end{align}
and on the other hand:
\begin{align}
&\int_{\partial B_R}(\hat{g}'_{\pm,R})^2\lesssim E+D,\label{ma11}\\
&W_{\partial B_R}^2(\hat{g}'_{\pm,R},\bar f'_{\pm,R})\lesssim M(E+D),\label{ma12}\\
&\int_{\partial B_R}(\hat{g}'_{\pm,R}-\hat{g}_{\pm})^2 \lesssim \tau^2 E +D,\label{ma13}\\
&  W^2_{B_R}(\mu'_R, \kappa_{\mu'_R})+\frac{1}{\kappa_{\mu'_R}}(\kappa_{\mu'_R}-1)^2+ W^2_{B_R}(\lambda'_R, \kappa_{\lambda'_R})+\frac{1}{\kappa_{\lambda'_R}}(\kappa_{\lambda'_R}-1)^2 \nonumber \\ & \qquad \les \tau^2E +D.\label{ma14}
\end{align}
\end{proposition}
\begin{proof}
 Estimate \eqref{wg21} follows from an application of Lemma \ref{Lres} both to $\mu$ and $\lambda$; \eqref{wg23} is obtained from \eqref{wg23'} of Lemma \ref{Lres'};
 \eqref{ma9} and \eqref{ma10} follow from Lemma \ref{Lapp};  \eqref{ma11}, \eqref{ma12} and \eqref{ma13} are a by-product of Lemma \ref{Lapp'} and \eqref{ma14} follows from Lemma \ref{lem:connectboundarylayer}. 
 \end{proof}
From now  on we fix a radius $R\in (3,4)$ satisfying all the conditions of Proposition \ref{prop:goodR}.\\
\textbf{Convention:} we drop the notational dependence in $R$ when it does not lead to confusion (and write in particular $f$ for $f_R$, $\hat{g}$ for $\hat{g}_R$ etc...).\\

 We recall the definition of $\Phi$ from \eqref{ma89}, the solution to the Poisson equation with  non-regularized flux boundary conditions (recall the definition \eqref{barf} of $\bar f$)
 \begin{equation}\label{defPhi}
  \Delta \Phi= c \textrm{ in } B_R \qquad \nu\cdot\nabla \Phi= \bar{f} \textrm{ on } \partial B_R 
 \end{equation}
and define the function $\phi$, which is the solution to the same Poisson equation but with regularized flux boundary conditions
\begin{equation}\label{defphi}
  \Delta \phi= c \textrm{ in } B_R \qquad \nu\cdot\nabla \phi= \hat{g} \textrm{ on } \partial B_R. 
 \end{equation}
We normalize $\Phi$ and $\phi$ by imposing $\int_{B_R}\Phi=\int_{B_R} \phi=0$. As already explained, we will first prove the harmonic approximation result with the function $\phi$ and then prove that we can replace it by the function $\Phi$.  
 \subsection{Eulerian version of the harmonic approximation result}
 The aim of this section is to prove the Eulerian version of the harmonic approximation result, that is Proposition \ref{eulerian1stepintro}, which we recall for the reader's convenience:
\begin{prop1.6}
 For every $0<\tau\ll1$, there exist positive constants $\eps(\tau)>0$ and $C(\tau)>0$ such that if $E+D\le \eps(\tau)$, then there exists $R\in(3,4)$ such that if $\Phi$ is defined through \eqref{defPhi},
 \begin{equation}\label{ma36}
  \int_{B_2}\int_0^1\frac{1}{\rho} |j-\rho \nabla \Phi|^2\le \tau E+C(\tau) D.
 \end{equation}
 Moreover,
 \begin{equation}\label{estimPhi}
  \sup_{B_2} |\nabla^2 \Phi|^2+ \sup_{B_2} |\nabla \Phi|^2 \les E+D.
 \end{equation}

\end{prop1.6}

 Most of the work will be devoted to the proof of the analog statement but with $\Phi$ replaced by $\phi$.
 \begin{proposition}\label{prop:eulerian1stepphi}
 For every $0<\tau\ll1$, there exist positive constants $\eps(\tau)>0$ and $C(\tau)>0$ such that if $E+D\le \eps(\tau)$, then there exists $R\in(3,4)$ such that if $\phi$ is defined through \eqref{defphi},
 \begin{equation}\label{ma16}
  \int_{B_2}\int_0^1\frac{1}{\rho} |j-\rho \nabla \phi|^2\le \tau E+C(\tau) D.
 \end{equation}
\end{proposition}
\begin{proof}
           The proof is a direct combination of the orthogonality Lemma \ref{Lortintro} and the construction Lemma \ref{Lconintro} which are contained in the next two subsections  together with \eqref{locmin} and the choice of parameters $r=\tau$.       
                 \end{proof}
Before coming to the proof of Lemma \ref{Lortintro} and  Lemma \ref{Lconintro}, let us derive Proposition \ref{eulerian1stepintro} from Proposition \ref{prop:eulerian1stepphi}.
\begin{proof}[Proof of Proposition \ref{eulerian1stepintro}]
 Let $\vp_{\pm}$ be the solution of 
 \[
  \Delta \vp_\pm= 0 \ \textrm{ in } B_R \qquad \textrm{ and } \qquad \nu\cdot \nabla \vp_{\pm}= \bar f_\pm-\hat{g}_\pm \ \textrm{ on } \partial B_R 
 \]
with $\int_{B_R} \vp_\pm=0$ so that $\Phi-\phi=\vp_+-\vp_-$. Applying Lemma \ref{lem:difphi} to $\vp_\pm$, we find that for every $0<\alpha<1$, 
\begin{align*}\sup_{B_2}|\nabla^2 (\Phi-\phi)|+\sup_{B_2} |\nabla (\Phi-\phi)|&\les W^\alpha_{\partial B_R} (\bar f_\pm,\hat{g}_\pm) \lt[ \hat{g}_{\pm}(\partial B_R)\rt]^{1-\frac{\alpha}{2}}\\
 &\stackrel{\eqref{ma9}}{\les}\lt(M(E+D)\rt)^{\frac{\alpha}{2}} \lt(\int_{\partial B_R} \hat{g}_\pm^2\rt)^{\frac{1}{2} -\frac{\alpha}{4}} \\
 &\stackrel{\eqref{ma10}}{\les}\lt(M(E+D)\rt)^{\frac{\alpha}{2}} \lt(E+D\rt)^{ \frac{1}{2} -\frac{\alpha}{4}}\\
 &\les M^{\frac{\alpha}{2}}(E+D)^{\frac{1}{2}+\frac{\alpha}{4}}.
\end{align*}
Combining this with \eqref{wg12} applied to $\phi$ and \eqref{ma10} we first obtain \eqref{estimPhi}. This also yields 
\begin{align*}
 \int_{B_2}\int_0^1 \frac{1}{\rho}|j-\rho \nabla \Phi|^2&\les \int_{B_2}\int_0^1 \frac{1}{\rho}|j-\rho \nabla \phi|^2+\int_{B_2}  |\nabla \Phi-\nabla \phi|^2 d\bar \rho\\
 &\stackrel{\eqref{ma16}}{\les} \tau E+C(\tau) D+M^{\alpha}(E+D)^{1+\frac{\alpha}{2}}\int_{B_2}d\bar \rho\\
 &\stackrel{\eqref{ma42}}{\les}\tau E+C(\tau) D+M^{\alpha}(E+D)^{1+\frac{\alpha}{2}},
\end{align*}
which concludes the proof of \eqref{ma36}. 
\end{proof}

 \subsubsection{Approximate orthogonality}\label{sec:approxorth}
 In this section we prove the approximate orthogonality property, that is Lemma \ref{Lortintro}, which we recall for the reader's convenience\footnote{Notice that since in the second step of the proof we appeal to interior regularity estimates, we need to go from $B_R$ in the term between parenthesis 
 on the left-hand side of \eqref{io15} to the smaller ball $B_2$ for the first term on the left-hand side of \eqref{io15}} :
 \begin{lem1.7}[Orthogonality]
For every $0<r\ll 1$, there exist $\eps(r)>0$ and $C(r)>0$ such that if $E+D\le \eps$, 
  \begin{equation}\label{io15}
   \int_{B_2}\int_0^1 \frac{1}{\rho}|j-\rho\nabla \phi|^2- \lt(\int_{B_R}\int_0^1 \frac{1}{\rho}|j|^2-\int_{B_R}|\nabla \phi|^2\rt)
   \le r E +C  D.
  \end{equation}
 \end{lem1.7}
\begin{proof}
Let $0<r\ll1$ be given and let us recall the notation from Lemma \ref{lem:ellipticr}. If $\hat{g}$ is given by \eqref{wg97}, we let $\hat{g}_r$ be the convolution with a smooth convolution kernel at scale $r$ (on $\partial B_R$) of $\hat{g}$. We then let $\phi^r$ with $\int_{B_R} \phi^r=0$ be the solution of
 \begin{equation}\label{wg01}
  \Delta \phi^r=c \ \textrm{ in }  B_R \qquad \textrm{ and } \qquad  \nu\cdot \nabla \phi^r=\hat{g}_r\  \textrm{ on } \partial B_R.
 \end{equation}
We first prove \eqref{io15} with $\phi^r$ instead of $\phi$ and then get rid of this regularization.\\

\medskip
{\sc Step 1.} We prove that 
\begin{multline}\label{ma17}
 \int_{B_R}\int_0^1\frac{1}{\rho}|j-\rho \nabla\phi^r|^2-\lt(\int_{B_R}\int_0^1 \frac{1}{\rho}|j|^2-\int_{B_R} |\nabla \phi^r|^2\rt)\\
 \les r E +\frac{1}{r} D +\frac{1}{r^{d+2}} D^2 +\frac{1}{r^{d+1}} E^2+\frac{1}{r^{\frac{d+1}{2}}}  M^{\frac{1}{2}} (E+D)^{\frac{5}{4}}.
\end{multline}
Our argument relies on the formula 
\begin{align*}
\int_{B_R}\int_0^1 \frac{1}{\rho}|j-\rho\nabla \phi^r|^2=&\int_{B_R}\int_0^1 \frac{1}{\rho}|j|^2-\int_{B_R} |\nabla \phi^r|^2\\
 &+2\int_{B_R} \phi^r d(\mu-\lambda) +2\int_{\partial B_R} \phi^r d(\hat{g}_r-\bar f)\\
 \qquad &+\int_{B_R} |\nabla \phi^r|^2(d\bar \rho-dx),
\end{align*}
where we recall the definitions of  $\bar f$ in \eqref{barf} and $\bar \rho$ in \eqref{ma15}. This formula is obtained by expanding the square and combining \eqref{locconteq} with \eqref{wg01}, both tested with $\phi^r$ (which is smooth), and using that $\int_{B_R} \phi^r=0$.\\

\medskip
{\sc Step 1.1.} We first estimate $\int_{B_R} \phi^r d(\mu-\lambda)$. For this, using again $\int_{B_R} \phi^r=0$, we write it as (recall \eqref{wg21} for the definition of $\kappa_{\mu,R}$ and $\kappa_{\lambda,R}$)
\[
 \int_{B_R} \phi^r d(\mu-\lambda)=\int_{B_R} \phi^r (d\mu-\kappa_{\mu,R} dx)+\int_{B_R} \phi^r (d\lambda-\kappa_{\lambda,R} dx)
\]
and estimate using $\kappa_{\mu,R}\sim 1$ in the second step
\begin{align*}
 \lefteqn{\lt|\int_{B_R} \phi^r  (d\mu-\kappa_{\mu,R} dx)\rt|}\\
 &\stackrel{\eqref{wg80}}{\le  }\lt(\kappa_{\mu,R}\int_{B_R} |\nabla \phi^r|^2\rt)^{\frac{1}{2}} W_{B_R}(\mu,\kappa_{\mu,R})+\frac{1}{2} \sup_{B_R} |\nabla^2 \phi^r| W_{B_R}^2(\mu,\kappa_{\mu,R})\\
 &\stackrel{\eqref{wg12}\&\eqref{wg03}}{\les} \lt(\int_{\partial B_R} \hat{g}_r^2\rt)^{\frac{1}{2}} W_{B_R}(\mu,\kappa_{\mu,R})+\frac{1}{r}\lt(\frac{1}{r^{d-1}}\int_{\partial B_R} \hat{g}_r^2\rt)^{\frac{1}{2}} W_{B_R}^2(\mu,\kappa_{\mu,R})\\
 &\stackrel{\eqref{ma10}\&\eqref{wg21}}{\les} (E+D)^{\frac{1}{2}}D^{\frac{1}{2}}+ \frac{1}{r^{\frac{d+1}{2}}}(E+D)^{\frac{1}{2}} D \\
 &\les r E +\frac{1}{r} D +\frac{1}{r^{d+2}} D^2.
\end{align*}
The term $\int_{B_R} \phi^r (d\lambda-\kappa_{\lambda,R} dx)$ may be estimated analogously.\\

\medskip
{\sc Step 1.2.} We then estimate $\int_{\partial B_R} \phi^r d(\hat{g}_r-\bar f)$. For this we split this term as follows 
\[
 \int_{\partial B_R} \phi^r d(\hat{g}_r-\bar f)=\int_{\partial B_R} (\phi^r_r-\phi^r) d\hat{g}+\int_{\partial B_R} \phi^r d(\hat{g}-\bar f).
\]
On the one hand we have 
\begin{equation*}
 \lt|\int_{\partial B_R} (\phi^r_r-\phi^r) d\hat{g}\rt|\le r\lt(\int_{\partial B_R} |\nabla \phi^r|^2 \int_{\partial B_R} \hat{g}^2\rt)^{\frac{1}{2}}
 \stackrel{\eqref{wg15}}{\les} r \int_{\partial B_R} \hat{g}^2\stackrel{\eqref{ma10}}{\les} r(E+D). 
\end{equation*}

On the other hand, treating separately the positive and negative parts and  arguing as for \eqref{wg80}, we obtain 
\begin{align*}
 \lefteqn{\lt|\int_{\partial B_R} \phi^r d(\hat{g}_\pm-\bar f_\pm)\rt|}\\
 &\le \sup_{\partial B_R} |\nabla_{tan} \phi^r| \lt(\int_{\partial B_R} \hat{g}_\pm \rt)^{\frac{1}{2}}W_{\partial B_R} (\bar f_\pm,\hat g_\pm)+ \frac{1}{2} \sup_{\partial B_R} |\nabla^2_{tan} \phi^r| \, W^2_{\partial B_R} (\bar f_\pm,\hat g_\pm)\\
  &\stackrel{\eqref{wg03}}{\les} \frac{1}{r^{\frac{d-1}{2}}}\lt(\int_{\partial B_R} \hat{g}_\pm^2\rt)^{\frac{3}{4}}W_{\partial B_R} (\bar f_\pm,\hat g_\pm)+ \frac{1}{r} \lt(\frac{1}{r^{d-1}}\int_{\partial B_R} \hat{g}^2_\pm\rt)^{\frac{1}{2}}W^2_{\partial B_R}(\bar f_\pm,\hat g_\pm)\\
 &\stackrel{\eqref{ma9}\&\eqref{ma10}}{\les} \frac{1}{r^{\frac{d-1}{2}}} M^{\frac{1}{2}}(E+D)^{\frac{5}{4}}+\frac{1}{r^{\frac{d+1}{2}}} M (E+D)^{\frac{3}{2}}\les \frac{1}{r^{\frac{d+1}{2}}}  M^{\frac{1}{2}} (E+D)^{\frac{5}{4}}.
\end{align*}

\medskip
{\sc Step 1.3.} We finally estimate $\int_{B_R} |\nabla \phi^r|^2(d\bar \rho-dx)$. Let $\zeta:=|\nabla \phi^r|^2$ (which is smooth) so that 
\begin{align*}
 &\lt|\int_{B_R} \zeta (d \bar \rho- dx)\rt|\\
 &\le \lt|\int_{B_R} \zeta (d\bar \rho- \kappa dx)\rt|+ |\kappa-1|\int_{B_R} \zeta\\
 &\stackrel{\eqref{wg80}}{\le} \lt(\int_{B_R} |\nabla \zeta|^2\rt)^{\frac{1}{2}} W_{B_R}(\bar \rho,\kappa)+ \frac{1}{2}\sup_{B_R} |\nabla^2 \zeta| W_{B_R}^2(\bar \rho,\kappa)+|\kappa-1|\int_{B_R} \zeta\\
 &\stackrel{\eqref{wg23}}{\les} \lt(\int_{B_R} |\nabla \zeta|^2\rt)^{\frac{1}{2}} (E+D)^{\frac{1}{2}}+\sup_{B_R} |\nabla^2 \zeta| (E+D)+\lt(\int_{B_R} \zeta\rt)(E+D)^{\frac{1}{2}}.
\end{align*}
Since
\begin{align*}
 \int_{B_R} \zeta &=\int_{B_R}|\nabla \phi^r|^2\stackrel{\eqref{wg12}}{\les} \int_{\partial B_R} \hat{g}^2\stackrel{\eqref{ma10}}{\les} E+D,\\[6pt]
\lt(\int_{B_R} |\nabla \zeta|^2\rt)^{\frac{1}{2}}&\les \sup_{B_R} |\nabla^2 \phi^r| \lt(\int_{B_R} |\nabla \phi^r|^2\rt)^{\frac{1}{2}}\\
&\stackrel{\eqref{wg03}\&\eqref{wg12}}{\les} \frac{1}{r^{\frac{d+1}{2}}} \int_{\partial B_R} \hat{g}^2\stackrel{\eqref{ma10}}{\les}\frac{1}{r^{\frac{d+1}{2}}}(E+D),\\[6pt]
&\textrm{and}\\[6pt]
\sup_{B_R} |\nabla^2\zeta|&\les  \sup_{B_R} |\nabla \phi^r| \sup_{B_R} |\nabla^3 \phi^r|+ \sup_{B_R} |\nabla^2 \phi^r|^2\\
&\stackrel{\eqref{wg03}}{\les} \frac{1}{r^{d+1}} \int_{\partial B_R} \hat{g}^2\stackrel{\eqref{ma10}}{\les} \frac{1}{r^{d+1}} (E+D),
\end{align*}
we conclude that  
\begin{multline*}
 \lt|\int_{B_R} |\nabla \phi^r|^2 (d \bar \rho- dx)\rt|\les \frac{1}{r^{\frac{d+1}{2}}} (E+D)^{\frac{3}{2}}+\frac{1}{r^{d+1}}(E+D)^2+(E+D)^{\frac{3}{2}}\\
 \les \frac{1}{r^{\frac{d+1}{2}}} (E+D)^{\frac{3}{2}}+\frac{1}{r^{d+1}}(E+D)^2.
\end{multline*}
Note that because of $M=(E+D)^{\frac{1}{d+2}}$, also the first right-hand side term is contained in  the last right-hand side term of \eqref{ma17}.\\

Combining {\sc Step 1.1, Step 1.2} and { \sc Step 1.3} together, we obtain \eqref{ma17}.\\

\medskip
{\sc Step 2.} We now pass from $\phi^r$ to $\phi$.  
 Using the identity
 \[
  \frac{1}{\rho}|j-\rho\nabla \phi|^2-\frac{1}{\rho}|j-\rho\nabla \phi^r|^2=2 j\cdot \nabla( \phi-\phi^r)+\rho(|\nabla \phi|^2-|\nabla \phi^r|^2), 
 \]
and Cauchy-Schwarz's inequality, we first obtain 
\begin{align*}
 \lefteqn{\int_{B_2}\int_0^1 \frac{1}{\rho}|j-\rho \nabla \phi|^2-\int_{B_2}\int_0^1 \frac{1}{\rho}|j-\rho \nabla \phi^r|^2}\\
 &\les \lt(\int_{B_2} d\bar \rho \int_{B_2}\int_0^1 \frac{1}{\rho}|j|^2\rt)^{\frac{1}{2}} \sup_{B_2} |\nabla(\phi-\phi^r)|\\
 &\qquad \qquad +  \lt(\sup_{B_2} |\nabla (\phi-\phi^r)|\rt)\lt(\sup_{B_2} |\nabla \phi|+|\nabla \phi^r|\rt) \int_{B_2} d\bar \rho\\
 &\stackrel{\eqref{ma42}\&\eqref{wg17}\&\eqref{wg12}}{\les} r E^{\frac{1}{2}} \lt(\int_{\partial B_R} \hat{g}^2\rt)^{\frac{1}{2}}+r \int_{\partial B_R} \hat{g}^2\\
  &\stackrel{\eqref{ma10}}{\les} r(E+D).
\end{align*}
Now by \eqref{defphi} and \eqref{wg01}, we also  have
\begin{align*}
 \int_{B_R} |\nabla \phi|^2-\int_{B_R} |\nabla \phi^r|^2&=\int_{\partial B_R} (\phi+\phi^r)(\hat g-\hat{g}_r)\\
 &= \int_{\partial B_R} \lt[(\phi+\phi^r)-(\phi+\phi^r)_r\rt] \hat{g}\\
 &\les r \lt(\int_{\partial B_R} |\nabla(\phi+\phi^r)|^2\rt)^{\frac{1}{2}} \lt(\int_{\partial B_R} \hat{g}^2\rt)^{\frac{1}{2}}\\
 &\stackrel{\eqref{wg15}\&\eqref{ma10}}{\les} r (E+D),
\end{align*}
so that combined with \eqref{ma17} and the obvious $\int_{B_2}\int_0^1 \frac{1}{\rho}|j-\rho\nabla \phi^r|^2\le \int_{B_R}\int_0^1 \frac{1}{\rho}|j-\rho\nabla \phi^r|^2$, we conclude
\begin{align*}
 \lefteqn{ \int_{B_2}\int_0^1 \frac{1}{\rho}|j-\rho\nabla \phi|^2- \lt(\int_{B_R}\int_0^1 \frac{1}{\rho}|j|^2-\int_{B_R}|\nabla \phi|^2\rt)}\\
 & = \int_{B_2}\int_0^1 \frac{1}{\rho}|j-\rho\nabla \phi|^2-\int_{B_R}\int_0^1 \frac{1}{\rho}|j-\rho\nabla \phi^r|^2\\
 &\qquad +\int_{B_R}|\nabla \phi|^2-\int_{B_R} |\nabla \phi^r|^2\\
  &\qquad +\int_{B_R}\int_0^1 \frac{1}{\rho}|j-\rho\nabla \phi^r|^2-\lt(\int_{B_R}\int_0^1 \frac{1}{\rho}|j|^2-\int_{B_R}|\nabla \phi^r|^2\rt)\\
  &\les r(E+D)+r E +\frac{1}{r} D +\frac{1}{r^{d+2}} D^2 +\frac{1}{r^{d+1}} E^2+\frac{1}{r^{\frac{d+1}{2}}}  M^{\frac{1}{2}} (E+D)^{\frac{5}{4}}\\
  &\les r E +\frac{1}{r} D +\frac{1}{r^{d+2}} D^2 +\frac{1}{r^{d+1}} E^2+\frac{1}{r^{\frac{d+1}{2}}}  M^{\frac{1}{2}} (E+D)^{\frac{5}{4}}.
\end{align*}
Therefore, for every fixed $r>0$, provided $E+D$ is small enough we see that \eqref{io15} holds. 
\end{proof}

\subsubsection{The construction}\label{sec:construct}
We now turn to the construction part of the proof i.e. Lemma \ref{Lconintro}, 
\begin{lem1.8}
 For every $0<\tau\ll1$, there exist $\eps(\tau)>0$ and $C(\tau)>0$ such that if $E+D\le \eps$, then there exists a density-flux pair $(\tilde \rho, \tilde j)$ satisfying \eqref{locconteqcomp} and such that 
 \begin{equation}\label{io31}
  \int\int_0^1 \frac{1}{\tilde \rho}|\tilde{j}|^2-\int_{B_R}|\nabla \phi|^2 \le \tau E +C D.
 \end{equation}
 
\end{lem1.8}
\begin{proof}
 We split the construction and
its estimate into the following steps:
\begin{itemize}

\item {\sc Step 1:} the construction in the initial layer $B_R\times (0,\tau)$, using \eqref{ma14} to connect the measure $\mu'$ defined in \eqref{io07bis} to the constant $\kappa_{\mu'}$ and
in the final layer $B_R\times (1-\tau,1)$, connecting in a similar way the constant $\kappa_{\lambda'}$ to the measure $\lambda'$ (see Figure \ref{fig:init}), 
\item{\sc Step 2:} the construction in the boundary layer $(B_R\backslash\bar B_{R-r})\times(\tau,1-\tau)$ (see Figure \ref{fig:lead}),  to connect the constant-in-$t$ flux (cf.\ \eqref{io10}) $\hat{g}'$   to (cf.\ \eqref{io11} and \eqref{io10}) 
\begin{equation}\label{io04}
 \tilde{g}':=\begin{cases}
              \frac{1}{\tau} \hat{g}'_+ & \textrm{for } t\in (\tau,2\tau)\\
              -\frac{1}{\tau} \hat{g}'_- &\textrm{for } t\in (1-2\tau,1-\tau),
             \end{cases}
\end{equation}
\item{\sc Step 3:} the leading order construction in $B_R\times(\tau,1-\tau)$ to connect the constant  $\kappa_{\mu'}$ to the constant $\kappa_{\lambda'}$ (see Figure \ref{fig:lead}), using $\nabla \phi'$ defined by 
\begin{equation}\label{io06}
 \Delta \phi'= c \ \textrm{ in } B_R \qquad \textrm{ and } \qquad  \nu\cdot \nabla \phi'= \hat{g}' \ \textrm{ on } \partial B_R,
\end{equation}
as flux,
\item{\sc Step 4:} the construction in the boundary $\partial B_R\times(\tau,1-\tau)$, to connect the flux $\tilde{g}'$ to the flux $f'$ defined in \eqref{io05} (see Figure \ref{fig:boundary}),
\item{\sc Step 5:} the estimate of cost of kept trajectories (see Figure \ref{fig:kept}),
\item{\sc Step 6:} the passage from $\phi'$ to $\phi$,
\item{\sc Step 7:} the collection of the error terms.
\end{itemize}
Before starting, we notice for further reference that 
\begin{equation}\label{wg60}
 \int_{B_R} \lt(|\nabla \phi'|^2+|\nabla \phi|^2\rt)\stackrel{\eqref{estim:elliptic}}{\les} \int_{\partial B_R} \lt((\hat{g}')^2+\hat{g}^2\rt)\stackrel{\eqref{ma11}\&\eqref{ma10}}{\les}E+D.
\end{equation}

\medskip
{\sc Step 1.} {\sc Construction in the initial and final layers}.
By the Eulerian formulation of the Wasserstein distance (cf.\ \eqref{ma14}), there exists $(\rhoin,\jin)$ supported in $\bar B_R\times[0,\tau]$  such that (see Figure \ref{fig:init})
\begin{figure}\begin{center}
 \resizebox{7.cm}{!}{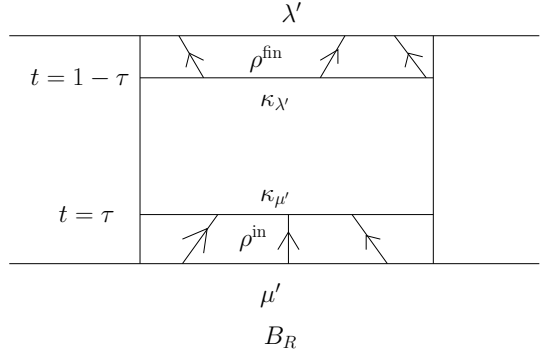} 
   \caption{Initial and final layers.} \label{fig:init}
 \end{center}
 \end{figure}
\begin{equation}\label{ma18}
 \int_{B_R}\int_0^1 \partial_t \zeta d \rhoin+\nabla \zeta\cdot d\jin=\int_{B_R} (\zeta_{\tau} \kappa_{\mu'} dx -\zeta_0 d\mu')
\end{equation}
and 
\begin{equation}\label{ma19}
 \int_{B_R}\int_0^1 \frac{1}{\rhoin}|\jin|^2 = \frac{1}{\tau} W^2_{B_R}(\mu',\kappa_{\mu'})\stackrel{\eqref{ma14}}{\les} \tau E +\frac{1}{\tau} D.
\end{equation}

Similarly, there exists $(\rhofin,\jfin)$ supported in $B_R\times[1-\tau,1]$ such that  (see Figure \ref{fig:init})
\begin{equation}\label{ma20}
 \int_{B_R}\int_0^1 \partial_t \zeta d \rhofin+\nabla \zeta\cdot d\jfin=\int_{B_R} (\zeta_1 d\lambda' -\zeta_{1-\tau} \kappa_{\lambda'} dx)
\end{equation}
and 
\begin{equation}\label{ma21}
 \int_{B_R}\int_0^1 \frac{1}{\rhofin}|\jfin|^2\les  \tau E +\frac{1}{\tau} D.
\end{equation}

\medskip
{\sc Step 2.} {\sc Construction in the boundary layer}. By definition \eqref{io04} of $\tilde{g}'$, we have for every $x\in \partial B_R$, 
\[
 \int_\tau^{1-\tau} \tilde{g}'(t,x) dt=\int_{\tau}^{1-\tau}\frac{1}{1-2\tau} \hat{g}'(x) dt
\]
 and 
 \[
 \int_{\partial B_R} \int_\tau^{1-\tau} \lt(\tilde{g}'-\frac{1}{1-2\tau} \hat{g}'\rt)^2\les \frac{1}{\tau} \int_{\partial B_R} (\hat{g}')^2\stackrel{\eqref{ma11}}{\les} \frac{1}{\tau}(E+D).
 \]
Therefore, applying \cite[Lem. 2.4]{GO} to $\tilde{g}'-\frac{1}{1-2\tau} \hat{g}'$ and choosing $r$ to be a large but order one multiple of $\lt(\int_{B_R} \int_\tau^{1-\tau} \lt(\tilde{g}'-\frac{1}{1-2\tau} \hat{g}'\rt)^2\rt)^{\frac{1}{d+1}}$, we get the existence of a  
pair $(s,q)$  supported in $B_R\backslash \bar{B}_{R-r}\times(\tau,1-\tau)$ (see Figure \ref{fig:lead}) with $|s|\le \frac{1}{2}$, such that 
\begin{equation}\label{ma23}
 \int_{B_R}\int_{\tau}^{1-\tau} \partial_t \zeta ds+\nabla \zeta \cdot dq= \int_{\partial B_R}\int_{\tau}^{1-\tau} \zeta (\tilde{g}'-\frac{1}{1-2\tau} \hat{g}')
\end{equation}
and
\begin{equation}\label{wg67}
 \int_{B_R\backslash B_{R-r}}\int_{\tau}^{1-\tau} |q|^2\les  r \int_{\partial B_R}\int_{\tau}^{1-\tau} \lt(\tilde{g}'-\frac{1}{1-2\tau} \hat{g}'\rt)^2\les \lt(\frac{1}{\tau}(E+D)\rt)^{\frac{d+2}{d+1}}.
\end{equation}

\medskip
{\sc Step 3.} {\sc Leading order construction}.
For $t\in [\tau,1-\tau]$ we define (see Figure \ref{fig:lead})
\begin{figure}\begin{center}
 \resizebox{10.cm}{!}{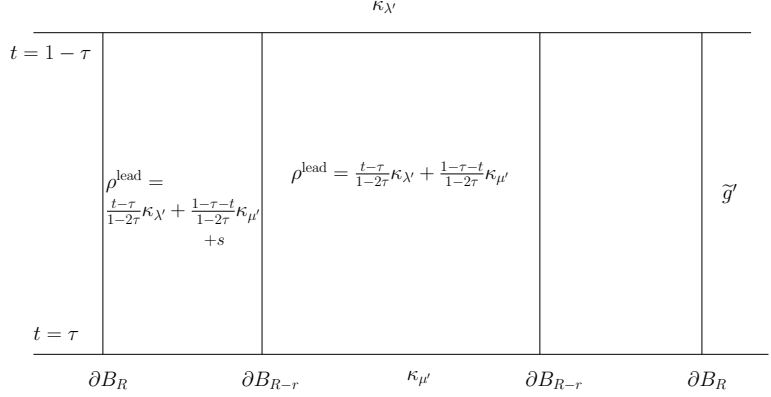} 
   \caption{Leading order construction.} \label{fig:lead}
 \end{center}
 \end{figure}
\begin{equation}\label{rhobulk}
 \rho^\lead:=\frac{t-\tau}{1-2\tau} \kappa_{\lambda'}+\frac{1-\tau-t}{1-2\tau} \kappa_{\mu'} +s, \qquad j^\lead:=
                                                                                                            \frac{1}{1-2\tau}\nabla \phi' +q, 
\end{equation}
both restricted to $B_R\times [\tau,1-\tau]$. Notice that by \eqref{ma5} and \eqref{ma4}, $\Omega_{0}\backslash \Omega_{1}= \Omega'_{+}\backslash \Omega'_{-}=\{|X(0)|<R,|X(1)|\geq R, t_+ \geq 3\tau\}$ and likewise
$\Omega_{1}\backslash \Omega_{0}= \Omega'_{-}\backslash \Omega'_{+},$ and therefore 
\begin{align*}
 |B_R|(\kappa_{\mu'}-\kappa_{\lambda'})= \int_{B_R} (d\mu'-d\lambda')\stackrel{\eqref{io07bis}}{=} &\P(\Omega_{0})-\P(\Omega_{1})\\
  = &\P(\Omega'_{+})-\P(\Omega'_{-})\stackrel{\eqref{io11}}{=}\int_{\partial B_R} (\hat{g}'_+-\hat{g}'_-),
\end{align*}
so that using \eqref{ma23} and \eqref{io06}
\begin{equation}\label{ma22}
 \int_{B_R}\int_{\tau}^{1-\tau} \partial_t \zeta d \rho^\lead+\nabla \zeta\cdot dj^\lead=\int_{B_R} (\zeta_{1-\tau} \kappa_{\lambda'}-\zeta_{\tau} \kappa_{\mu'}) +\int_{\partial B_R}\int_{\tau}^{1-\tau}  \zeta \tilde{g}'. 
\end{equation}
Using that $(s,q)$ is supported in $B_R\backslash \bar{B}_{R-r}\times(\tau,1-\tau)$ with $|s|\le \frac{1}{2}$, the cost is majorized by 
\begin{multline*}
 \int_{B_R}\int_{\tau}^{1-\tau} \frac{1}{\rho^\lead}|j^\lead|^2-\frac{1}{1-2\tau}\int_{B_r}|\nabla \phi'|^2\\
 \les    \int_{B_R\backslash B_{R-r}} |\nabla \phi'|^2+\int_\tau^{1-\tau}\int_{B_R\backslash B_{R-r}}|q|^2.
 \end{multline*}
Therefore, using \eqref{wg60}, \eqref{ma14},  \eqref{wg15}, and \eqref{wg67} and recalling the choice of $r$ (see Step 2) we find 
\begin{multline}\label{wg68}
 \int_{B_R}\int_{\tau}^{1-\tau} \frac{1}{\rho^\lead}|j^\lead|^2-\int_{B_R} |\nabla \phi'|^2\\
 \les (\tau+(\tau^2 E+D)^{\frac{1}{2}})(E+D)+ r(E+D)+\lt(\frac{1}{\tau}(E+D)\rt)^{\frac{d+2}{d+1}}\\ \les \tau E + \frac{1}{\tau} D+\lt(\frac{1}{\tau}(E+D)\rt)^{\frac{d+2}{d+1}}.
\end{multline}
 
\medskip
{\sc Step 4.} {\sc Construction in the boundary}. We will separate the construction for the terms with the $+$-sign and with the $-$-sign. We start with the $+$-sign.
Notice that by definition \eqref{io04}, the positive part of  $\tilde{g}'$, denoted by $\tilde{g}'_+$, satisfies $ \int_0^1 \tilde{g}'_+=\hat{g}_+'$ and $\spt \tilde g'_+\subset \partial B_R\times (\tau,2\tau)$.
By definition \eqref{io05}, we have $ \int_0^1 {f}'_+=\bar{f}'_+$  and $\spt f'_+\subset \partial B_R \times (3\tau,1)$. Hence, we may apply Lemma \ref{lem:consbound} to $f=\tilde{g}'_+$ and $g=f'_+$ to obtain a density-flux pair 
$(\rho^\bdr_+,j^\bdr_+)$  in $\partial B_R \times (0,1)$ with (see Figure \ref{fig:boundary})
\begin{figure}\begin{center}
 \resizebox{7.cm}{!}{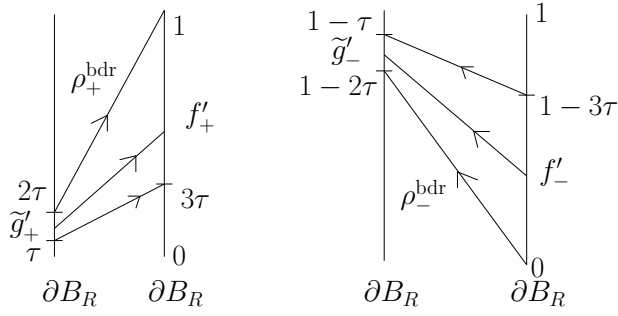} 
   \caption{Construction in the boundary.} \label{fig:boundary}
 \end{center}
 \end{figure}
\[
 \int_{\partial B_R} \int_0^1 \partial_t \zeta d\rho^\bdr_+ +\nabla \zeta\cdot dj^{\bdr}_+=\int_{\partial B_R} \int_0^1\zeta d( f'_+-\tilde{g}_+') 
\]
and 
\[
 \int_{\partial B_R}\int_0^1 \frac{1}{\rho^\bdr_+}|j^{\bdr}_+|^2\le \frac{1}{\tau} W_{\partial B_R}^2 (\hat{g}_+'\bar{f}'_+)\stackrel{\eqref{ma12}}{\les} \frac{M}{\tau} (E+D).
\]
Similarly, applying Lemma \ref{lem:consbound} to $f=f'_-$ and $g=\tilde{g}'_-$, we obtain a density-flux pair 
$(\rho^\bdr_-,j^\bdr_-)$  in $\partial B_R \times (0,1)$ with (see Figure \ref{fig:boundary}) 
\[
 \int_{\partial B_R} \int_0^1 \partial_t \zeta d\rho^\bdr_- +\nabla \zeta\cdot dj^{\bdr}_-=\int_{\partial B_R} \int_0^1\zeta d( \tilde{g}_-'-f'_-) 
\]
and 
\[
 \int_{\partial B_R}\int_0^1 \frac{1}{\rho^\bdr_-}|j^{\bdr}_-|^2\les \frac{M}{\tau} (E+D).
\]
Letting $(\rho^\bdr,j^\bdr):=(\rho^\bdr_++\rho^\bdr_-,j^\bdr_++j^\bdr_-)$, we conclude that 
\begin{equation}\label{ma29}
 \int_{\partial B_R} \int_0^1 \partial_t \zeta d\rho^\bdr +\nabla\zeta\cdot dj^{\bdr}=\int_{\partial B_R} \int_0^1\zeta d( f'-\tilde{g}') 
\end{equation}
and 
\begin{equation}\label{ma30}
 \int_{\partial B_R}\int_0^1 \frac{1}{\rho^\bdr}|j^{\bdr}|^2\les \frac{M}{\tau} (E+D).
\end{equation}

\medskip

{\sc Step 5.} {\sc Cost of kept trajectories}. Recalling the definition \eqref{io01} of $\del\Omega$, we let for $t\in (0,1)$ in line with \eqref{io18} (see Figure \ref{fig:kept}),
\begin{figure}\begin{center}
 \resizebox{7.cm}{!}{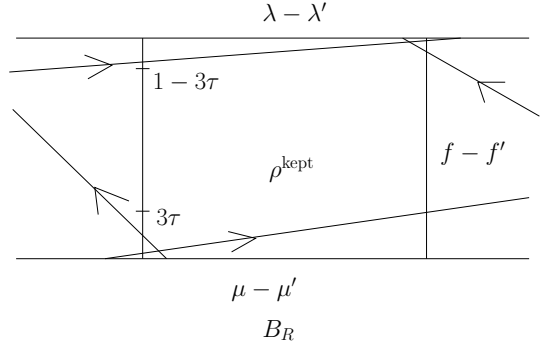} 
   \caption{Kept trajectories.} \label{fig:kept}
 \end{center}
 \end{figure}
\begin{align*}
 \int \zeta d\rho^{\kept}_t&=\int_{\del\Omega} I(|X(t)|<R) \zeta (X(t)) d\P, \\
 \int \xi\cdot dj_t^{\kept}&=\int_{\del\Omega} I(|X(t)|< R)\xi(X(t))\cdot \dot{X}(t) d\P
\end{align*}
so that by definition \eqref{io05} of $f'_R$ and \eqref{io07bis} of $\mu'$ and $\lambda'$,
\begin{align}\label{ma31}
 \int_{ B_R} \int_0^1 & \partial_t \zeta d\rho^\kept +\nabla \zeta\cdot dj^{\kept}\nonumber \\&=\int_{B_R} \zeta_1 d(\lambda-\lambda')+\int_{B_R} \zeta_0 d(\mu-\mu') +\int_{\partial B_R}\int_0^1 \zeta d (f-f').
\end{align}
By the representation \eqref{dualBB}, we may bound
\begin{multline}\label{ma32}
 \frac{1}{2}\int_{ B_R} \int_0^1\frac{1}{\rho^\kept}|j^\kept|^2=\sup_{\xi\in C^0_c(B_R\times [0,1])^d} \int \xi\cdot dj-\frac{|\xi|^2}{2} d\rho\\
 =\sup_{\xi\in C^0_c(B_R\times [0,1])^d} \int_{\del\Omega}\int_0^1 I(|X(t)|<R) \lt(\xi \cdot \dot{X}-\frac{|\xi|^2}{2}\rt) dt d\P\\
 \le \frac{1}{2}\int_{\del\Omega}\int_{t_-}^{t_+} |\dot{X}|^2  dt d\P \les \tau E,
\end{multline}
where in the last estimate we used that trajectories are straight and that for $X\in \del\Omega$, $|t_+-t_-|\le 3\tau$.
\medskip

{\sc Step 6.} {\sc The passage from $\phi'$ to $\phi$}. Since $\int_{B_R} \phi=\int_{B_R} \phi'=0$, using the definitions \eqref{defphi}, \eqref{io06} and integration by parts, we obtain
\[
 \int_{B_R} |\nabla \phi'|^2-\int_{B_R} |\nabla \phi|^2=\int_{\partial B_R} (\phi'+\phi)(\hat{g}'-\hat{g}).
\]
Therefore, by Cauchy-Schwarz and the Poincar\'e-trace estimate,
\begin{align}
\lt| \int_{B_R} |\nabla \phi'|^2-\int_{B_R} |\nabla \phi|^2\rt|&\les \lt( \lt(\int_{B_R} |\nabla \phi'|^2+\int_{B_R} |\nabla \phi|^2\rt) \int_{\partial B_R} (\hat{g}'-\hat{g})^2\rt)^{\frac{1}{2}}\nonumber \\
&\stackrel{\eqref{wg60}}{\les}  \lt((E+D)\int_{\partial B_R} (\hat{g}'_{\pm}-\hat{g}_{\pm})^2\rt)^{\frac{1}{2}} \nonumber \\
&\stackrel{\eqref{ma13}}{\les}\lt( (E+D)(\tau^2 E+ D)\rt)^{\frac{1}{2}} \les  \tau E +\frac{1}{\tau} D \label{ma33}.
\end{align}

\medskip

{\sc Step 7.} {\sc Collecting error terms}. We use 
\[(\tilde{\rho},\tilde{j}):= (\rhoin+\rhofin+\rho^\lead+\rho^\bdr+\rho^\kept,\jin+\jfin+j^\lead+j^\bdr+j^\kept)\]
as our competitor. Putting together \eqref{ma18}, \eqref{ma20}, \eqref{ma22},  \eqref{ma29} and \eqref{ma31}, we see that the continuity equation \eqref{locconteqcomp} is satisfied. 
Putting together estimates \eqref{ma19} and  \eqref{ma21} coming from the construction in the final and initial layers, \eqref{wg68} coming from the leading order construction,  \eqref{ma30} coming from transporting in $\partial B_R$, \eqref{ma32} counting the contribution from the kept trajectories and
\eqref{ma33} estimating the error  of passing from $\phi'$ to $\phi$, we get
\begin{align*}
 \lefteqn{\int_{B_R} \int_0^1 \frac{1}{\tilde{\rho}}|\tilde{j}|^2-\int_{B_R} |\nabla \phi|^2}\\
 &\les \tau E +\frac{1}{\tau} D+\lt(\frac{1}{\tau}(E+D)\rt)^{\frac{d+2}{d+1}}+\frac{M}{\tau}(E+D)+ \tau E +\frac{1}{\tau} D\\
 &\les  \lt(\tau +\frac{M}{\tau}\rt) E +\lt(\frac{1}{\tau}(E+D)\rt)^{\frac{d+2}{d+1}}+ \frac1\tau D,
\end{align*}
from which \eqref{io31} follows if $E+D$ is small enough.
\end{proof}

\subsection{Proof of the harmonic approximation result}
 We may now finally prove our main harmonic approximation result in its Lagrangian version, Theorem \ref{theo:harmonicLagintro} 
 \begin{theo1.4}
  For every $\tau>0$, there exist positive constants $\eps(\tau)$ and $C(\tau)$ such that if  $\pi$ is an optimal transference plan between  two measures $\mu$ and $\lambda$ with $E+D\le \eps$, then there exists 
  a radius $R\in (3,4)$ such that if $\Phi$ is given by 
  \[
   \Delta \Phi=c \ \textrm{ in } B_R \qquad \textrm{ and } \qquad \nu\cdot \nabla \Phi=  \nu\cdot \bar j \ \textrm{ on } \partial B_R,
  \]
then 
\begin{equation}\label{wg74}
 \int_{(B_1\times \R^d)\cup(\R^d\times B_1)} |x-y+\nabla \Phi(x)|^2 d\pi\le \tau E +C D
\end{equation}
and 
\begin{equation}\label{ma37}
 \sup_{B_2} |\nabla^2 \Phi|^2+ \sup_{B_2} |\nabla \Phi|^2\les E+D.
\end{equation}

 \end{theo1.4}
\begin{proof}
 The proof is almost identical to the proof of \cite[Prop. 3.6]{GO} but we give it for the reader's convenience.
Let us first notice that \eqref{ma37} is given by \eqref{estimPhi}. We thus only need to derive \eqref{wg74} from \eqref{ma36}. 
 Since $\nabla \Phi$ is smooth in $B_2$, we obtain as for \eqref{je15}, that 
\[
 \int \int_0^1 I(|ty+(1-t)x|<2)  |x-y+\nabla \Phi(ty+(1-t)x)|^2 d\pi dt=\int_{B_2} \int_0^1 \frac{1}{\rho}|j-\rho\nabla\Phi|^2.
\]
Recalling that by \eqref{ma39} of Lemma \ref{Linf}, 
if $x\in B_1$ or $y\in B_1$ and $(x,y)\in \spt \pi$ then $ty+(1-t)x\in B_{2}$, we may therefore estimate
\begin{align*}
 \lefteqn{\int_{(B_1\times \R^d)\cup(\R^d\times B_1)} |x-y+\nabla \Phi(x)|^2 d\pi}\\
 &\les  \int_{(B_1\times \R^d)\cup(\R^d\times B_1)}\int_0^1 |x-y+\nabla \Phi(ty+(1-t)x)|^2 d\pi dt\\
 &\qquad + \int_{(B_1\times \R^d)\cup(\R^d\times B_1)}\int_0^1 |\nabla \Phi (x) -\nabla \Phi(ty+(1-t)x)|^2 d\pi dt\\
 &\les \int_{B_2}\int_0^1 \frac{1}{\rho}|j-\rho \nabla \Phi|^2+ \sup_{B_{2}} |\nabla^2 \Phi|^2 \int_{(B_1\times \R^d)\cup(\R^d\times B_1)} |y-x|^2 d\pi\\
 &\stackrel{\eqref{ma36}\& \eqref{estimPhi}}{\les} \tau E +C(\tau) D +(E+D) E.
\end{align*}
 This concludes the proof of \eqref{wg74}, provided $E+D$ is small enough.
 \end{proof}

 \section{Quantitative bounds on the displacement}
 In this section we prove Theorem \ref{theo:main}.  For a given measure $\pi$ on $ \R^d\times\R^d$
 with cyclically monotone support and $\barR>0$, we let $\mu:=\pi_1\restr B_{\barR}$ and assume from now on that $\pi_2\restr B_{\barR}=dx$.
 \medskip
 
 Let us recall that we fixed a nonnegative rate function $\beta$ satisfying the following hypothesis (see \eqref{hypbeta}): 
it is increasing, $R\mapsto\frac{\beta(R)}{R}$ is decreasing, and  
 there exists a constant $C_\beta>0$ such that 
\begin{equation}\label{hypbetamain}
 \sum_{\ell\geq0} \frac{\beta(2^\ell R)}{2^\ell R}\le C_\beta \frac{\beta(R)}{R} \quad \textrm{ for every }  R\ge1 \qquad \textrm{and} \qquad \beta(1)\le C_\beta.
\end{equation} 
Note that this in particular implies for any $\ell\in\N$ and $R>2^\ell$
\begin{equation}\label{hypCampa}
 2^{-\ell} \beta(R)\les \beta( 2^{-\ell} R).
\end{equation}
Moreover, since $R\mapsto\frac{\beta(R)}{R}$ is decreasing, for  $R\gg 1$, there holds
\begin{equation}\label{ma153}
 \beta(R)\ll R^2.
\end{equation}

 \subsection{The Campanato iteration}
 We first prove that as a consequence of Theorem \ref{theo:harmonicLagintro} and a Campanato iteration, we can obtain an $L^2$ bound on the displacement, that is Proposition \ref{theo:campintro}, which we now restate.  
To this aim, let us first recall some notation. For a sequence of approximately geometric radii i.e.
\begin{equation}\label{so30}
 \barR\ge R_0\ge\cdots\ge R_K\ge C \qquad \textrm{with } \ R_{k-1}\ge 2 R_k\ge \frac{1}{C} R_{k-1}
\end{equation}
with $\barR\sim R_0$ and $R_K\sim 1$, we put
\[
 E_k=\frac{1}{|B_{6R_k}|}\int_{(B_{6R_k}\times \R^d)\cup(\R^d\times B_{6R_k})} |x-y|^2 d\pi_k
\]
and consider $\pi_k$ the coupling defined recursively by $\pi_0:=\pi$ and 
\begin{equation}\label{so34}
 \pi_k=({\rm id},{\rm id}-\nabla \Phi_{k-1}(0))\# \pi_{k-1},
\end{equation}
where $\Phi_k$ solves the Poisson equation
\begin{equation}\label{so33}
 \Delta \Phi_k=c \ \textrm{ in } B_{R_k} \qquad \textrm{ and } \qquad \nu\cdot\nabla\Phi_k=\nu\cdot \bar{j}_k \ \textrm{ on } \partial B_{R_k}.
\end{equation}
As in \eqref{defrhoj}, the flux $j_k$ is the Eulerian flux related to $\pi_k$, that is
\begin{equation}\label{so08}
 \int \xi\cdot j_{k,t}=\int \xi((1-t)x+ty)\cdot(y-x) d\pi_k,
\end{equation}
from which $\bar j_k$ is obtained by integrating in time (see \eqref{defbar}). Let us also recall that from \eqref{eq:divjbar} and the invariance of the Lebesgue measure under translation,
\begin{equation}\label{so02}
 \nabla\cdot \bar j_k=\mu-1 \qquad \textrm{ in } B_{R_k}.
\end{equation}

 \begin{prop1.9}
Assume that 
 \begin{equation}\label{hypstartcampa}
  \frac{1}{|B_{\barR}|}\int_{(B_{\barR}\times\R^d)\cup(\R^d\times B_{\barR})}|x-y|^2 d\pi \le \beta(\barR)
 \end{equation}
and 
\begin{equation}\label{hypdataCampa}
 \frac{1}{|B_R|} W^2_{B_R}(\mu,\kappa)\le \beta(R) \qquad \forall R\in [1,\barR].
\end{equation}
Then, there exists a sequence of approximately geometric radii $(R_k)_{0\leq k \leq K}$ such that 
\begin{equation}\label{so01}
 E_k\les \beta(R_k),
\end{equation}
and 
\begin{equation}\label{so39}
 |\nabla \Phi_k(0)|^2 \les \beta(R_k).
\end{equation}

 \end{prop1.9}
\begin{proof}
  Let $6\barR_0=\barR$ and  for $\theta\ll1$  dyadic to be fixed later, let $6\barR_k:= \theta^k\barR_0$. We prove by induction that there exists a constant $C_{E}>0$ only depending on $d$ and $\beta$ such that 
  for $k\in[0,K]$, there exists $R_k\in (3\barR_k, 4\barR_k)$ such that defining
 $\Phi_{k-1}$ through \eqref{so33} and then $\pi_k$ as in \eqref{so34}, we have that the support of $\pi_k$ is cyclically monotone,
 that the first  marginal of $\pi_k$ in $B_{6 \barR_k}$ is $\mu$, that its second marginal in $B_{6 \barR_k}$ is the Lebesgue measure, and that both   \eqref{so39} and 
 \begin{equation}\label{ma75}
 \bar{E}_k:= \frac{1}{|B_{6 \barR_k}|}\int_{(B_{6\barR_k}\times\R^d)\cup(\R^d\times B_{6\barR_k})} |x-y|^2 d\pi_k\le C_{E} \beta(\barR_k)
 \end{equation}
 hold. This would upgrade to \eqref{so01}. Indeed, by assumption \eqref{hypstartcampa}, \eqref{so01} holds for $k=0$ and if $k\ge 1$, we have by \eqref{so34},
\begin{align*}
 \lefteqn{\frac{1}{|B_{6R_k}|}\int_{(B_{6R_k}\times \R^d)\cup(\R^d\times B_{6R_k})} |x-y|^2 d\pi_k}\\
 &=\frac{1}{|B_{6R_k}|}\int_{(B_{6R_k}\times \R^d)\cup(\R^d\times B_{6R_k}(\nabla \Phi_{k-1}(0)))} |x-y+\nabla\Phi_{k-1}(0)|^2 d\pi_{k-1}\\
 &\les \frac{1}{|B_{6R_k}|}\int_{(B_{7R_k}\times \R^d)\cup(\R^d\times B_{7R_k})} |x-y|^2 d\pi_{k-1}+ |\nabla\Phi_{k-1}(0)|^2,
\end{align*}
where we used that by \eqref{so39} and \eqref{ma153}  $ B_{6R_k}(\nabla \Phi_{k-1}(0))\subset B_{7R_k}$ and that by the $L^\infty$ bound \eqref{wg31} which we may apply since $E_{k-1}\leq \beta(R_{k-1}) \ll R_{k-1}^2 \sim R_k^2$, $\spt \pi_{k-1}\cap (B_{7R_k}\times \R^d)\cup(\R^d\times B_{7R_k})\subset \R^d\times B_{9 R_k}$.
Therefore, if $\theta$ is sufficiently small such that $7 R_k\le 6 \barR_{k-1}$, using \eqref{ma75} and \eqref{so39}, we would conclude the proof of \eqref{so01}.
\medskip 

We now turn to the inductive argument. Let $0<\tau\ll1$ be such that Theorem \ref{theo:harmonicLagintro} applies and define for $k$ such that $6\barR_k\geq 1$ (recall Remark \ref{rem:DataLebesgueintro})
\begin{equation}\label{ma53}
 \bar{D}_k:= \frac{1}{|B_{6\barR_k}|} W^2_{B_{6\barR_k}}(\mu,\kappa) \stackrel{\eqref{hypstartcampa}}{\leq} \beta(6\barR_k)\le 6 \beta(\barR_k).
\end{equation}
By hypothesis \eqref{hypstartcampa}, \eqref{so01} holds for $k=0$ and $\frac{1}{ \barR_0^2} \bar{E}_0+ \frac{1}{ \barR_0^2} \bar{D}_0 \les \frac{\beta(\barR_0)}{\barR_0^2}\ll1$ so that by Theorem \ref{theo:harmonicLagintro}, 
there exists $R_0\in (3  \barR_0, 4 \barR_0)$ such that defining $\Phi_0$ through \eqref{so33}, \eqref{so39} holds for $k=0$.
\medskip 

Assume now that  \eqref{ma75} holds for $k$. By \eqref{ma75} and  \eqref{ma53}, we have in particular $\frac{1}{\barR_{k}^2} \bar{E}_{k} +\frac{1}{\barR_{k}^2} \bar{D}_{k}\ll1$ 
and thus by Theorem \ref{theo:harmonicLagintro}, for any $0<\tau\ll 1$ there exists 
$R_{k}\in (3 \barR_{k},4 \barR_{k}) $ such that defining $\Phi_{k}$ through \eqref{so33}, we have \eqref{so39} and 
\begin{equation}\label{camp1}
 \frac{1}{|B_{\barR_{k}}|}\int_{(B_{\barR_{k}}\times \R^d)\cup(\R^d\times B_{\barR_{k}})} |x-y+\nabla \Phi_{k}(x)|^2 d\pi_{k}\le \tau \bar{E}_{k}+ C(\tau) \bar{D}_{k}.
\end{equation}
Defining $\pi_{k+1}$ through \eqref{so34}, we notice that if the support of $\pi_k$ is cyclically monotone,
then also the support of $\pi_{k+1}$ is cyclically monotone. Moreover, using that by \eqref{so39}, 
\begin{equation}\label{ma65}
 B_{6 \barR_{k+1}}(\nabla \Phi_{k}(0))= B_{6\theta \barR_k}(\nabla \Phi_{k}(0))\subset B_{7\theta \barR_k}  \subset B_{\barR_k} 
\end{equation}
we still have that in $B_{6 \barR_{k+1}}$, the first marginal of $\pi_{k+1}$ is $\mu$ and the second marginal is the Lebesgue measure. Moreover, by the $L^\infty$ bound \eqref{wg31},
\begin{equation}\label{camp2}
 \spt \pi_k\cap \lt[(B_{6 \barR_{k+1}}\times\R^d)\cup (\R^d\times B_{6 \barR_{k+1}}(\nabla \Phi_{k}(0)))\rt]\subset  B_{7\theta\barR_k}\times \R^d,
\end{equation}
and we may compute 
\begin{align*}
 \lefteqn{\frac{1}{|B_{6\barR_{k+1}}|}\int_{(B_{6\barR_{k+1}}\times\R^d)\cup(\R^d\times B_{6\barR_{k+1}})} |x-y|^2 d\pi_{k+1}}\\
 &= \frac{1}{|B_{6\theta\barR_{k}}|}\int_{(B_{6\theta\barR_{k}}\times\R^d)\cup(\R^d\times B_{6\theta\barR_{k}}(\nabla \Phi_{k}(0)))} |x-y+\nabla \Phi_{k}(0)|^2 d\pi_{k}\\
 &\stackrel{\eqref{ma65}\&\eqref{camp2}}{\les} \frac{1}{\theta^d|B_{\barR_{k}}|}\int_{(B_{\barR_{k}}\times \R^d)\cup(\R^d\times B_{\barR_{k}})} |x-y+\nabla \Phi_{k}(x)|^2 d\pi_{k} \\
 &\qquad \qquad + \sup_{B_{7\theta\barR_k}}|\nabla \Phi_k-\nabla\Phi_k(0)|^2\\
 &\stackrel{\eqref{camp1}\&\eqref{ma37}}{\les} \frac{\tau}{\theta^d} \bar{E}_k+ \frac{C(\tau)}{\theta^d}\bar{D}_k +\theta^2 (\bar{E}_k+ \bar{D}_k)\\
&\stackrel{\eqref{ma75}\&\eqref{ma53}}{\le}\lt( C C_E\lt(\frac{\tau}{\theta^d} +\theta^2\rt)+ \theta C(\tau,\theta) \rt)\beta(\barR_k)\\
&\stackrel{\eqref{hypCampa}}{\le} \lt( C C_E\lt(\frac{\tau}{\theta^{d+1}} +\theta\rt)+ C(\tau,\theta)\rt)\beta(\barR_{k+1}),
\end{align*}
where $C$ denotes a constant only depending on $d$ and $\beta$, but neither on $\tau$ nor on $\theta$.
We first fix  a dyadic  $0<\theta\ll1$ such that $C\theta\le \frac{1}{4}$, then $0<\tau\ll1$ such that $C \frac{\tau}{\theta^{d+1}}\le \frac{1}{4}$ and finally let $C_E:=2 C(\tau,\theta)$, proving that \eqref{ma75} holds for $k+1$.
\end{proof}

 \subsection{Choosing good radii}
 
 In this subsection, we show that without loss of generality we can make further restrictions on the sequence of radii $R_k$ given by Proposition \ref{theo:campintro}. 
 Let us recall that $\barR_k=\theta^k \barR_0$ and that for every $k\ge0$, $R_k\in (3\barR_k,4\barR_k)$ was chosen to be a good radius in the sense of Proposition \ref{prop:goodR}. 
 For notational simplicity we take as convention that 
\begin{equation}\label{convention}
\bar j_{-1}:=0 \qquad \textrm{ and } \qquad E_{-1}:= \frac{1}{|B_{\barR}|}\int_{(B_{\barR}\times\R^d)\cup(\R^d\times B_{\barR})}|x-y|^2 d\pi. 
\end{equation}
 \begin{lemma}\label{lem:goodradii}
 For every $k\in[0,K]$, we may assume that $R_k$ is both a good radius in the sense of Proposition \ref{prop:goodR}, that $\mu(\partial B_{R_k})=|\bar j_{k-1}|(\partial B_{R_k})=0$ and  
  \begin{align}
   \frac{R_k}{|B_{R_k}|}\int_0^1\int\big|I(|(1-t)x+ty|<R_k)-I(|y|<R_k)\big|d\pi_kdt \lesssim & \, \sqrt{E_{k-1}},\label{so76}\\
   \frac{R_k}{|B_{R_k}|}\int_0^1\int
\big|I(|(1-t)x+ty|<R_k)-I(|y|<R_k)\big||x-y|d\pi_kdt\lesssim &\,E_{k-1},\label{so10}\end{align}
hold. Moreover, setting $b_k:=\nabla \Phi_k(0)$, we may also assume that 
\begin{align}
\frac{R_{k+1}}{|B_{R_{k+1}}|} \int_0^1\int\big|I((1-t)x+ty|<R_{k+1})-I(|(1-t)x+t(y+b_k)|&<R_{k+1})\big|\nonumber\\
	|x-y|d\pi_{k+1}dt\lesssim &\, |b_k|\sqrt{E_{k-1}}.\label{so40}
  \end{align}
 holds. 

 \end{lemma}
Let us comment on Lemma \ref{lem:goodradii} before proving it. We require here that the transport across $\partial B_{R_k}$ is not above average in two different ways:
first on the level of the mass in \eqref{so76} and second at the level of the displacement in \eqref{so10}. As a consequence, we prove in \eqref{so42} that on the level of the Eulerian fluxes, the Lagrangian shift $\nabla \Phi_k(0)$
has an almost additive effect. In \eqref{so40}, we further require that when choosing $R_{k+1}$ (after having chosen $R_k$), the displacement induced by the shift $\nabla \Phi_k(0)$ is not above average near $\partial B_{R_{k+1}}$.

\begin{proof}[Proof of Lemma \ref{lem:goodradii}]
Since for every fixed $k\in [0,K]$, there exists a set of measure of order $\barR_k$  in $(3\barR_k,4\barR_k)$ of good radii in the sense of Proposition \ref{prop:goodR}, we may find $R_k$
which also satisfies $\mu(\partial B_{R_k})=|\bar j_{k-1}|(\partial B_{R_k})=0$, \eqref{so76}, \eqref{so10} and \eqref{so40} provided we show that
\begin{align}
\frac{1}{|B_{6\barR_k}|}\int_{3\barR_k}^{4\barR_k}\int_0^1\int\big|I(|(1-t)x+ty|<R)-I(|y|<R)\big|&d\pi_kdt dR\les \sqrt{E_{k-1}}, \label{so76b}\\
\frac{1}{|B_{6\barR_k}|}\int_{3\barR_k}^{4\barR_k}\int_0^1\int\lt|I(|(1-t)x+ty|<R)-I(|y|<R)\rt|&\nonumber\\ 
|x-y|d\pi_k&dtdR\le E_{k-1},\label{so38}\\
\frac{1}{|B_{6\barR_{k+1}}|}\int_{3\barR_{k+1}}^{4\barR_{k+1}}\int_0^1\int\big|I((1-t)x+ty|<R)-I(|(1-&t)x+t(y+b_k)|<R)\big|\nonumber \\
|x-y|d\pi_{k+1}&dt dR \lesssim \, |b_k|\sqrt{E_{k-1}}. \label{so35}
\end{align}
\medskip

{\sc Step 1.}  
Let us start with the proof of \eqref{so76b}. Using an exchange of integrals, 
\begin{align*}
	\int_0^\infty&\big|I(|(1-t)x+ty|<R)-I(|y|<R)\big|dR=\big||(1-t)x+ty|-|y|\big|\le(1-t)|y-x|,
\end{align*}
and  the $L^\infty$ bound \eqref{ma38} to reduce the integral to $(B_{6 \barR_k}\times \R^d)\cup(\R^d\times B_{6 \barR_k})$ , we obtain

\begin{multline*}\int_{3\barR_k}^{4\barR_k}\int_0^1\int\big|I(|(1-t)x+ty|<R)-I(|y|<R)\big|d\pi_kdt dR\\
\le \int_{(B_{6 \barR_k}\times \R^d)\cup(\R^d\times B_{6 \barR_k})} |x-y| d\pi_k.
\end{multline*}
The desired estimate \eqref{so76b} follows since by Cauchy-Schwarz's inequality and the mass bound  \eqref{ma88} 
\begin{align*}
\lefteqn{ \int_{(B_{6 \barR_k}\times \R^d)\cup(\R^d\times B_{6 \barR_k})} |x-y| d\pi_k}\\
&\les \lt(\int_{(B_{6 \barR_k}\times \R^d)\cup(\R^d\times B_{6 \barR_k})} d\pi_k\rt)^{\frac{1}{2}}
 \lt(\int_{(B_{6 \barR_k}\times \R^d)\cup(\R^d\times B_{6 \barR_k})} |x-y|^2 d\pi_k\rt)^{\frac{1}{2}}\\
 &\les \sqrt{\bar E_k} \cdot |B_{6 \barR_k}| \les \sqrt{ E_{k-1} }\cdot |B_{6 \barR_k}|, \numberthis \label{po02}
\end{align*}
where for the last inequality, one can argue as for the proof of \eqref{so01}.\\

\medskip
{\sc Step 2.} Estimate  \eqref{so38} is obtained  arguing as for \eqref{so76b}.

\medskip
{\sc Step 3.}
We now turn to the proof of \eqref{so35}. 
On the one hand, this relies on the elementary
\begin{align*}
	\int_0^\infty\big|I((1-t)x+ty|<R)-I(|(1-t)x+t(y+b_k)|<R)\big|dR\nonumber\\
=\big||(1-t)x+t(y+b_k)|-|(1-t)x+ty|\big|\le t|b_k|,
\end{align*}
and on the other hand, on the fact that for fixed $t\in[0,1]$ and $R\in (3\barR_k,4\barR_k)$, the statement
($|(1-t)x+ty|<R$ or $|(1-t)x+t(y+b_k)|<R$) implies $|(1-t)x+ty|<R+o(R)$,  since by \eqref{so39} $|b_k|^2\les\beta(R)\ll R^2$, which in turn  implies the statement ($|x|<6\barR_k$ or $|y|<6\barR_k$) by \eqref{ma38}. Then, we can argue as for \eqref{so76b}.
\end{proof}

We may derive Lemma \ref{lem:additive} (which we now restate) from Lemma \ref{lem:goodradii}
 \begin{lem1.12}
  For every $k\in[1,K]$,
\begin{align}\label{so42}
	\left|\frac{1}{|B_{R_{k}}|}\int_{B_{R_{k}}}(\bar j_{k}+\nabla \Phi_{k-1}(0)-\bar j_{k-1})\rt|
	\lesssim \frac{\beta(R_k)}{R_k}.
\end{align}
 \end{lem1.12}

 \begin{proof}
Recall from Lemma \ref{lem:goodradii} the abbreviation $b_{k-1}=\nabla \Phi_{k-1}(0)$. Using the definition (\ref{so08}) of $j_{k-1}$,
the fact that the second marginal of $\pi_{k-1}$ is the Lebesgue measure, \eqref{hypCampa}, \eqref{so01}, and \eqref{so39}, it is enough to prove that (recall \eqref{convention})
\begin{align*}
	\lt|\int_0^1\Big(\rt.
	&\int\big(I(|(1-t)x+ty|<R_{k})(y-x)d\pi_{k}\nonumber\\
        +&\lt.\int\big(I(|y|<R_{k})b_{k-1}-I(|(1-t)x+ty|<R_{k})(y-x)\big)d\pi_{k-1}\Big)dt\rt|
				        \nonumber\\
	&\lesssim |b_{k-1}| \sqrt{E_{k-2}}  |B_{R_{k}}|.
	\end{align*}
By the recursive relation (\ref{so34}) this may be rearranged to
\begin{align*}
	\lt|\int_0^1\Big(\rt.&\int
	\big(I(|(1-t)x+ty|<R_{k})-I(|(1-t)x+t(y+b_k)|<R_{k})\big)(y-x)d\pi_{k}\nonumber\\
	-&\lt.\int\big(I(|(1-t)x+ty|<R_{k})-I(|y|<R_{k})\big)b_{k-1}d\pi_{k-1}\Big)dt\rt|\\
	&\lesssim |b_{k-1}|\sqrt{E_{k-2}}  |B_{R_{k}}|.
\end{align*}
This is obtained using  (\ref{so40}) for the first contribution and  (\ref{so76}) for the second contribution.
\end{proof}

 \subsection{Flux of solutions to Poisson equations on concentric annuli}

 For $R>0$ and $\mu$ a measure on $\R^d$, let $v^R$ be defined (up to an additive constant) through
 \begin{equation}\label{so46}
  \Delta v^R=\mu-c \ \textrm{ in } B_R \qquad \textrm{ and } \qquad \nu\cdot \nabla v^R=0 \ \textrm{ on } \partial B_R.
 \end{equation}
The aim of this section is to prove the following summation formula (which is a restatement of Lemma \ref{lem:sumvk}):
\begin{lem1.14}
 Let $R_0>\cdots >R_k$ be a sequence of decreasing radii with $\mu(\partial B_{R_\ell})=0$ for every $\ell\in[1,k]$, then 
 \begin{equation}\label{eq2}
 \frac{1}{|B_{R_k}|}\int_{\partial B_{R_k}}x\nu\cdot\nabla v^{R_0}
	=\sum_{\ell=1}^{k}\frac{1}{|B_{R_{\ell}}|}\int_{\partial B_{R_{\ell}}}x\nu\cdot\nabla v^{R_{\ell-1}}.
\end{equation}
\end{lem1.14}
This is obtained through a representation formula for $\frac{1}{|B_{r}|}\int_{\partial B_{r}}x\nu\cdot\nabla v^{R}$ for every $0<r<R$.
Concentrating for instance on the first component, we have
\begin{lemma}\label{lem:omegaR}
 For $0<r<R$, let $\omega_r^R$ be the (piecewise smooth) solution of 
 \begin{align}\label{so47}\left\{
	\begin{array}{lcl}
		\Delta\omega_r^R=0                 &\mbox{in}&B_r\cup(B_R\backslash\bar B_r),\\
	\mbox{$[\omega_r^R]$}=\frac{x_1}{|B_r|},\;[\nu\cdot\nabla\omega_r^R]=0&\mbox{on}&\partial B_r,\\
	\nu\cdot\nabla\omega_r^R=0            &\mbox{on}&\partial B_R,\\
	\omega_r^R=0 & \mbox{in} & B_R^c
\end{array}\right.
\end{align}
with $[\cdot]$ denoting the jump across $\partial B_r$ (outside minus inside);
normalized by requiring 
\begin{align}\label{so49}
\int_{B_R}\omega_r^R=0.
\end{align}
Then, for every $0<r<R$ with $\mu(\partial B_r)=\mu(\partial B_R)=0$, we have the representation formula
\begin{align}\label{so48}
\frac{1}{|B_{r}|}\int_{\partial B_{r}}x_1\nu\cdot\nabla v^R
	=-\int_{B_R}\omega_r^R d\mu.
\end{align}
Moreover, for $R_0>\cdots >R_k$ a sequence of decreasing radii, we have the telescoping formula
\begin{align}\label{so50}
\omega_{R_k}^{R_0}=\sum_{\ell=1}^{k}\omega_{R_{\ell}}^{R_{\ell-1}}.
\end{align}
Finally, we have the explicit formula for $\nabla \omega_r^R$:
\begin{multline}\label{so66}
	\nabla\omega_r^R(x)\\
	=\left\{\begin{array}{lcr}
\frac{d-1}{\mathcal{H}^{d-1}(\partial B_1)}(\frac{1}{R^d}-\frac{1}{r^d})&\mbox{for}&|x|<r\\
\frac{d-1}{\mathcal{H}^{d-1}(\partial B_1)}\frac{1}{R^d}
+\frac{1}{\mathcal{H}^{d-1}(\partial B_1)}\frac{1}{|x|^d}({\rm id}-d\frac{x}{|x|}\otimes\frac{x}{|x|})
				        &\mbox{for}&r<|x|<R\\
	0&\mbox{for}&|x|>R\end{array}\right\}e_1,
\end{multline}
where $e_1$ denotes the unit vector in the $x_1$-direction and $\rm id$ denotes the identity matrix.

\end{lemma}
\begin{proof}[Proof of Lemma \ref{lem:sumvk}]
 Because of $v^{R_k}=v_k$, cf.\ \eqref{defvkintro} and \eqref{so46} this is a direct consequence of \eqref{so48} and \eqref{so50}.
\end{proof}

\begin{proof}[Proof of Lemma \ref{lem:omegaR}]
We start by arguing that there is an explicit representation for $\omega_r^R$ given by
\begin{align}\label{so51}
\omega_r^R(x)=\left\{\begin{array}{lcr}
	\frac{d-1}{\mathcal{H}^{d-1}(\partial B_1)}(\frac{1}{R^d}-\frac{1}{r^d})x_1&\mbox{for}&|x|<r\\
	\frac{d-1}{\mathcal{H}^{d-1}(\partial B_1)}\frac{x_1}{R^d}+\frac{1}{\mathcal{H}^{d-1}(\partial B_1)}\frac{x_1}{|x|^d}
	&\mbox{for}&r<|x|<R\\
0&\mbox{for}&|x|>R\end{array}\right\}.
\end{align}
This is easily seen to hold: The normalization (\ref{so49}) is satisfied
(because the function is odd), the so defined function is piecewise harmonic
(because the dipole $\frac{x_1}{|x|^d}$ is the $\partial_1$-derivative of the fundamental solution),
it satisfies the first jump condition at $|x|=r$ (because of $\mathcal{H}^{d-1}(\partial B_1)=d|B_1|$),
and has no jump in the normal component at $|x|=r,R$ 
in view of formula \eqref{so66} for $\nabla \omega_r^R$ (because of $\nu=\frac{x}{|x|}$).\\

We now give the argument for \eqref{so48}.  By\footnote{Note that under our assumption of
$\mu(\partial B_r)=\mu(\partial B_R)=0$, the right-hand side is well-defined since $\mu$ does not charge
the jump set of $\omega^R_r$ and the left-hand side is well-defined since the inner and outer normal distributional traces  of $\nabla v^R$ on $\partial B_r$ agree since $\nabla v^R \in L^p(B_R)$ for $1\leq p <\frac{d}{d-1}$ by Lemma \ref{lem:gradu}.
Since by \cite[Prop. 3.2]{ChenFriedreview} both sides are well approximated by mollification (on $\R^d$),  it is sufficient to show \eqref{so48} for smooth approximation of $\mu$ and thus $v^R$.} Green's formula and \eqref{so49} there holds
\begin{align*}
-\int_{B_R}\omega_r^R d\mu
&\stackrel{(\ref{so46}),(\ref{so49})}{=}-\int_{B_r}\omega_r^R\Delta v^R
-\int_{B_R\backslash\bar B_r}\omega_r^R\Delta v^R\nonumber\\
&=-\int_{B_r}v^R\Delta\omega_r^R-\int_{B_R\backslash\bar B_r}v^R\Delta\omega_r^R\nonumber\\
&\quad +\int_{\partial B_r}[\omega_r^R]\nu\cdot\nabla v^R-\int_{\partial B_r} v^R [\nu\cdot \nabla \omega_r^R]\\
&\quad-\int_{\partial B_R}\omega_r^R\nu\cdot\nabla v^R
+\int_{\partial B_R}v^R\nu\cdot\nabla\omega_r^R\nonumber\\
&\stackrel{(\ref{so46}),(\ref{so47})}{=}\frac{1}{|B_{r}|}\int_{\partial B_{r}}x_1\nu\cdot\nabla v^R.
\end{align*}

We finally prove the telescoping formula \eqref{so50}. To show this, it is convenient
to split the functions into the affine part $x_1$ and the dipolar part $\frac{x_1}{|x|^d}$.
The telescoping for the dipolar part is obvious. The affine part can be written as
\begin{align*}
\frac{(d-1)x_1}{\mathcal{H}^{d-1}(\partial B_1)}
\lt(\frac{1}{R^d}I(|x|<R)-\frac{1}{r^d}I(|x|<r)\rt)
\end{align*}
and therefore is also telescoping. 
\end{proof}

\subsection{Linearization on nearby scales}
 
 In this section we establish the main ingredient for the proof of Theorem \ref{theo:main}, namely Lemma \ref{lem:Le1intro}, a linearization result on the level of (Eulerian) flux. 
 The main difference with Theorem \ref{theo:main} is that it will only deal with nearby scales whereby Theorem \ref{theo:main} bridges scales.
 Before restating the result we  recall a bit of notation.

 For each $k\in [0,K]$, let $u_k$ be the (distributional) solution to
 \begin{align}\label{so16}
\Delta u_k=\mu-1\ \textrm{ in } B_{R_k}\qquad\textrm{ and }\qquad
	\nu\cdot\nabla u_k=\nu\cdot\bar j_k\ \textrm{ on } \partial B_{R_k}
\end{align}
and let $v_k$ be the solution to the Poisson problem with homogeneous flux boundary data
\begin{align}\label{so43}
 \Delta v_k=\mu-1-c\ \textrm{ in } B_{R_k}\qquad\textrm{ and }\qquad
	\nu\cdot\nabla v_k=0\ \textrm{ on } \partial B_{R_k},
\end{align}
so that (up to additive constants) 
\begin{equation}\label{so41}
 u_k=v_k+\Phi_k.
\end{equation}

Let $\eta$ be a smooth and radial cutoff function as in \eqref{cutoff}, i.e.
\begin{align}\label{ao15} \eta\in C^\infty_0(B_1)\;\mbox{with}\;\int\eta=1, \ \sup |\nabla^{ 2} \eta|\les 1 \quad
	\mbox{ and we set }\;\eta_R:=\frac{1}{R^d}\eta\lt(\frac{\cdot}{R}\rt).\end{align}

\begin{lem1.13}
For $R\in[R_{k+1},R_k]$, there holds 
\begin{align}
 	\lt|\int\eta_R(\bar j_k-\nabla u_k)\rt|&\lesssim\frac{E_k}{R_k},\label{ao13}\\
	\frac{1}{|B_{R_{k+1}}|}\lt|\int_{B_{R_{k+1}}}(\bar j_k-\nabla u_k)\rt|&\lesssim \frac{E_k}{R_k}.\label{so06}
\end{align}

\end{lem1.13}
To appreciate this result, note that the vector field 
\begin{align}\label{ao22}
\sigma_k:=\bar j_k-\nabla u_k
\end{align}
satisfies the homogeneous conditions
\begin{align}\label{so04}
\nabla\cdot \sigma_k=0\ \textrm{ in } B_{R_k}\qquad\mbox{ and }\qquad
\nu\cdot \sigma_k=0\ \mbox{ on } \partial B_{R_k},
\end{align}
cf.~(\ref{so02}) and \eqref{so16}. In Step \ref{Le1S1} of the proof of Lemma \ref{lem:Le1intro} we will establish that 
$\bar j_k\approx{\rm id}-\nabla\psi_k$ on average,
where $\psi_k$ is the Kantorovich potential.
In particular, $\sigma_k$ is weakly  close to a gradient
\begin{align}\label{so05}
\sigma_k\approx\nabla w_k\quad\mbox{on average},
\end{align}
namely the gradient of
\begin{align}\label{ao23}
w_k(y):=\frac{1}{2}|y|^2-\psi_k(y)-u_k(y).
\end{align}
Smallness of large-scale averages of $\sigma_k$ will follow from the div-curl
system (\ref{so04}) \& (\ref{so05}) via a suitable representation
in Steps \ref{Le1S2} and \ref{Le1S0} relating averages of $\sigma_k$ with averages of $\sigma_k-\nabla w_k$ (it is for this representation that it is crucial that $\sigma_k$ is divergence free and $\nabla w_k$ is a gradient field).

\begin{proof}[Proof of Lemma \ref{lem:Le1intro}]
For notational convenience, we can momentarily  assume by scaling that $R_k=1$,  drop the index $k$  and set $r:=R_{k+1}$ (and therefore consider radii $R\in[r,1]$).
\medskip

Let us first recall that thanks to Lemma \ref{lem:goodradii}, $r$ is chosen to be a good radius in the sense of  $|\bar j|(\partial B_r)=0$. Since  Lemma \ref{lem:gradu} implies $\nabla u\in L^1(B_1)$  we also have $|\sigma|(\partial B_r)=0$ (cf.\ \eqref{ao22}). We use $\nabla u\in L^1(B_1)$ once more together with the
Lipschitz continuity of the Kantorovich potential $\psi$ to arrive at (cf.\ \eqref{ao23})
 \begin{align}\label{ao12}
  \nabla w \in L^1(B_1).
 \end{align}

\newcounter{Le1S} 
\refstepcounter{Le1S} 

{\sc Step} \arabic{Le1S}.\label{Le1S2}\refstepcounter{Le1S}
We claim the representation formula of $\int\eta_R \sigma$ in terms of $\nabla w-\sigma$,
restricting without loss of generality to the first component:
\begin{align}\label{ao14}
\int\eta_R \sigma_1=\int_{B_1}(\eta_R e_1-\nabla\omega_R)\cdot(\sigma-\nabla w),
\end{align}
where $\omega_R$ is a  solution of (compare to \eqref{ao09})
\begin{equation}\label{je02}
 \Delta \omega_R=\partial_1 \eta_R \ \textrm{ in } B_1 \qquad \textrm{ and } \qquad \nu\cdot\nabla \omega_R=0 \ \textrm{ on } \partial B_1.
\end{equation}
Indeed, it follows
that (the smooth) $\eta_R e_1-\nabla\omega_R$ is divergence-free 
on $B_1$ with vanishing flux boundary data, so that the contribution of $\nabla w$
on the right-hand side of \eqref{ao14} vanishes. The contribution $\int_{B_1}\nabla\omega_R\cdot \sigma$
vanishes because of \eqref{so04}.

\medskip

{\sc Step} \arabic{Le1S}.\label{Le1S0}\refstepcounter{Le1S}
We claim the representation formula of $\frac{1}{|B_r|}\int_{B_r}\sigma$ in terms of $\nabla w-\sigma$:
\begin{align}\label{so65}
	&\frac{1}{|B_r|}\int_{B_r}\sigma\nonumber\\
	&=\frac{d-1}{\mathcal{H}^{d-1}(\partial B_1)}\lt(1-\frac{1}{r^d}\rt)\int_{B_r}(\nabla w-\sigma)\nonumber\\
	&+\frac{1}{\mathcal{H}^{d-1}(\partial B_1)}\int_{B_1\backslash\bar B_r}\lt((d-1){\rm id}
	+\frac{1}{|x|^d}\lt({\rm id}-d\frac{x}{|x|}\otimes\frac{x}{|x|}\rt)\rt)(\nabla w-\sigma).
\end{align}
Indeed, in view of the formula (\ref{so66}) equation (\ref{so65}) assumes the form\footnote{Note that this
expression is well-defined since the measure $\sigma$ does not ``charge'' $\partial B_r$,
the only place where $\nabla\omega_r^1$ is discontinuous, cf.\ beginning of this proof.}
\begin{equation}\label{je01}
\frac{1}{|B_r|}\int_{B_r}\sigma_1 =\int_{B_1}\nabla\omega_r^1\cdot(\nabla w-\sigma).
\end{equation}
By (\ref{so04}) and the divergence theorem (recall \eqref{eq:div})
$\frac{1}{|B_r|}\int_{B_r}\sigma_1=\frac{1}{|B_r|}\int_{\partial B_r}x_1\nu\cdot \sigma$,
so that by the jump condition in (\ref{so47}) we may write
\begin{align*}
\frac{1}{|B_r|}\int_{B_r}\sigma_1=\int_{\partial B_r}[\omega_r^1]\nu\cdot \sigma.
\end{align*}
Recall that $[\omega_r^1]$ denotes the jump across $\partial B_r$ calculated as outside minus inside. 
By (\ref{so04}) and \eqref{so47} (and since inner and outer distributional normal traces of $\sigma$ agree
in view of (\ref{so04})) we then obtain\footnote{Note that this holds for any function $w$ that is smooth 
on the closure $\bar B_1$;
our $w$ can be approximated by such functions in the strong $L^1$-topology on the level
of the gradients, cf.~(\ref{ao12}), which is sufficient to pass to the limit in the 
representation formula,
since $\nabla\omega_r^1$ is in $L^\infty$.}
\begin{align*}
\frac{1}{|B_r|}\int_{B_r}\sigma_1&=-\int_{B_r}\sigma\cdot\nabla\omega_r^1-\int_{B_1\backslash \bar B_r}\sigma\cdot\nabla\omega_r^1
\nonumber\\
&=\int_{B_r}(\nabla w-\sigma)\cdot\nabla\omega_r^1+\int_{B_1\backslash \bar B_r}(\nabla w-\sigma)\cdot\nabla\omega_r^1\nonumber\\
&\qquad -\lt(\int_{B_r}\nabla w\cdot\nabla\omega_r^1+\int_{B_1\backslash \bar B_r}\nabla w\cdot\nabla\omega_r^1\rt).
\end{align*}
The term in the parentheses vanishes since $\nabla\omega_r^1$ is weakly divergence-free
in $B_1$ with vanishing normal component on $\partial B_1$, cf.~(\ref{so47}). This proves \eqref{je01}.

\medskip

{\sc Step} \arabic{Le1S}.\label{Le1S1}\refstepcounter{Le1S}
We establish (\ref{so05}) in the following sense
\begin{equation}
 \label{so09}
\lt|\int_{B_1}\xi\cdot(\nabla w-\sigma)\rt|\lesssim
\big(\sup|\nabla\xi|+\sup|\xi|\big)
\int_{(B_{6}\times\mathbb{R}^d)\cup(\mathbb{R}^d\times B_{6})}|x-y|^2d\pi
\end{equation}
for any vector fields $\xi\in C^\infty_0(\mathbb{R}^d)^d$, 
and the same estimate with $B_1$ replaced by $B_r$.
\medskip

Indeed, we first note that by definition \eqref{ao22} and \eqref{ao23} of $\sigma$ and $w$ we have
\begin{align*}
\nabla w-\sigma={\rm id}-\nabla\psi-\bar j.
\end{align*}
Since\footnote{Apart from the $L^\infty$ estimate this is the only place where we use that $\pi$ is an optimal transport plan. } $\pi$ is supported on the graph $(\partial\psi(y),y)$ 
of the 
subgradient $\partial\psi$ of a  convex function $\psi$,
and since the $y$-marginal of $\pi$ is given by the Lebesgue measure $dy$, we have
by the almost-everywhere differentiability of $\psi$
\begin{align*}
\int\zeta d\pi=\int\zeta(\nabla\psi(y),y)dy
\end{align*}
for any test function $\zeta$. Hence we have by definition (\ref{defrhoj}) of $\bar j=\int_0^1j_tdt$
\begin{align*}
\lefteqn{\int_{B_1}\xi(y)\cdot(y-\nabla\psi(y))dy-\int_{B_1}\xi\cdot d\bar j}\nonumber\\
&=\int_0^1\int \big(I(B_1)\xi)(y)-(I(B_1)\xi)((1-t)x+ty)\big)
\cdot(y-x)d\pi dt.
\end{align*}
Hence the desired estimate splits into the obvious
\begin{multline*}
\lt|\int_0^1\int I(|y|\le 1)\big(\xi(y)-\xi((1-t)x+ty)\big)
\cdot(y-x)d\pi dt\rt|
\\
\le\frac{1}{2}\sup|\nabla\xi|\int_{\mathbb{R}^d\times B_{1}}|x-y|^2d\pi
\end{multline*}
and into
\begin{align*}
\int_0^1\int&\big|I(|y|\le 1)-I(|(1-t)x+ty|\le 1)\big||x-y|d\pi dt
\lesssim\int_{(B_{6}\times\mathbb{R}^d)\cup(\mathbb{R}^d\times B_{6})}|x-y|^2d\pi,
\end{align*}
for which we need that $R_k=1$ is a good radius in the sense of (\ref{so10}),
cf.~Lemma \ref{lem:goodradii}.

\medskip

{\sc Step} \arabic{Le1S}.\label{Le1S3}\refstepcounter{Le1S}
Proof of (\ref{ao13}). In view of the representation (\ref{ao14}), we
need to apply (\ref{so09}) to $\xi=\eta_R e_1-\nabla\omega_R$.
It follows from \eqref{je02} that $\nabla\omega_R$ is of the same (scaling) form 
as $\eta_R$, cf.~(\ref{ao15}), namely $\nabla\omega_R=\frac{1}{R^d}\nabla\omega(\frac{\cdot}{R})$ for some $\omega\in C^\infty(\bar B_1)$,
which we may extend to an $\omega\in C^\infty_0(\mathbb{R}^d)$.
In particular, we have 
$$\sup|\nabla\xi|+\sup|\xi|\lesssim\frac{1}{R^{d+1}}+\frac{1}{R^d}\lesssim 1,$$
where we used closeness of scales in form of $R\ge r\gtrsim 1$, cf.~(\ref{so30}).
Hence (\ref{ao13}) follows from (\ref{so09}).

\medskip

{\sc Step} \arabic{Le1S}.\label{Le1S-1}\refstepcounter{Le1S}
Proof of (\ref{so06}). We recall (\ref{so65})
and apply first (\ref{so09}) with $\xi$ a constant unit vector $e$ and $R_k=1$ replaced
by $r$. Since $r\gtrsim R$, cf.~(\ref{so30}), we obtain
\[
 \lt|\lt(1-\frac{1}{r^d}\rt)\int_{B_r} (\nabla w-\sigma)\rt|\les \int_{(B_{6r}\times\mathbb{R}^d)\cup(\mathbb{R}^d\times B_{6r})}|x-y|^2d\pi.
\]

We now turn to the second contribution to (\ref{so65}). Let $\xi\in C^\infty(B_1\backslash \bar B_r)$. Appealing to (\ref{so09})
for both $R_k=1$ and $r\sim 1$, and to a Lipschitz extension of $\xi$ from $B_1\backslash \bar B_r$ to $B_1$
that preserves the $(\sup|\nabla\xi|+\sup|\xi|)$-norm, we obtain
\[
	\lt|\int_{B_1\backslash \bar B_r}\xi\cdot(\nabla w-\sigma)\rt|\lesssim
	\big(\sup_{B_1\backslash \bar B_r}|\nabla\xi|+\sup_{B_1\backslash \bar B_r}|\xi|\big)
	\int_{(B_{6}\times\mathbb{R}^d)\cup(\mathbb{R}^d\times B_{6})}|x-y|^2d\pi,
\]
which we apply to $\xi:=\frac{1}{\mathcal{H}^{d-1}(\partial B_1)}\big((d-1){\rm id}
	+\frac{1}{|x|^d}({\rm id}-d\frac{x}{|x|}\otimes\frac{x}{|x|})\big)e$.
Since, using again $r\gtrsim 1$, 
we have $\sup_{B_1\backslash \bar B_r}|\nabla\xi|+\sup_{B_1\backslash \bar B_r}|\xi|\lesssim1$, and thus
\begin{multline*}
 \lt|\int_{B_1\backslash \bar B_r} \frac{1}{\mathcal{H}^{d-1}(\partial B_1)}\big((d-1){\rm id}
	+\frac{1}{|x|^d}({\rm id}-d\frac{x}{|x|}\otimes\frac{x}{|x|})\big)\cdot (\nabla w-\sigma)\rt| \\
	\les \int_{(B_{6}\times\mathbb{R}^d)\cup(\mathbb{R}^d\times B_{6})}|x-y|^2d\pi,
\end{multline*}
and the desired estimate follows.
\end{proof}

Let us recall that from Lemma \ref{lem:additive} and Lemma \ref{lem:Le1intro}, we derived in the introduction Lemma \ref{lem:shiftonestep}, that is:
\begin{lem1.11}
 For every $k\in[1,K-1]$,
 \begin{equation}\label{so45}
  \lt|\nabla\Phi_{k}(0)
-\frac{1}{|B_{R_{k}}|}\int_{\partial B_{R_{k}}}x\nu\cdot\nabla v_{k-1}\rt|
	\lesssim \frac{\beta(R_{k})}{R_{k}}.
 \end{equation}
\end{lem1.11}

 \subsection{Representation of the cumulative shift}
 In this section, we establish the last ingredient to the proof of Theorem \ref{theo:main}, a crucial representation for the cumulative shift that links it to the flux of the Poisson equation, i.e. Proposition \ref{prop:repres intro}, which we now restate. 
Recall the definition \eqref{so33} of $\Phi_k$ and \eqref{so16} of $u_k$. 

 \begin{prop1.10}
For every $k\in [0,K-1]$ and $R\in[R_{k+1}, R_k]$, we have the two estimates on the cumulative flux
\begin{align}\label{ao08}
\lt|\sum_{\ell=0}^k\nabla\Phi_\ell(0)
-\frac{1}{|B_{R_k}|}\int_{\partial B_{R_k}}x\nu\cdot\nabla u_0\rt|\lesssim \frac{\beta(R_k)}{R_k}
\end{align}
and
\begin{align}\label{ao21}
\lt|\sum_{\ell=0}^{k-1}\nabla\Phi_\ell(0)+\int\eta_R\nabla u_k
-\int\eta_R\nabla u_0\rt|\lesssim \frac{\beta(R_k)}{R_k}.
\end{align}
 \end{prop1.10} 
 
 \begin{proof}

We start with the proof of \eqref{ao08} and first point out that from \eqref{eq2} (and the fact that  \eqref{so46} and \eqref{so43}
are the same boundary problem and thus $v_k=v^{R_k}$),
\begin{align*}
\frac{1}{|B_{R_k}|}\int_{\partial B_{R_k}}x\nu\cdot\nabla v_0
	=\sum_{\ell=1}^{k}\frac{1}{|B_{R_{\ell}}|}\int_{\partial B_{R_{\ell}}}x\nu\cdot\nabla v_{\ell-1}.
\end{align*}
Therefore, using that $u_0=v_0+\Phi_0$ (recall \eqref{so41}) and the divergence theorem (cf.\ \eqref{eq:div}) we have 
\begin{align*}
	\nabla\Phi_{0}(0)
	+\sum_{\ell=1}^{k}\frac{1}{|B_{R_{\ell}}|}\int_{\partial B_{R_{\ell}}}x\nu\cdot\nabla v_{\ell-1}
	=\frac{1}{|B_{R_{k}}|}\int_{\partial B_{R_{k}}}x\nu\cdot\nabla u_0.
\end{align*}
We conclude from Lemma \ref{lem:shiftonestep} that 

\begin{align*}
 \lt|\sum_{\ell=0}^k\nabla\Phi_\ell(0)
-\frac{1}{|B_{R_k}|}\int_{\partial B_{R_k}}x\nu\cdot\nabla u_0\rt|&=
\lt|\sum_{\ell=1}^{k} \lt(\nabla \Phi_{\ell}(0) - \frac{1}{|B_{R_{\ell}}|}\int_{\partial B_{R_{\ell}}}x\nu\cdot\nabla v_{\ell-1}\rt)\rt|\\
&\stackrel{\eqref{so45}}{\les} \sum_{\ell=1}^{k} \frac{\beta(R_{\ell})}{R_{\ell}}\stackrel{\eqref{hypbetamain}}{\les} \frac{\beta(R_k)}{R_k},
\end{align*}
which is \eqref{ao08}.\\

From \eqref{ao08}, estimate  \eqref{ao21} follows via the identity
\begin{align*}
\nabla\Phi_k(0)-\frac{1}{|B_{R_k}|}\int_{\partial B_{R_k}}x\nu\cdot\nabla u_0
=\int\eta_R\nabla u_k-\int\eta_R\nabla u_0
\end{align*}
for $R\in[R_{k+1},R_k]$. Notice that  since $\nabla u_k=\nabla v_k+\nabla \Phi_k$ (recall (\ref{so41})) and  by the mean-value property for the harmonic $\nabla\Phi_k$
in form of $\int\eta_R\nabla\Phi_k$ $=\nabla\Phi_k(0)$, cf.~the normalization in (\ref{ao15}), this identity reduces to
\begin{align*}
-\frac{1}{|B_{R_k}|}\int_{\partial B_{R_k}}x\nu\cdot\nabla u_0
=\int\eta_R\nabla v_k-\int\eta_R\nabla u_0.
\end{align*}
By the boundary condition $\nu\cdot \nabla v_k=0$ on $\partial B_{R_k}$ (cf. (\ref{so43})) the last statement takes the form
$$\frac{1}{|B_{R_k}|}\int_{\partial B_{R_k}}x\nu \cdot\nabla(v_k-u_0)=\int\eta_R\nabla(v_k-u_0).$$
Since $\nabla(v_k-u_0)$ is harmonic in $B_{R_k}$,
cf.~(\ref{so43}) \& (\ref{so16}), the latter holds by the divergence theorem (see \eqref{eq:div}) and the mean-value property
in form of 
$$\int\eta_R\nabla(v_k-u_0)=\frac{1}{|B_{R_k}|}\int_{B_{R_k}}\nabla(v_k-u_0).$$ 
\end{proof}

 Let us close this section by  estimating the distance between  spherical and radial averages of $\nabla u$. 
 This implies that when identifying the cumulated  shift, one can use either of them.
 \begin{lemma}\label{lem:shift}
  For every $k\in [0,K-1]$ and $R\in[R_{k+1},R_k]$,
  \begin{equation}\label{eq:mer01}
   \lt|\int \eta_R \nabla u_0-\frac{1}{|B_{R_k}|}\int_{\partial B_{R_k}} x\nu\cdot \nabla u_0 \rt|^2\les \beta(R).
  \end{equation}

 \end{lemma}

 \begin{proof}
 By \eqref{hypdataCampa}, it is enough to prove that 
  \begin{equation}\label{eq:mer02}
   \lt|\int \eta_R \partial_1 u_0-\frac{1}{|B_{R_k}|}\int_{\partial B_{R_k}} x_1\nu\cdot \nabla u_0 \rt|^2\les\frac{1}{|B_{R_k}|}W_{B_{R_k}}^2(\mu,\kappa).
  \end{equation}
By scaling we may assume that $R_k=1$.
The argument is somewhat reminiscent of the one leading to Lemma \ref{lem:omegaR}. 
We claim that there exists $\omega\in C^\infty(\overline{B}_1)$ with $\sup_{B_1} |\nabla \omega|\les 1$ such that for $\zeta\in C^\infty_c(\R^d)$,
\begin{equation}\label{eq:mer03}
 \int \eta_R \partial_1\zeta -\frac{1}{|B_1|}\int_{\partial B_1} x_1 \nu\cdot\nabla\zeta=\int_{B_1} \omega \Delta \zeta.
\end{equation}
Then, testing this equation with $\zeta=u_0$ (thanks to a simple approximation argument) and observing that $\int_{B_1}\omega=0$ (choosing $\zeta=|x|^2$), we would obtain \eqref{eq:mer02} from
\[
 \left|\int_{B_1} \omega d(\mu-\kappa)\rt|^2\les W^2_{B_1}(\mu,\kappa).
\]
Let $\hat{\omega}$ be a solution of 
\[
 \Delta \hat{\omega}=-\eta_R -\frac{d-1}{|B_1|}  \ \textrm{ in } B_1 \qquad \textrm{ and } \qquad  \nu\cdot \nabla \hat{\omega}=-\frac{1}{|B_1|} \ \textrm{ on } \partial B_1,
\]
which is solvable since $\int \eta_R=1$ and $\mathcal{H}^{d-1}(\partial B_1)=d|B_1|$. Since $\eta_R$ is radially symmetric so is $\hat{\omega}$ and it is readily checked that $\omega:=\partial_1\hat{\omega}=\frac{d \hat{\omega}}{dr}(r)\frac{x_1}{r}$ solves 
\[
 \Delta \omega=-\partial_1 \eta_R \ \textrm{ in } B_1 \qquad \textrm{ and } \qquad \nu\cdot \nabla \omega=0, \ \omega=-\frac{x_1}{|B_1|} \ \textrm{ on } \partial B_1 
\]
which is equivalent to \eqref{eq:mer03}. Finally since $R\sim 1$ we have $\sup_{B_1} |\nabla \omega|\les 1$. 
 \end{proof}

 \subsection{Proof of Theorem \ref{theo:main}}
 We start this section with a technical lemma that we will need for the proof of Theorem \ref{theo:main}.
\begin{lemma}\label{lem:aux}
 For every $k\in[0,K-1]$ and every $R\in[R_{k+1},R_k]$, there holds 
 \begin{align}
	\lt|\int\eta_R\bar j_k-\int\eta_R(x)(y-x)d\pi_k\rt|&\lesssim\frac{E_k}{R_k},\label{ao19}\\
	\lt|\int\eta_R(x)d\pi-1\rt|&\lesssim\frac{\sqrt{E_k}}{R_k}\ll 1,\label{ao17}\\
	\lt|\int\eta_R\nabla u_k\rt|&\lesssim\sqrt{E_k}.\label{ao18}
\end{align}
\end{lemma}
\begin{proof}
 We start with (\ref{ao19}) and note that by definition (\ref{so08}) of $\bar j$, the left-hand side can be written as 
 \begin{multline*}
 \int_0^1\int(\eta_R((1-t)x+ty)-\eta_R(x))(y-x)d\pi_kdt\\= \int_0^1\int_0^t \int \nabla \eta_R((1-s)x+sy)\cdot (y-x)^2 d\pi_kdsdt
 \end{multline*}
 and thus  
\begin{align*}
	\lefteqn{\lt|\int\eta_R\bar j_k-\int\eta_R(x)(y-x)d\pi_k\rt|}\\
	&\leq \frac12\sup|\nabla\eta_R|\int_0^1\int I(|(1-t)x+ty|<R)|y-x|^2d\pi_k dt\\
	&\stackrel{\eqref{ao15}\&(\ref{ma38})}{\les} \frac{1}{R^{d+1}}\int_{(B_{3R}\times\mathbb{R}^d)\cup(\mathbb{R}^d\times B_{3R})}|x-y|^2d\pi_k.
\end{align*}
 Combining this with $R\sim R_k$ (cf. \eqref{so01}),
we obtain (\ref{ao19}).
\medskip

We turn to (\ref{ao17}) and reformulate the left-hand side: Clearly, since the $x$-component
of the transference plan is not affected by the shift (\ref{so34}), we may replace
$\pi$ by $\pi_k$; the $y$-marginal of $\pi_k$ is the Lebesgue measure and
by the normalization in (\ref{ao15}) we have $1=\int\eta_R(y)d\pi_k$, so that as above
\begin{align*}
 \lt|\int\eta_R(x)d\pi-1\rt|&
 \leq \frac12 \sup|\nabla\eta_R|\int_0^1\int I(|(1-t)x+ty|<R)|y-x|d\pi_k dt\\
 &\les \frac1{R^{d+1}}\int_{(B_{3R}\times\mathbb{R}^d)\cup
(\mathbb{R}^d\times B_{3R})}|x-y|d\pi_k.
\end{align*}
 In view of $R\sim R_k$
it is sufficient  to show the estimate 
\begin{align}\label{ao20}
\frac{1}{R_k^d}\int_{(B_{3R_k}\times\mathbb{R}^d)\cup(\mathbb{R}^d\times B_{3R_k})}
|x-y|d\pi_k\lesssim\sqrt{E_k}.
\end{align}
which follows from the Cauchy-Schwarz inequality and \eqref{ma88}. 

\medskip

We finally turn to (\ref{ao18}) and start by applying (\ref{ao13}), yielding
\begin{align}\label{ao16}
\lt|\int\eta_R(\bar j_k-\nabla u_k)\rt|\lesssim\frac{E_k}{R_k}.
\end{align}
Noting that  $\frac{E_k}{R_k}$
$\ll\sqrt{E_k}$ thanks to \eqref{so01} and \eqref{ma153}, it is enough to establish \eqref{ao18} with $\nabla u_k$
replaced by $\bar j_k$. We observe that by definition (\ref{so08}) we have
\begin{align*}
\lt|\int\eta_R\bar j_k\rt|&=\lt|\int_0^1\int\eta_R((1-t)x+ty)(y-x)d\pi_kdt\rt|\\
&\stackrel{\eqref{ao15}}{\les} \int I(|(1-t)x+ty|<R)|x-y|d\pi_k,
\end{align*}
 so that the argument continuous as for \eqref{ao17}.
\end{proof}

 Now we may prove Theorem \ref{theo:main} which we restate and reformulate here for the reader's convenience.
 
\begin{theo1.2}
Let 
\[
 h_R:=\int \eta_R \nabla u_0.
\]
Then, for every $R\in[R_K,R_0]$,
\begin{equation}\label{po03}
  \lt|\int \eta_R(x)(y-x-h_R)d\pi\rt|+\lt|\int \eta_R(y-h_R)(y-x-h_R)d\pi\rt| \les \frac{\beta(R)}{R}
 \end{equation}
and 
\begin{multline}\label{je04}
 \sup \lt\{ |y-x-h_R| \ : \ (x,y)\in\spt \pi \cap [(B_{R}\times \R^d)\cup( \R^d\times B_{R}(h_R))]\rt\}\\
 \les R \lt(\frac{\beta(R)}{R^2}\rt)^{\frac{1}{d+2}}.\
\end{multline}
\end{theo1.2}

 \begin{proof}
 \newcounter{theomain} 
\refstepcounter{theomain}

{\sc Step} \arabic{theomain}. \refstepcounter{theomain} We  actually start with the  proof of \eqref{je04}.
We claim that
\begin{align}\label{so62}
\frac{1}{|B_{2R}|}\int_{(B_{2R}\times\mathbb{R}^d)\cup(\mathbb{R}^d\times B_{2R}(h_R))}|x-y+h_R|^2d\pi
\lesssim \beta(R).
\end{align}
Combining this with the $L^\infty$ bound \eqref{wg31} applied to $\tilde \pi=(\textrm{id},\textrm{id}-h_R)\# \pi$ would yield \eqref{je04}.
\medskip

Let $k\in[0,K-1]$ be such that $R\in [R_{k+1},R_k]$. Undoing the iterative shift (\ref{so34}), estimate (\ref{so01}) takes the form
\begin{align}\label{so61}
\frac{1}{|B_{6R_k}|}
\int_{(B_{6R_k}\times\mathbb{R}^d)\cup(\mathbb{R}^d\times B_{6R_k}(\tilde h_k))}
|x-y+\tilde h_k|^2d\pi\lesssim \beta(R_k)
\end{align}
with shift $\tilde h_k:=\sum_{\ell=0}^{k}\nabla\Phi_\ell(0)$. 
By \eqref{ao08}, \eqref{eq:mer01} (and \eqref{hypbetamain}), we have
\begin{align}\label{so64}
|\tilde h_k-h_R|\lesssim\frac{\beta( R_k)}{R_k} + \beta^{\frac{1}{2}}(R)\les \beta^{\frac{1}{2}}(R).
\end{align}
Because $\beta(R)\ll R^2$  (cf.\ \eqref{ma153}) this in particular implies that $B_{2R}(h_R)\subset B_{6R_k}(\tilde h_k)$. We may thus estimate
\begin{align*}
 \lefteqn{\frac{1}{|B_{2R}|}\int_{(B_{2R}\times\mathbb{R}^d)\cup(\mathbb{R}^d\times B_{2R}(h_R))}|x-y+h_R|^2d\pi}\\
 &\les \frac{1}{|B_{2R}|}\int_{(B_{2R}\times\mathbb{R}^d)\cup(\mathbb{R}^d\times B_{2R}(h_R))}(|x-y+\tilde{h}_k|^2 + |h_R-\tilde{h}_k|^2)d\pi\\
 &\stackrel{\eqref{so64}}{\les} \frac{1}{|B_{6R_k}|}\int_{(B_{6R_k}\times\mathbb{R}^d)\cup(\mathbb{R}^d\times B_{6R_k}(\tilde h_k))}|x-y+\tilde{h}_k|^2 d\pi \\
 &\qquad + \lt(\frac{1}{|B_{2R}|}\int_{(B_{2R}\times\mathbb{R}^d)\cup(\mathbb{R}^d\times B_{2R}(h_R))} d\pi\rt) \beta(R) \\
 &\stackrel{\eqref{so61}\&\eqref{ma88}}{\les}  \beta(R).
\end{align*}
This proves \eqref{so62}.

\medskip

{\sc Step} \arabic{theomain}. \refstepcounter{theomain}
We now turn to the proof of \eqref{po03}.
Let again $k\in [0,K-1]$ be such that $R\in [R_{k+1},R_k]$. Writing  
 \[
  \int \eta_R(x) (y-x) d\pi_k-\int \eta_R \nabla u_k= \int \eta_R(x) (y-x) d\pi_k-\int \eta_R \bar j_k + \int \eta_R (\nabla u_k-\bar j_k)
 \]
and evoking  (\ref{ao19}) of Lemma \ref{lem:aux} and (\ref{ao13}), we obtain 
\[\lt| \int \eta_R(x) (y-x) d\pi_k-\int \eta_R \nabla u_k\rt|\les \frac{E_k}{R_k}.\]
Appealing
to the iterative definition (\ref{so34}) of $\pi_k$ this is equivalent to 
\begin{align*}
\lt|\int\eta_R(x)\lt(y-x-\sum_{\ell=0}^{k-1}\nabla\Phi_\ell(0)\rt)d\pi
-\int\eta_R\nabla u_k\rt|
\lesssim\frac{E_k}{R_k}.
\end{align*}
We now use (\ref{ao17}) \& (\ref{ao18})  of Lemma \ref{lem:aux} in the form of 
\[\lt|\int \eta_R \nabla u_k - \int \eta_R(x) \lt(\int \eta_R \nabla u_k\rt) d\pi\rt|\les \frac{E_k}{R_k}\]
to reformulate this as
\begin{equation}\label{almost}
\lt|\int\eta_R(x)\lt(y-x\rt) d\pi - \lt( \int \eta_R(x) d\pi\rt)\lt(\sum_{\ell=0}^{k-1}\nabla\Phi_\ell(0)+\int\eta_R\nabla u_k\rt)
\rt|\lesssim\frac{E_k}{R_k}.
\end{equation}
By \eqref{ao21},  \eqref{so01} and the  monotonicity of $\frac{\beta(R)}{R}$  we get the first part of \eqref{po03},
\[
 \lt|\int \eta_R(x) (y-x-h_R) d\pi\rt|\les \frac{\beta(R)}{R}.
\]
By \eqref{ao15},
\begin{multline*}
 \lt|\int (\eta_R(x)-\eta_R(y-h_R))(y-x-h_R)d\pi\rt|\\
 \les \frac{1}{R^{d+1}}\int_{(B_R\times\R^d)\cup(\R^d\times B_R(h_R))} |x-y+h_R|^2d\pi,
\end{multline*}
and thus the second part of \eqref{po03} follows from \eqref{so62}. 
\end{proof}

 We finally derive Corollary \ref{cormain}, which we recall for the reader's convenience, from Theorem \ref{theo:main}.
 \begin{cor1.1}
 Let $\mu$ be  such that for some $0\le \alpha<1$,  and some  $\barR\gg1$, there holds
 \begin{equation}\label{hypdatafinal}
  \frac{1}{|B_R|} W^2_{B_R}\lt(\mu, \kappa\rt)\le C R^\alpha \qquad \forall  R\in[ 1, \barR].
 \end{equation}
Then, assuming for simplicity that $\mu(B_{\barR})=|B_{\barR}|$, if $\pi$ is the optimal coupling for $W_{B_{\barR}}(\mu,1)$,  for every   $\barR\ge R\ges 1$, 
\begin{equation}\label{mainweakcorfinal}
 \lt|\int \eta_R(x)(y-x-h_R)d\pi\rt| 
 \les \frac{1}{R^{1-\alpha}},
\end{equation}
where for $R>0$, $h_R=\int \eta_R \nabla u$ with  $u$ the up to an additive constant unique (distributional) solution of the Poisson equation 
\[
 \Delta u=\mu-1 \ \textrm{ in } B_{\barR} \qquad \textrm{ and } \qquad \nu\cdot \nabla u=0 \ \textrm{ on } \partial B_{\barR}.
\]

\end{cor1.1}
\begin{proof}
The statement follows directly from Theorem \ref{theo:main} provided we may  choose  $R_0=\barR$.
To prove this, notice that on the one hand $\spt \pi\subset B_{\barR}\times B_{\barR}$ implies
\[
 \frac{1}{|B_{6\barR}|}\int_{(B_{6\barR}\times \R^d)\cup (\R^d\times B_{6\barR})} |x-y|^2d\pi\le\frac{1}{|B_{\barR}|}\int_{(B_{\barR}\times \R^d)\cup (\R^d\times B_{\barR})} |x-y|^2d\pi,
\]
and thus $R_0=\barR$ and $\Phi_0=0$ satisfy the conclusions \eqref{so01} and \eqref{so39} of the Campanato iteration. On the other hand, $R_0$ is a good radius in the sense 
of Lemma \ref{lem:goodradii} (since having no transport across $\partial B_{R_0}$, the left-hand sides of \eqref{so76} and \eqref{so10} are both equal to zero). Therefore $R_0=\barR$ is indeed admissible.
\end{proof}

 \bibliographystyle{amsplain}

\bibliography{OT}

\providecommand{\bysame}{\leavevmode\hbox to3em{\hrulefill}\thinspace}
\providecommand{\MR}{\relax\ifhmode\unskip\space\fi MR }
\providecommand{\MRhref}[2]{%
  \href{http://www.ams.org/mathscinet-getitem?mr=#1}{#2}
}
\providecommand{\href}[2]{#2}
\begin{thebibliography}{10}

\bibitem{AKT84}
M.~{Ajtai}, J.~{Koml\'os}, and G.~{Tusn\'ady}, \emph{{On optimal matchings.}},
  {Combinatorica} \textbf{4} (1984), 259--264.

\bibitem{AmbFig}
L.~Ambrosio and A.~Figalli, \emph{On the regularity of the pressure field of
  {B}renier's weak solutions to incompressible {E}uler equations}, Calc. Var.
  Partial Differential Equations \textbf{31} (2008), no.~4, 497--509.

\bibitem{AFP}
L.~Ambrosio, N.~Fusco, and D.~Pallara, \emph{Functions of bounded variation and
  free discontinuity problems}, Oxford Mathematical Monographs, The Clarendon
  Press, Oxford University Press, New York, 2000.

\bibitem{AGS}
L.~Ambrosio, N.~Gigli, and G.~Savar\'e, \emph{Gradient flows in metric spaces
  and in the space of probability measures}, second ed., Lectures in
  Mathematics ETH Z\"urich, Birkh\"auser Verlag, Basel, 2008.

\bibitem{AGS19}
L.~Ambrosio, F.~Glaudo, and D.~Trevisan, \emph{On the optimal map in the
  2-dimensional random matching problem}, Discrete Contin. Dyn. Syst.
  \textbf{39} (2019), no.~12, 7291--7308.

\bibitem{AmStTr16}
L.~Ambrosio, F.~Stra, and D.~Trevisan, \emph{A {PDE} approach to a
  2-dimensional matching problem}, {Probab. Theory Relat. Fields} \textbf{173}
  (2019), no.~1-2, 433--477.

\bibitem{ArFeKu20}
S.~Armstrong, S.~J. Ferguson, and T.~Kuusi, \emph{Homogenization,
  linearization, and large-scale regularity for nonlinear elliptic equations},
  Communications on Pure and Applied Mathematics (2020).

\bibitem{AKMbook}
S.~Armstrong, T.~Kuusi, and J.-C. Mourrat, \emph{Quantitative stochastic
  homogenization and large-scale regularity}, Grundlehren der Mathematischen
  Wissenschaften, vol. 352, Springer, Cham, 2019.

\bibitem{MR3481355}
S.~Armstrong and C.~K. Smart, \emph{Quantitative stochastic homogenization of
  convex integral functionals}, Ann. Sci. \'Ec. Norm. Sup\'er. (4) \textbf{49}
  (2016), no.~2, 423--481.

\bibitem{MR0910954}
M.~Avellaneda and F.-H. Lin, \emph{Compactness methods in the theory of
  homogenization}, Comm. Pure Appl. Math. \textbf{40} (1987), no.~6, 803--847.

\bibitem{BaBo13}
F.~{Barthe} and C.~{Bordenave}, \emph{{Combinatorial optimization over two
  random point sets.}}, {S\'eminaire de probabilit\'es XLV}, Cham: Springer,
  2013, pp.~483--535.

\bibitem{Be18}
R.~J. Berman, \emph{Convergence rates for discretized monge-amp{\`e}re
  equations and quantitative stability of optimal transport}, arXiv:1803.00785
  (2018).

\bibitem{Breincompr}
Y.~Brenier, \emph{Minimal geodesics on groups of volume-preserving maps and
  generalized solutions of the {E}uler equations}, Comm. Pure Appl. Math.
  \textbf{52} (1999), no.~4, 411--452.

\bibitem{BreLoe}
Y.~Brenier and G.~Loeper, \emph{A geometric approximation to the {E}uler
  equations: the {V}lasov-{M}onge-{A}mp\`ere system}, Geom. Funct. Anal.
  \textbf{14} (2004), no.~6, 1182--1218.

\bibitem{CafCPAM91}
L.~A. Caffarelli, \emph{Some regularity properties of solutions of {M}onge
  {A}mp\`ere equation}, Comm. Pure Appl. Math. \textbf{44} (1991), no.~8-9,
  965--969.

\bibitem{CafJAMS92}
\bysame, \emph{The regularity of mappings with a convex potential}, J. Amer.
  Math. Soc. \textbf{5} (1992), no.~1, 99--104.

\bibitem{CaLuPaSi14}
S.~Caracciolo, C.~Lucibello, G.~Parisi, and G.~Sicuro, \emph{Scaling hypothesis
  for the euclidean bipartite matching problem}, Physical Review E \textbf{90}
  (2014), no.~1.

\bibitem{ChenFriedreview}
G.-Q. {Chen} and H.~{Frid}, \emph{{On the theory of divergence-measure fields
  and its applications.}}, {Bol. Soc. Bras. Mat., Nova S\'er.} \textbf{32}
  (2001), no.~3, 401--433.

\bibitem{ChenFried}
G.-Q. Chen and H.~Frid, \emph{Extended divergence-measure fields and the
  {E}uler equations for gas dynamics}, Comm. Math. Phys. \textbf{236} (2003),
  no.~2, 251--280.

\bibitem{DePFig}
G.~De~Philippis and A.~Figalli, \emph{Partial regularity for optimal transport
  maps}, Publ. Math. Inst. Hautes \'Etudes Sci. \textbf{121} (2015), 81--112.

\bibitem{Dupaigne}
L.~Dupaigne, \emph{Stable solutions of elliptic partial differential
  equations}, Chapman \& Hall/CRC Monographs and Surveys in Pure and Applied
  Mathematics, vol. 143, Chapman \& Hall/CRC, Boca Raton, FL, 2011.

\bibitem{FigKim}
A.~Figalli and Y.-H. Kim, \emph{Partial regularity of {B}renier solutions of
  the {M}onge-{A}mp\`ere equation}, Discrete Contin. Dyn. Syst. \textbf{28}
  (2010), no.~2, 559--565.

\bibitem{FiNe19}
J.~Fischer and S.~Neukamm, \emph{Optimal homogenization rates in stochastic
  homogenization of nonlinear uniformly elliptic equations and systems},
  arXiv:1908.02273 (2019).

\bibitem{arXiv:1409.2678}
A.~{Gloria}, S.~{Neukamm}, and F.~{Otto}, \emph{{A regularity theory for random
  elliptic operators}}, ArXiv e-prints (2014).

\bibitem{GHO}
M.~{Goldman}, M.~{Huesmann}, and F.~{Otto}, \emph{{A large-scale regularity
  theory for the Monge-Ampere equation with rough data and application to the
  optimal matching problem}}, arXiv:1808.09250 (2018).

\bibitem{GO}
M.~Goldman and F.~Otto, \emph{A variational proof of partial regularity for
  optimal transportation maps}, arXiv:1704.05339 (2017).

\bibitem{HuSt13}
M.~{Huesmann} and K.-T. {Sturm}, \emph{{Optimal transport from Lebesgue to
  Poisson}}, {Ann. Probab.} \textbf{41} (2013), no.~4, 2426--2478.

\bibitem{Le17}
M.~Ledoux, \emph{On optimal matching of {G}aussian samples}, Zap. Nauchn. Sem.
  S.-Peterburg. Otdel. Mat. Inst. Steklov. (POMI) \textbf{457} (2017),
  no.~Veroyatnost' \ i Statistika. 25, 226--264.

\bibitem{Le18}
\bysame, \emph{On optimal matching of {G}aussian samples {II}}, preprint
  (2018).

\bibitem{MeDeCh19}
Q.~M{\'e}rigot, A.~Delalande, and F.~Chazal, \emph{Quantitative stability of
  optimal transport maps and linearization of the 2-wasserstein space},
  arXiv:1910.05954 (2019).

\bibitem{Santam}
F.~Santambrogio, \emph{Optimal transport for applied mathematicians}, Progress
  in Nonlinear Differential Equations and their Applications, vol.~87,
  Birkh\"auser/Springer, Cham, 2015, Calculus of variations, PDEs, and
  modeling.

\bibitem{Santamduality}
\bysame, \emph{Regularity via duality in calculus of variations and degenerate
  elliptic {PDE}s}, J. Math. Anal. Appl. \textbf{457} (2018), no.~2,
  1649--1674.

\bibitem{Ta14}
M.~Talagrand, \emph{Upper and lower bounds for stochastic processes: modern
  methods and classical problems}, vol.~60, Springer Science \& Business Media,
  2014.

\bibitem{Troianello}
G.~M. Troianiello, \emph{Elliptic differential equations and obstacle
  problems}, The University Series in Mathematics, Plenum Press, New York,
  1987.

\bibitem{Viltop}
C.~Villani, \emph{Topics in optimal transportation}, Graduate Studies in
  Mathematics, vol.~58, American Mathematical Society, Providence, RI, 2003.

\bibitem{VilOandN}
\bysame, \emph{Optimal transport: old and new}, Grundlehren der Mathematischen
  Wissenschaften, vol. 338, Springer-Verlag, Berlin, 2009.

\end{thebibliography}
\end{document}